\documentclass{article}
\usepackage{amsmath,amssymb,amsthm,amsxtra,bm,graphicx}
\usepackage[all]{xy}
\usepackage[hypertex]{hyperref}
\newcommand{\fg}{f.g.}

\newtheorem{thm}{Theorem}[subsection]
\newtheorem{prop}[thm]{Proposition}
\newtheorem{lem}[thm]{Lemma}
\newtheorem{cor}[thm]{Corollary}
\newtheorem*{lem2.0.3}{Lemma~2.0.3}
\newtheorem*{lem2.0.4}{Lemma~2.0.4}
\newtheorem*{lem2.0.5}{Lemma~2.0.5}
\newtheorem*{lem2.0.6}{Lemma~2.0.6}
\theoremstyle{definition}
\newtheorem{dfn}[thm]{Definition}
\newtheorem{rem}[thm]{Remark}
\newtheorem{notation}[thm]{Notation}
\newtheorem{assumption}[thm]{Assumption}
\newtheorem{example}[thm]{Example}
\newtheorem{construction}[thm]{Construction}
\newtheorem*{dfn2.0.1}{Definition~2.0.1}
\newtheorem*{dfn2.0.2}{Definition~2.0.2}
\newtheorem*{notation3.0.1}{Notation~3.0.1}
\newcommand{\N}{\mathbb{N}}
\newcommand{\Q}{\mathbb{Q}}
\newcommand{\Z}{\mathbb{Z}}
\newcommand{\C}{\mathbb{C}}
\newcommand{\R}{\mathbb{R}}
\newcommand{\A}{\mathbb{A}}
\newcommand{\B}{\mathbb{B}}
\newcommand{\D}{\mathbb{D}}
\newcommand{\E}{\mathbb{E}}
\newcommand{\V}{\mathbb{V}}
\newcommand{\zp}{\mathbb{Z}_p}
\newcommand{\oo}{\mathcal{O}}
\newcommand{\gk}{G_K}
\newcommand{\nrig}{\wtil{\mathbb{N}}_{\mathrm{rig}}^{\nabla+}}
\newcommand{\ndr}{\mathbb{N}_{\mathrm{dR}}}
\newcommand{\brig}{\widetilde{\mathbb{B}}_{\mathrm{rig}}^{\nabla+}}
\newcommand{\brigdagk}{\mathbb{B}_{\mathrm{rig},K}^{\dagger}}
\newcommand{\brigdag}{\mathbb{B}_{\mathrm{rig}}^{\dagger}}

\newcommand{\btrigdag}{\widetilde{\mathbb{B}}_{\mathrm{rig}}^{\dagger}}
\newcommand{\btdag}{\widetilde{\mathbb{B}}^{\dagger}}
\newcommand{\bdagk}{\mathbb{B}^{\dagger}_K}
\newcommand{\bdag}{\mathbb{B}^{\dagger}}

\newcommand{\und}[1]{\underline{#1}}
\newcommand{\wtil}[1]{\widetilde{#1}}
\newcommand{\rig}{\mathrm{rig}}
\newcommand{\dR}{\mathrm{dR}}
\newcommand{\geom}{\mathrm{geom}}
\newcommand{\arith}{\mathrm{arith}}
\newcommand{\alg}{\mathrm{alg}}
\newcommand{\sep}{\mathrm{sep}}
\newcommand{\ur}{\mathrm{ur}}
\newcommand{\art}{\mathrm{Art}}
\newcommand{\sw}{\mathrm{Swan}}
\newcommand{\LT}{\mathrm{LT}}

\newcommand{\Hom}{\mathrm{Hom}}
\newcommand{\Der}{\mathrm{Der}}
\newcommand{\cont}{\mathrm{cont}}
\newcommand{\inc}{\mathrm{inc}}
\newcommand{\AS}{\mathrm{AS}}
\newcommand{\cris}{\mathrm{cris}}
\newcommand{\con}{\mathrm{con}}
\newcommand{\can}{\mathrm{can}}
\newcommand{\an}{\mathrm{an}}
\newcommand{\Mod}{\mathrm{Mod}}
\newcommand{\rep}{\mathrm{Rep}}
\newcommand{\et}{\mathrm{et}}
\newcommand{\st}{\mathrm{st}}
\newcommand{\pst}{\mathrm{pst}}
\newcommand{\fil}{\mathrm{Fil}}
\newcommand{\gr}{\mathrm{gr}}
\newcommand{\id}{\mathrm{id}}
\newcommand{\pr}{\mathrm{pr}}
\newcommand{\spe}{\mathrm{sp}}
\newcommand{\qt}{\mathrm{qt}}
\newcommand{\Zar}{\mathrm{Zar}}
\newcommand{\lex}{\mathrm{lex}}
\newcommand{\Spec}{\mathrm{Spec}}
\newcommand{\Idem}{\mathrm{Idem}}
\begin{document}

\title{On differential modules associated to de Rham representations\\ in the imperfect residue field case}
\author{Shun Ohkubo
\footnote{
Graduate School of Mathematics, Nagoya University, Furocho, Chikusaku, Nagoya 4648602, Japan, E-mail address: shun.ohkubo@gmail.com}
}
\date{\today}

\maketitle

\begin{abstract}
Let $K$ be a complete discrete valuation field of mixed characteristic $(0,p)$, whose residue fields may not be perfect, and $G_K$ the absolute Galois group of $K$. In the first part of this paper, we prove that Scholl's generalization of fields of norms over $K$ is compatible with Abbes-Saito's ramification theory. In the second part, we construct a functor $\N_{\dR}$ associating a de Rham representation $V$ with a $(\varphi,\nabla)$-module in the sense of Kedlaya. Finally, we prove a compatibility between Kedlaya's differential Swan conductor of $\N_{\dR}(V)$ and Swan conductor of $V$, which generalizes Marmora's formula.

\end{abstract}

\tableofcontents


\section*{Introduction}

Hodge theory relates the singular cohomology of complex projective manifolds $X$ to the spaces of harmonic forms on $X$. Its $p$-adic analogue, $p$-adic Hodge theory enables us to compare the $p$-adic \'etale cohomology $H_{\text{\'et}}^m(X_{\overline{\Q}_p},\mathbb{Q}_p)$ of proper smooth varieties $X$ over the $p$-adic field $\mathbb{Q}_p$ with de Rham cohomology of $X$. Precisely speaking, the natural action of the absolute Galois group $G_{\mathbb{Q}_p}$ of $\mathbb{Q}_p$ on the $p$-adic \'etale cohomology can be recovered after tensoring $\B_{\dR}$, which is the ring of $p$-adic periods introduced by Jean-Marc Fontaine. If $X$ has a semi-stable reduction, then one can obtain more precise comparison theorem of the $p$-adic \'etale cohomology of $X$ to the log-cristalline cohomology of the special fiber of $X$. Thus, we have a satisfactory theory on $p$-adic \'etale cohomology of proper smooth varieties over $\mathbb{Q}_p$.

A $p$-adic representation $V$ of $G_{\mathbb{Q}_p}$ is a finite dimensional $\mathbb{Q}_p$-vector space with a continuous linear $G_{\mathbb{Q}_p}$-action. Fontaine (\cite{Fon}) defined the notions of de Rham, crystalline, and semi-stable representations, which form important subcategories of the category of $p$-adic representations of $G_{\mathbb{Q}_p}$. Then, he associated linear algebraic objects such as filtered vector spaces with extra structures to objects in each category. Fontaine's classification is compatible with geometry in the following sense: For a proper smooth variety $X$ over $\mathbb{Q}_p$, the $p$-adic representation $H_{\text{\'et}}^m(X_{\overline{\mathbb{Q}}_p},\mathbb{Q}_p)$ of $G_{\mathbb{Q}_p}$ is only de Rham in general. However, if $X$ has a semi-stable reduction (resp. good reduction), then $H_{\text{\'et}}^m(X_{\overline{\mathbb{Q}}_p},\mathbb{Q}_p)$ is semi-stable (resp. crystalline). 


There also exists a more analytic description of general $p$-adic representations. Let $\B_{\Q_p}$ be the fraction field of the $p$-adic completion of $\Z_p[[t]][t^{-1}]$. We define the action of $\Gamma_{\Q_p}:=G_{\mathbb{Q}_p(\mu_{p^{\infty}})/\Q_p}$ on $\B_{\Q_p}$ by $\gamma(t)=(1+t)^{\chi(g)}-1$, where $\chi:\Gamma_{\Q_p}\to\Z_p^{\times}$ is the cyclotomic character. We also define a Frobenius lift $\varphi$ on $\B_{\Q_p}$ by $\varphi(t)=(1+t)^p-1$. An \'etale $(\varphi,\Gamma_{\Q_p})$-module over $\B_{\Q_p}$ is a finite dimensional $\B_{\Q_p}$-vector space $M$ endowed with compatible actions of $\varphi$ and $\Gamma_{\Q_p}$such that the Frobenius slopes of $M$ are all zero. By using Fontaine-Wintenberger's isomorphism of Galois groups
\[
G_{\mathbb{Q}_p(\mu_{p^{\infty}})}\cong G_{\mathbb{F}_p((t))},
\]
Fontaine (\cite{Fes}) proved an equivalence between the category of $p$-adic representations and the category of \'etale $(\varphi,\Gamma_{\mathbb{Q}_p})$-modules over $\B_{\Q_p}$. We consider the overconvergent subring
\[
\B_{\Q_p}^{\dagger}:=\{f=\sum_{n\in\Z}a_nt^n\in\B_{\Q_p};a_n\in\Q_p,|a_n|\rho^n\to 0\ (n\to-\infty)\  \text{for some }\rho\in (0,1]\}
\]
of $\B_{\Q_p}$. Fr\'ed\'eric Cherbonnier and Pierre Colmez (\cite{CC1}) proved that the category of \'etale $(\varphi,\Gamma_{\mathbb{Q}_p})$-modules over $\B_{\Q_p}$ is equivalent to the category of \'etale $(\varphi,\Gamma_{\mathbb{Q}_p})$-modules over $\B^{\dagger}_{\Q_p}$. As a consequence of Cherbonnier-Colmez' theorem, $p$-adic analysis over the Robba ring
\[
\mathcal{R}_{\Q_p}:=\cup_{\rho'\in (0,1)}\{f=\sum_{n\in\Z}a_nt^n;a_n\in\Q_p,|a_n|\rho^n\to 0\ (n\to\pm\infty)\text{ for all }\rho\in (\rho',1]\}
\]
comes into play. Actually, Laurent Berger (\cite{Inv}) associated a $p$-adic differential equation $\N_{\dR}(V)$ over $\mathcal{R}_{\Q_p}$ to a de Rham representation $V$ via the above equivalences. By using this functor $\N_{\dR}$ and the quasi-unipotence of $p$-adic differential equations due to Yves Andr\'e, Zoghman Mebkhout, and Kiran Kedlaya, Berger proved Fontaine's $p$-adic local monodromy conjecture, which is a $p$-adic analogue of Grothendieck's $l$-adic monodromy theorem. We note that in the above theory, $G_{\mathbb{Q}_p}$ is usually replaced by $G_K$, where $K$ is a complete valuation field of mixed characteristic $(0,p)$ with a perfect residue field.

Recently, based on earlier works of Gerd Faltings and Osamu Hyodo, Fabrizio Andreatta and Olivier Brinon (\cite{AB}) started to generalize Fontaine's theory in the relative situation: Instead of complete discrete valuation rings with perfect residue fields, they work over higher dimensional ground rings $R$ such as the generic fiber of Tate algebra $\Z_p\{T_1,T_1^{-1},\dots,T_d,T_d^{-1}\}$. In this paper, we work in the most basic case of Andreatta-Brinon's setup. That is, our ground ring $K$ is still a complete valuation field, however, with a non-perfect residue field $k_K$ such that $p^d=[k_K:k_K^p]<\infty$. Such a complete discrete valuation field arises as the completion of ground rings along the special fiber in Andreatta-Brinon's setup. 

Even in our situation, a generalization of Fontaine's theory could be useful as in the proof of Kato's divisibilty result in Iwasawa Main conjecture for $GL_2$ (\cite{Kat}). Using compatible systems of $K_2$ of affine modular curves $Y(p^nN)$ varying $n$, Kato defines ($p$-adic) Euler systems in Galois cohomology groups over $\mathbb{Q}_p$ whose coefficients are related to cusp forms. A key ingredient in \cite{Kat} is that Kato's Euler systems are related with some products of Eisenstein series via Bloch-Kato dual exponential map $\exp^*$. In the proof of this fact, $p$-adic Hodge theory for ``the field of $q$-expansions'' $\mathcal{K}$ plays an important role, where $\mathcal{K}$ is the fraction field of the $p$-adic completion of $\mathbb{Z}_p[\zeta_{p^N}][[q^{1/N}]][q^{-1}]$. Roughly speaking, Tate's universal elliptic curve together with torsion points induces a morphism $\mathrm{Spec}(\mathcal{K}(\zeta_{p^n},q^{p^{-n}}))\to Y(p^nN)$. Using a generalization of Fontaine's ring $\mathbb{B}_{\mathrm{dR}}$ over $\mathcal{K}$, Kato defines a dual exponential map for Galois cohomology groups over $\mathcal{K}(\zeta_{p^n},q^{p^{-n}})$, and proves its compatibility with $\exp^*$. Then, the image of Kato's Euler system under $\exp^*$ is calculated by using Kato's generalized explicit reciprocity law for $p$-divisible groups over $\mathcal{K}(\zeta_{p^n},q^{p^{-n}})$.

To explain our results, we recall Anthony Scholl's theory of field of norms (\cite{Sch}), which is a generalization of Fontaine-Wintenberger's theorem when $k_K$ is non-perfect. In the rest of the introduction, for simplicity, we restrict ourselves to ``Kummer tower case'': That is, we choose a lift $\{t_j\}_{1\le j\le d}$ of a $p$-basis of $k_K$ and define a tower $\mathfrak{K}:=\{K_n\}$ of fields by $K_n:=K(\mu_{p^n},t_1^{p^{-n}},\dots,t_d^{p^{-n}})$ for $n>0$, and set $K_{\infty}:=\cup_n{K_n}$. Then, the Frobenius on $\mathcal{O}_{K_{n+1}}/p\mathcal{O}_{K_{n+1}}$ factors through $\mathcal{O}_{K_n}/p\mathcal{O}_{K_n}\hookrightarrow\mathcal{O}_{K_{n+1}}/p\mathcal{O}_{K_{n+1}}$, and the limit $X^+_{\mathfrak{K}}:=\varprojlim_n\mathcal{O}_{K_n}/p\mathcal{O}_{K_n}$ is a complete valuation ring of characteristic $p$. Here, we denote the integer ring of a valuation field $F$ by $\mathcal{O}_F$. Let $X_{\mathfrak{K}}$ be the fraction field of $X^+_{\mathfrak{K}}$. Then, Scholl proved that a similar limit procedure gives an equivalence of categories $\mathbf{F\acute{E}t}_{K_{\infty}}\cong \mathbf{F\acute{E}t}_{X_{\mathfrak{K}}}$, where $\mathbf{F\acute{E}t}_A$ denotes the category of finite \'etale algebras over $A$. In particular, we obtain an isomorphism of Galois groups
\[
\tau:G_{K_{\infty}}\cong G_{X_{\mathfrak{K}}}.
\]

The Galois group of a complete valuation field $F$ is canonically endowed with non-log and log ramification filtrations in the sense of Abbes-Saito (\cite{AS}). By using the ramification filtrations, one can define Artin and Swan conductors of Galois representations, which are important arithmetic invariants. It is natural to ask that Scholl's isomorphism $\tau$ is compatible with ramification theory. The first goal of this paper is to answer this question in the following form:
\begin{thm}[Theorem~\ref{thm:normcompat}]\label{thm:intro1}
Let $V$ be a $p$-adic representation of $\gk$, where the $\gk$-action of $V$ factors through a finite quotient. Then, Artin and Swan conductors of $V|_{K_n}$ are stationary and their limits coincide with Artin and Swan conductors of $\tau^*(V|_{K_{\infty}})$. 
\end{thm}

We briefly mention the idea of the proof in the Artin case. Note that in the prefect residue field case, it follows from the fact that the upper numbering ramification group is a renumbering one, which is compatible with the field of norms construction (see \cite[Lemme~5.4]{Mar}). However, in the imperfect residue field case, since Abbes-Saito's ramification filtration is not a renumbering of the lower numbering one, we proceed as follows. Let $L/K$ be a finite Galois extension. Let $X_{\mathfrak{L}}$ be an extension of $X_{\mathfrak{K}}$ corresponding to the tower $\mathfrak{L}=\{L_n:=LK_n\}$ under Scholl's equivalence. Then, we may reduce to prove that the non-log ramification filtrations of $G_{L_n/K_n}$ and $G_{X_{\mathfrak{L}}/X_{\mathfrak{K}}}$ coincide with each other. Abbes-Saito's non-log ramification filtration of a finite extension $E/F$ of complete discrete valuation fields is described by a certain family of rigid analytic spaces $as^a_{E/F}$ for $a\in\Q_{\ge 0}$ attached to $E/F$. In terms of Abbes-Saito's setup, we have only to prove that the number of connected components of $as^a_{X_{\mathfrak{L}}/X_{\mathfrak{K}}}$ and $as^a_{L_n/K_n}$ for sufficiently large $n$ are the same. An optimized proof of this assertion is as follows: We construct a characteristic $0$ lift $R$ of $X^+_{\mathfrak{K}}$, which is realized as the ring of functions on the open unit ball over a complete valuation ring. We can find a prime ideal $\mathfrak{p}_n$ of $R$ such that $R/\mathfrak{p}_n$ is isomorphic to $\mathcal{O}_{K_n}$. Then, we construct a lift $AS^a_{X_{\mathfrak{L}}/X_{\mathfrak{K}}}$ over $\mathrm{Spec}(R)$ of $as^a_{X_{\mathfrak{L}}/X_{\mathfrak{K}}}$, whose generic fiber at $\mathfrak{p}_n$ is isomorphic to $as^a_{L_n/K_n}$. We may also regard $AS^a_{X_{\mathfrak{L}}/X_{\mathfrak{K}}}$ as a family of rigid spaces parametrized by $\mathrm{Spec}(R)$. What we actually prove is that in such a family of rigid spaces over $\mathrm{Spec}(R)$, the number of the connected components of the fiber varies ``continuously''. This is done by Gr\"obner basis arguments over complete regular local rings, which extends the method of Liang Xiao (\cite{Xia}). The continuity result implies our assertion since the point $\mathfrak{p}_n\in \mathrm{Spec}(R)$ ``converges'' to the point $(p)\in \mathrm{Spec}(R)$.

Note that Shin Hattori (\cite{Hat}) reproved the above ramification compatibility of Scholl's isomorphism $\tau$ by using Peter Scholze's perfectoid spaces (\cite{Scholze}), which are a geometric interpretation of Fontaine-Wintenberger theorem. We briefly explain Hattori's proof. Let $\C_p$ (resp. $\C_p^{\flat}$) be the completion of the algebraic closure of $K_{\infty}$ (resp. $X_{\mathfrak{K}}$). Scholze proved the tilting equivalence between certain adic spaces (perfectoid spaces) over $\C_p$ and $\C_p^{\flat}$. Let $C$ be a perfectoid field and $Y$ a subvariety of $\A_{C}^n$. A perfection of $Y$ is a perfectoid space defined as the pull-back of $Y$ under the canonical projection $\varprojlim_{T_i\mapsto T_i^p}\A_C^n\to\A_C^n$, where $T_1,\dots,T_n$ denotes a coordinate of $\A_{C}^n$. Hattori proved that the tilting of the perfections of $(as^a_{L_n/K_n})_{\C_p}$ and $(as^{a}_{X_{\mathfrak{L}}/X_{\mathfrak{K}}})_{\C_p^{\flat}}$ are isomorphic under the tilting equivalence. Since the underlying topological spaces are homeomorphic under taking perfections and the tilting equivalence, he obtained the ramification compatibility.

The second goal of this paper is to generalize Berger's functor $\N_{\dR}$ and prove a ramification compatibility of $\N_{\dR}$, which extends Theorem~\ref{thm:intro1}. Precisely, we construct a functor from the category of de Rham representations to the category of $(\varphi,\nabla)$-modules over the Robba ring. Our target object $(\varphi,\nabla)$-modules are defined by Kedlaya as a generalization of $p$-adic differential differential equations (\cite{Sw}). Kedlaya also defined the differential Swan conductor $\mathrm{Swan}^{\nabla}(M)$ for a $(\varphi,\nabla)$-module $M$, which is a generalization of the irregularity of $p$-adic differential equations. Then, we prove the following de Rham version of Theorem~\ref{thm:intro1}:
\begin{thm}[Theorem~\ref{thm:main}]\label{thm:intro2}
Let $V$ be a de Rham representation of $\gk$. Then, we have
\[
\mathrm{Swan}^{\nabla}(\N_{\dR}(V))=\lim_{n\to\infty}\mathrm{Swan}(V|_{K_n}),
\]
where $\mathrm{Swan}$ on the RHS means Abbes-Saito's Swan conductor. Moreover, the sequence $\{\mathrm{Swan}(V|_{K_n})\}_{n>0}$ is eventually stationary.
\end{thm}

Both Theorem~\ref{thm:intro1}, \ref{thm:intro2} are due to Adriano Marmora (\cite{Mar}) when the residue field is perfect. Even when the residue field is perfect, our proof of Theorem~\ref{thm:intro2} is slightly different from Marmora's proof since we use d\'evissage argument to reduce the pure slope case. As is addressed in \cite[\S~3.7]{Sw}, it seems to be possible to define a ramification invariant of $\N_{\dR}(V)$ in terms of $(\varphi,\Gamma_K)$-modules so that one can compute $\mathrm{Swan}(V)$ instead of $\mathrm{Swan}(V|_{K_n})$. It is also important to extend the construction of $\N_{\dR}$ to the general relative case: One may expect that a relative version of slope theory, described in \cite{slope} for example, will be an important tool.


\section*{Structure of the paper}
In \S~\ref{sec:pre}, we gather various basic results used in this paper. These contain some $p$-adic Hodge theory, Abbes-Saito's ramification theory and Kedlaya's theory of overconvergent rings, Scholl's fields of norms.

In \S~\ref{sec:adeq}, we prove some ring theoretic property of overconvergent rings by using Kedlaya's slope theory.

In \S~\ref{sec:grobner}, we develop Gr\"obner basis argument over complete regular local rings and overconvergent rings. We apply Gr\"obner basis argument to study families of rigid spaces, and use it to prove Theorem~\ref{thm:intro1}.

In \S~\ref{sec:diff}, we generalize Berger's gluing argument to construct a differential module $\N_{\dR}(V)$ for de Rham representations $V$. We also study the graded pieces of $\N_{\dR}(V)$ with respect to Kedlaya's slope filtration to reduce Theorem~\ref{thm:intro2} to Theorem~\ref{thm:intro1} by d\'evissage.

\section*{Convention}
Throughout this paper, let $p$ be a prime number. All rings are assumed to be commutative unless otherwise is mentioned. For a ring $R$, denote by $\pi_0^{\Zar}(R)$ the set of connected component of $\Spec(R)$ with respect Zariski topology. For a field $E$, fix an algebraic closure (resp. a separable closure) of $E$, denote it by $E^{\alg}$ or $\overline{E}$ (resp. $E^{\sep}$) and let $G_E$ be the absolute Galois group of $E$. For a finite Galois extension $F/E$, denote the Galois group of $F/E$ by $G_{F/E}$. For a field $k$ of characteristic $p$, let $k^{\mathrm{pf}}:=k^{p^{-\infty}}$ be the perfect closure in a fixed algebraic closure of $k$.

For a complete valuation field $K$, we denote the integer ring by $\oo_K$ and a uniformizer by $\pi_K$, the residue field by $k_K$. Let $v_K:K\to\Z\cup\{\infty\}$ be the discrete valuation satisfying $v_K(\pi_K)=1$. We denote by $K^{\ur}$ the $p$-adic completion of the maximal unramified extension of $K$ and denote by $I_K$ the inertia subgroup of $\gk$. Moreover, assume that $K$ is of mixed characteristic $(0,p)$ and $[k_K:k_K^p]=p^d<\infty$ in the rest of this paragraph. Denote the absolute ramification index by $e_K$. Denote by $\C_p$ the $p$-adic completion of $K^{\alg}$ and let $v_p$ be the $p$-adic valuation of $\C_p$ normalized by $v_p(p)=1$. We fix a system of $p$-power roots of unity $\{\zeta_{p^n}\}_{n\in\N_{>0}}$ in $K^{\alg}$, i.e., $\zeta_p$ is a primitive $p$-th root of unity and $\zeta_{p^{n+1}}^p=\zeta_{p^n}$ for all $n\in\N_{>0}$. Let $\chi:\gk\to\zp^{\times}$ be the cyclotomic character defined by $g(\zeta_{p^n})=\zeta_{p^n}^{\chi(g)}$ for all $n\in\N_{>0}$. We denote the fraction field of a Cohen ring of $k_K$ by $K_0$. Denote a lift of a $p$-basis of $k_K$ in $\oo_K$ by $\{t_j\}_{1\le j\le d}$. For a given $\{t_j\}_{1\le j\le d}$, we can choose an embedding $K_0\hookrightarrow K$ such that $\{t_j\}_{1\le j\le d}\subset\oo_{K_0}$ (see \cite[\S~1.1]{Ohk}). Unless otherwise is mentioned, we always choose $\{t_j\}_{1\le j\le d}$ and an embedding $K_0\hookrightarrow K$ as this way. Unless a particular mention is stated, we also fix a sequences of $p$-power roots $\{t_j^{p^{-n}}\}_{n\in\N,1\le j\le d}$ of $\{t_j\}_{1\le j\le d}$ in $K^{\alg}$, i.e., we have $(t_j^{p^{-n-1}})^p=t_j^{p^{-n}}$ for all $n\in\N_{>0}$. For such a sequence, we define $K^{\mathrm{pf}}$ as the $p$-adic completion of $\cup_nK(\{t_j^{p^{-n}}\}_{1\le j\le d})$, which is a complete discrete valuation field with perfect residue field $k_K^{\mathrm{pf}}$, and we regard $\C_p$ as the $p$-adic completion of the algebraic closure of $K^{\mathrm{pf}}$.

For $h\in\N_{>0}$, denote $\Q_{p^h}:=W(\mathbb{F}_{p^h})[p^{-1}]$. Let $K$ be a complete discrete valuation field, $F/\Q_p$ a finite extension. A finite dimensional $F$-vector space $V$ with continuous semi-linear $G_K$-action is called an $F$-representation of $\gk$ and moreover, if $F=\Q_p$, then we call $V$ a $p$-adic representation of $\gk$. We denote the category of $F$-representations of $\gk$ by $\rep_F{(\gk)}$. We say that $V$ is finite (resp. of finite geometric monodromy) if $G_K$ (resp. $I_K$) acts on $V$ via a finite quotient. We denote the category of finite (resp. finite geometric monodromy) $F$-representations of $G_K$ by $\rep^f_F{(\gk)}$ (resp. $\rep^{\fg}_F{(\gk)}$).

For a ring $R$, denote the Witt ring with coefficients in $R$ by $W(R)$. If $R$ is of characteristic $p$, then we denote the absolute Frobenius on $R$ by $\varphi$ and also denote the ring homomorphism $W(\varphi):W(R)\to W(R)$ by $\varphi$. Denote by $[x]\in W(R)$ the Teichm\"uller lift of $x\in R$.

For group homomorphisms $f,g:M\to N$ of abelian groups, we denote by $M^{f=g}$ the kernel of the map $f-g:M\to N$.

For $x\in\mathbb{R}$, let $\lfloor x\rfloor:=\inf\{n\in\Z;n\ge x\}$ be the least integer greater than or equal to $x$.


\section{Preliminaries}\label{sec:pre}
In this section, we will recall basic results used in the following of the paper and fix notation.

\subsection{Fr\'echet spaces}\label{subsec:Frechet}
We will define some basic terminology of topological vector spaces. Although we will use both valuations and norms to consider topologies, we will define our terminology in terms of valuations for simplicity. See \cite{pde} or \cite{Schn} for details.

\begin{notation}
Let $M$ be an abelian group. A valuation $v$ of $M$ is a map $v:M\to\R\cup\{\infty\}$ such that $v(x-y)\ge \inf\{v(x),v(y)\}$ for all $x,y\in R$ and $v(x)=\infty$ if and only if $x=0$. Moreover, when $M=R$ is a ring, $v$ is multiplicative if $v(xy)=v(x)+v(y)$ for all $x,y\in R$. A ring with multiplicative valuation is called a valuation ring. If $(R,v)$ is a valuation ring and $(M,v_M)$ is an $R$-module with valuation, then we say $v_M$ is an $R$-valuation if $v_M(\lambda x)=v(\lambda)+v_M(x)$ for $\lambda\in R$ and $x\in M$.

Let $(R,v)$ be a valuation ring and $M$ a finite free $R$-module. For an $R$-basis $e_1,\dots,e_n$ of $M$, we define the $R$-valuation $v_M$ on $M$ (compatible with $v$) associated to $e_1,\dots,e_n$ by $v_M(\sum_{1\le i\le n}a_ie_i)=\inf_i{v(a_i)}$ for $a_i\in R$ (\cite[Definition~1.3.2]{pde}). The topology defined by $v_M$ is independent of the choice of a basis of $M$ (\cite[Definition~1.3.3]{pde}). Hence, we do not refer to a basis to consider $v_M$ and we just denote $v_M$ by $v$ unless otherwise is mentioned.
\end{notation}


For any valuation $v$ on $M$, we define the associated non-archimedean norm $|\cdot|:M\to\R$ by $|x|:=a^{-v(x)}$ for a fixed $a\in \R_{>1}$ (non-archimedean means that it satisfies the strong triangle inequality). Conversely, for any non-archimedean norm $|\cdot|$, $v(\cdot)=-\log_a|\cdot|$ is a valuation. We will apply various definitions made for norms to valuations, and vice versa in this manner.

\begin{notation}
Let $(K,v)$ be a complete valuation field. Let $\{w_r\}_{r\in I}$ be a family of $K$-valuations of a $K$-vector space $V$. Consider the topology $\mathcal{T}$ of $V$, whose neighborhoods at $0$ are generated by $\{x\in V;w_r(x)\ge n\}$ for all $r\in I$ and $n\in\N$. We call $\mathcal{T}$ the topology of $V$ defined by $\{w_r\}_{r\in I}$ and denote $V$ equipped with the topology $\mathcal{T}$ by $(V,\{w_r\}_{r\in I})$, or simply by $V$. If $\mathcal{T}$ is equivalent to the topology defined by $\{w_r\}_{r\in I_0}$ for some countable subset $I_0\subset I$, we call $\mathcal{T}$ the $K$-Fr\'echet topology defined by $\{w_r\}_{r\in I}$. For a $K$-vector space, it is well-known that a $K$-Fr\'echet topology is metrizable (and vice versa). Moreover, when $V$ is complete, we call $V$ a $K$-Fr\'echet space; Note that $V$ is just a $K$-Banach space when $\#I_0=1$. Also, note that a topological $K$-vector space $V$ is $K$-Fr\'echet space if and only if $V$ is isomorphic to an inverse limit of $K$-Banach spaces, whose transition maps consist of bounded $K$-linear maps: More precisely, let $V$ be a $K$-Fr\'echet space with valuations $w_0\ge w_1\ge\dots $ and $V_n$ the completion of $V$ with respect to $w_n$. Then, the canonical map $V\to\varprojlim_nV_n$ is an isomorphism of $K$-Fr\'echet spaces. Also, note that if $V$ and $W$ are $K$-Fr\'echet spaces, then $\Hom_K(V,W)$ is again a $K$-Fr\'echet space with respect to the operator norm.

Let $(R,\{w_r\})$ be a $K$-Fr\'echet space with $R$ a ring. If $\{w_r\}$ are multiplicative, then we call $R$ a $K$-Fr\'echet algebra. For a finite free $R$-module $M$, we choose a basis of $M$ and let $\{w_{r,M}\}$ be the $R$-valuations compatible with $\{w_r\}$. Obviously, $(M,\{w_{r,M}\})$ is a $K$-Fr\'echet space. Unless otherwise is mentioned, we always endow a finite free $R$-module with such a family of valuations. 

In the rest of the paper, we omit the prefix ``$K$-'' unless otherwise is mentioned.
\end{notation}

Recall that the category of Fr\'echet spaces is closed under quotient, completed tensor products and direct sums. We also recall that the open mapping theorem holds for Fr\'echet spaces.

\subsection{Continuous derivations over $K$}\label{subsec:der}
In this subsection, we recall the continuous K\"ahler differentials (\cite[\S~4]{Hyo}). In this subsection, let $K$ be a complete discrete valuation field of mixed characteristic $(0,p)$ such that $[k_K:k_K^p]=p^d<\infty$.

\begin{dfn}
Let $\hat{\Omega}^1_{\oo_K}$ be the $p$-adic Hausdorff completion of $\Omega^1_{\oo_K/\Z}$ and put $\hat{\Omega}^1_K:=\hat{\Omega}^1_{\oo_K}[p^{-1}]$. Let $d:K\to\hat{\Omega}^1_K$ be the canonical derivation.
\end{dfn}

Recall that $\hat{\Omega}^1_{K}$ is a finite $K$-vector space with basis $\{dt_j\}_{1\le j\le d}$. Moreover, if $K$ is absolutely unramified, then $\hat{\Omega}^1_{\oo_K}$ is a finite free $\oo_K$-module with basis $\{dt_j\}_{1\le j\le d}$. Also, $\hat{\Omega}^1_{\bullet}$ is compatible with base change, i.e., $L\otimes_{K}\hat{\Omega}^1_{K}\cong\hat{\Omega}^1_L$ for any finite extension $L/K$.

\begin{notation}
Let $R$ be a topological ring and $M$ a topological $R$-module. We denote the set of continuous derivations $d:R\to M$ by $\Der_{\cont}(R,M)$), which is naturally regarded as an $R$-module.
\end{notation}

One can prove the lemma below by d\'evissage and the universality of the usual K\"ahler differentials.
\begin{lem}\label{cor:kahler}
For $M$ an inductive limit of $K$-Fr\'echet spaces, we have the canonical isomorphism
\[
d^*:\Hom_K(\hat{\Omega}^1_K,M)\to\Der_{\cont}(K,M).
\]
\end{lem}

\begin{dfn}
Let $\{\partial_j\}_{1\le j\le d}\subset \Der_{\cont}(K_0,K_0)\cong \Hom_{K_0}(\Omega^1_{K_0},K_0)$ be the dual basis of $\{dt_j\}_{1\le j\le d}$. We call $\{\partial_j\}$ the derivations associated to $\{t_j\}$. We also denote by $\partial_j$ the canonical extension of $\partial_j$ to $\partial_j:K^{\alg}\to K^{\alg}$. Since $\partial_j(t_i)=\delta_{ij}$, we may denote $\partial_j$ by $\partial/\partial t_j$.
\end{dfn}

\subsection{Some Galois extensions}\label{subsec:Galois}
In this subsection, we will fix some notation of a certain Kummer extension, which will be studied later. See \cite[\S~1]{Hyo} for details. In this subsection, let $\wtil{K}$ be an absolutely unramified complete discrete valuation field of mixed characteristic $(0,p)$ with $[k_{\wtil{K}}:k_{\wtil{K}}^p]=p^d<\infty$. We put
\[
\wtil{K}_n:=\wtil{K}(\zeta_{p^n},t_1^{p^{-n}},\dots,t_1^{p^{-n}})\text{ for }n>0,\ \wtil{K}_{\infty}:=\cup_{n>0}{\wtil{K}_{n}},\ \wtil{K}_{\arith}:=\cup_{n>0}{\wtil{K}(\zeta_{p^n})},
\]
\[
\Gamma^{\geom}_{\wtil{K}}:=G_{\wtil{K}_{\infty}/\wtil{K}_{\arith}},\ \Gamma^{\arith}_{\wtil{K}}:=G_{\wtil{K}_{\arith}/\wtil{K}},
\]
\[
\Gamma_{\wtil{K}}:=G_{\wtil{K}_{\infty}/\wtil{K}},\ H_{\wtil{K}}:=G_{\wtil{K}^{\alg}/\wtil{K}_{\infty}}.
\]
Then, we have isomorphisms
\[
\Gamma_{\wtil{K}}^{\arith}\cong\Z_p^{\times},\ \Gamma_{\wtil{K}}^{\geom}\cong \Z_p^d,
\]
which are compatible with the action of $\Gamma^{\arith}_{\wtil{K}}$ on $\Gamma^{\geom}_{\wtil{K}}$. Precisely, the isomorphisms are given as follows: An element $a\in\Z_p^{\times}$ corresponds to $\gamma_a\in \Gamma^{\arith}_{\wtil{K}}$ such that $\gamma_a(\zeta_{p^n})=\zeta_{p^n}^a$ for all $n$. An element $b=(b_j)\in\Z_p^d$ corresponds to $\gamma_b\in\Gamma^{\geom}$ for $1\le j\le d$ such that $\gamma_b(\zeta_{p^n})=\zeta_{p^n}$ for all $n\in\N$ and $\gamma_b(t_{j}^{p^{-n}})=\zeta_{p^n}^{b_j}t_{j}^{p^{-n}}$. By regarding $\Gamma_{\wtil{K}}^{\arith}$ as a subgroup $G_{\wtil{K}_{\infty}/\cup_n\wtil{K}(t_1^{p^{-n}},\dots,t_1^{p^{-n}})}$ of $\Gamma_{\wtil{K}}$, we obtain isomorphisms
\[
\eta=(\eta_0,\dots,\eta_d):\Gamma_{\wtil{K}}\cong \Gamma_{\wtil{K}}^{\arith}\ltimes\Gamma_{\wtil{K}}^{\geom}\cong\Z_p^{\times}\ltimes\Z_p^d.
\]
Since we have a canonical isomorphism
\[
\Z_p^{\times}\ltimes\Z_p^d\cong\begin{pmatrix}
\Z_p^{\times}&\Z_p&\hdots&\Z_p\\
&1&&\\
&&\ddots&\\
&&&1
\end{pmatrix}
\le GL_{d+1}(\Z_p),
\]
the group $\Gamma_{\wtil{K}}$ can be regarded as a classical $p$-adic Lie group with Lie algebra
\[
\mathfrak{g}:=\mathrm{Lie}(\Gamma_{\wtil{K}})\cong\Q_p\ltimes\Q_p^d=\begin{pmatrix}
\Q_p&\dots&\Q_p\\
&&\\
&0&
\end{pmatrix}\subset \mathfrak{gl}_{d+1}(\Q_p).
\]
For a finite extension $L/\wtil{K}$, we put
\[
L_n:=\wtil{K}_{n}L\text{ for }n\in\N_{>0},\ L_{\infty}:=\wtil{K}_{\infty}L
\]
\[
\Gamma_L:=G_{L_{\infty}/L},\ H_L:=G_{\wtil{K}^{\alg}/L_{\infty}}.
\]
Then, $\Gamma_L$ is an open subgroup of $\Gamma_{\wtil{K}}$, hence, there exists an open normal subgroup of $\Gamma_L$, which is isomorphic to an open subgroup of $(1+2p\Z_p)\ltimes \Z_p^d$ by the map $\eta$. Also, we may identify the $p$-adic Lie algebra of $\Gamma_L$ with $\mathfrak{g}$. Finally, we define closed subgroups of $\Gamma_L$
\[
\Gamma_{L,0}:=\{\gamma\in\Gamma_L;\eta_j(\gamma)=0\text{ for all }1\le j\le d\},
\]
\[
\Gamma_{L,j}:=\{\gamma\in\Gamma_L;\eta_0(\gamma)=1, \eta_{i}(\gamma)=0\text{ for all }1\le i\le d,\ i\neq j\}\text{ for }1\le j\le d.
\]
\subsection{Basic construction of Fontaine's rings}\label{subsec:periods}
In this subsection, we will recall the definition of rings of $p$-adic periods due to Fontaine. See \cite[\S~3]{Ohk} for details.

Let $K$ be a complete discrete valuation field of mixed characteristic~$(0,p)$ with $[k_K:k_K^p]=p^d<\infty$. Let $\wtil{\E}^+:=\varprojlim_n\oo_{\C_p}/p\oo_{\C_p}$, where the transition maps are Frobenius. This is a complete valuation ring of characteristic $p$, whose fractional field is denoted by $\wtil{\E}$, which is algebraically closed. We have a canonical identification
\[
\wtil{\E}\cong\{(x^{(n)})_{n\in\N}\in \C_p^{\N};(x^{(n+1)})^p=x^{(n)}\text{ for all }n\in\N\}.
\]
For $x\in \C_p$, we denote by $\wtil{x}\in\wtil{\E}$ an element $\wtil{x}=(x^{(n)})$ such that $x^{(0)}=x$. In particular, we put $\varepsilon:=(1,\zeta_p,\zeta_{p^2},\dots),\ \wtil{t}_j:=(t_j,t_j^{p^{-1}},\dots)\in\wtil{\E}^+$. We define the valuation $v_{\wtil{\E}}$ of $\wtil{\E}$ by $v_{\wtil{\E}}((x^{(n)}))=v_p(x^{(0)})$. We put
\[
\wtil{\A}^+:=W(\wtil{\E}^+)\subset \wtil{\A}:=W(\wtil{\E}),
\]
\[
\wtil{\B}^+:=\wtil{\A}^+[p^{-1}]\subset\wtil{\B}:=\wtil{\A}[p^{-1}],
\]
\[
\pi:=[\varepsilon]-1,\ q:=\pi/\varphi^{-1}(\pi)=\sum_{0\le i<p}[\varepsilon^{p^{-1}}]^i\in\wtil{\A}^+
\]
and we define a surjective ring homomorphism
\[
\theta:\wtil{\B}^+\to\C_p;\sum_{n\gg -\infty}p^n[x_n]\mapsto p^nx_n^{(0)},
\]
which maps $\wtil{\A}^+$ to $\oo_{\C_p}$. Note that $q$ is a generator of the kernel of $\theta|_{\wtil{\A}^+}$.

Let $\mathcal{K}$ be a closed subfield of $\C_p$, whose value group $v_p(\mathcal{K}^{\times})$ is discrete. We will define rings
\[
\A_{\inf,\C_p/\mathcal{K}},\ \B_{\dR,\C_p/\mathcal{K}}^+,\ \B_{\dR,\C_p/\mathcal{K}}.
\]
Let $\A_{\inf,\C_p/\mathcal{K}}$ be the universal $p$-adically formal pro-infinitesimal $\oo_{\mathcal{K}}$-thickening of $\oo_{\C_p}$: If we denote by $\theta_{\C_p/\mathcal{K}}:\oo_{\mathcal{K}}\otimes_{\Z}\wtil{\A}^+\to\oo_{\C_p}$ the linear extension of $\theta$, then $\A_{\inf,\C_p/\mathcal{K}}$ is the $(p,\ker{\theta_{\C_p/\mathcal{K}}})$-adic Hausdorff completion of $\oo_{\mathcal{K}}\otimes_{\Z}\wtil{\A}^+$. The map $\theta_{\C_p/\mathcal{K}}$ extends to $\theta_{\C_p/\mathcal{K}}:\A_{\inf,\C_p/\mathcal{K}}\to\oo_{\C_p}$. Note that $\A_{\inf,\C_p/\Q_p}$ is canonically identified with $\wtil{\A}^+$. Let $\B_{\dR,\C_p/\mathcal{K}}^+$ be the $\ker{\theta_{\C_p/\mathcal{K}}}$-adic Hausdorff completion of $\A_{\inf,\C_p/\mathcal{K}}[p^{-1}]$ and $\theta_{\C_p/\mathcal{K}}:\B_{\dR,\C_p/\mathcal{K}}\to\C_p$ the canonical map induced by $\theta_{\C_p/\mathcal{K}}$. Let
\[
u_j:=t_j-[\wtil{t}_j]\in\A_{\inf,\C_p/K_0},
\]
\[
t:=\log([\varepsilon]):=\sum_{n\ge 1}(-1)^{n-1}\frac{([\varepsilon]-1)^n}{n}\in \B_{\dR,\C_p/\Q_p}^+\subset \B_{\dR,\C_p/\mathcal{K}}^+.
\]
Finally, we define $\B_{\dR,\C_p/\mathcal{K}}:=\B_{\dR,\C_p/\mathcal{K}}^+[t^{-1}]$. These constructions are functorial with respect to $\C_p$ and $\mathcal{K}$, particularly
\[
\A_{\inf,\C_p/\Q_p}\subset \A_{\inf,\C_p/\mathcal{K}},\ \B_{\dR,\C_p/\Q_p}^+\subset\B_{\dR,\C_p/\mathcal{K}}^+,\ \B_{\dR,\C_p/\Q_p}\subset\B_{\dR,\C_p/\mathcal{K}}.
\]
Therefore, any continuous $\mathcal{K}$-algebra automorphism of $\C_p$ acts on these rings. We also have the following explicit descriptions:
\[
\A_{\inf,\C_p/K_0}\cong \wtil{\A}^+[[u_1,\dots,u_d]],\ \B_{\dR,\C_p/K}^+\cong \B_{\dR,\C_p/\Q_p}^+[[u_1,\dots,u_d]]
\]
and $\B_{\dR,\C_p/\Q_p}^+$ is a complete discrete valuation field with uniformizer $t$ and residue field $\C_p$. Also, $\B_{\dR,\C_p/K}^+$ and $\B_{\dR,\C_p/K}$ are invariant after replacing $K$ by a finite extension, in particular, these rings are endowed with canonical $K^{\alg}$-algebra structures.

For $V\in\rep_{\Q_p}(\gk)$, we define $\D_{\dR}(V):=(\B_{\dR,\C_p/K}\otimes_{\Q_p}V)^{\gk}$, which is a finite dimensional $K$-vector space such that $\dim_K\D_{\dR}(V)\le \dim_{\Q_p}V$. When the equality holds, we call $V$ de Rham and denote the category of de Rham representations of $\gk$ by $\rep_{\dR}(\gk)$.

We define the canonical topology of $\B_{\dR,\C_p/\mathcal{K}}^+$ by identifying with $\varprojlim_k\A_{\inf,\C_p/\mathcal{K}}[p^{-1}]/(\ker{\theta_{\C_p/\mathcal{K}}})^k$ endowed with the inverse limit topology, where $\A_{\inf,\C_p/\mathcal{K}}[p^{-1}]/(\ker{\theta_{\C_p/\mathcal{K}}})^k$ is endowed with $\mathcal{K}$-Banach space structure, whose unit disc is the image of $\A_{\inf,\C_p/\mathcal{K}}$. Thus, $\B_{\dR,\C_p/\mathcal{K}}^+$ is a $\mathcal{K}$-Fr\'echet algebra.

The ring $\B_{\dR,\C_p/\mathcal{K}}^+$ is endowed with a continuous $\B_{\dR,\C_p/\Q_p}^+$-linear connection
\[
\nabla^{\geom}:\B_{\dR,\C_p/\mathcal{K}}^+\to \B_{\dR,\C_p/\mathcal{K}}^+\otimes_{\mathcal{K}}\hat{\Omega}^1_{\mathcal{K}},
\]
which is induced by the canonical derivation $d:\mathcal{K}\to\hat{\Omega}^1_{\mathcal{K}}$. More precisely, if we denote by $\{\partial_j\}_{1\le j\le d}$ the derivations of $\B_{\dR,\C_p/K}^+$ given by $\nabla^{\geom}(x)=\sum_j\partial_j(x)\otimes dt'_j$, then $\partial_j$ is the unique $\B_{\dR,\C_p/\Q_p}^+$-linear extension of $\partial/\partial t_j:K\to K$. Thus, we can regard the above connection as a connection associated to a ``coordinate'' $t_1,\dots,t_d$ of $K$, so we put the superscript ``geom''. We denote the kernel of $\nabla^{\geom}$ by $\B_{\dR,\C_p/\mathcal{K}}^{\nabla+}$, which coincides with the image of $\B_{\dR,\C_p/\Q_p}^+$. Therefore, we may identify $\B_{\dR,\C_p/\mathcal{K}}^{\nabla+}$ with $\B_{\dR,\C_p/\Q_p}^+$.

We also define a subring $\wtil{\B}_{\rig,\C_p/\Q_p}^{\nabla+}$ of $\B_{\dR,\C_p/\Q_p}^+$ as follows: Let $\A_{\cris,\C_p/\Q_p}$ be the universal $p$-adically formal $\Z_p$-thickening of $\oo_{\C_p}$, i.e., the $p$-adic Hausdorff completion of the PD-envelope of $\wtil{\A}^+$ with respect to the ideal $\ker{\theta_{\C_p/\Q_p}}$, compatible with the canonical PD-structure on the ideal $(p)$. The construction is functorial, hence, the Frobenius $\varphi:\wtil{\A}^+\to\wtil{\A}^+$ acts on $\A_{\cris,\C_p/\Q_p}$ and $\B_{\cris,\C_p/\Q_p}^+:=\A_{\cris,\C_p/\Q_p}[p^{-1}]$. We define $\B_{\rig,\C_p/\Q_p}^{\nabla+}:=\cap_{n\in\N}\varphi^n(\B_{\cris,\C_p/\Q_p}^+)$, which is the maximal subring of $\B_{\cris,\C_p/\Q_p}^+$ stable under $\varphi$. By construction, $\wtil{\B}^{\nabla+}_{\rig,\C_p/\Q_p}$ is a subring of $\B_{\dR,\C_p/\Q_p}^+\cong \B_{\dR,\C_p/\mathcal{K}}^{\nabla+}$.

Finally, for simplicity, we denote
\[
\B_{\dR}^{\nabla+}:=\B^+_{\dR,\C_p/\Q_p},\ \B_{\dR}^{\nabla}:=\B_{\dR,\C_p/\Q_p},\ \B_{\dR}^+:=\B^+_{\dR,\C_p/K},\ \B_{\dR}:=\B_{\dR,\C_p/K},\ \wtil{\B}_{\rig}^{\nabla+}:=\wtil{\B}_{\rig,\C_p/\Q_p}^{\nabla+}
\]
when no confusion arises.

\subsection{Ramification theory of Abbes-Saito}\label{subsec:AS}
In this subsection, we will review Abbes-Saito's ramification theory. See \cite{AS}, \cite{AS2} for details.

Let $K$ be a complete discrete valuation field with residue field of characteristic~$p$. Let $L/K$ be a finite separable extension. Let $Z=\{z_0,\dots,z_n\}$ be a set of generators of $\oo_L$ as an $\oo_K$-algebra. Let $\oo_K\langle Z_0,\dots,Z_n\rangle\to\oo_L;Z_i\mapsto z_i$ be the corresponding surjective $\oo_K$-algebra homomorphism and $I_Z$ its kernel, where $\oo_K\langle Z_0,\dots,Z_n\rangle$ is Tate algebra. For $a\in\Q_{>0}$, we define the non-log Abbes-Saito space by
\[
as_{L/K,Z}^a:=D^{n+1}(|\pi_K|^{-a}f;f\in I_Z)=\{x\in D^{n+1};|f(x)|\le |\pi_K|^a\ \forall f\in I_Z\},
\]
which is an affinoid subdomain of the $(n+1)$-dimensional polydisc $D^{n+1}$. Let $\pi_0^{\geom}(as_{L/K,Z}^a)$ be the geometric connected components of $as_{L/K,Z}^a$, i.e., the connected components of $as_{L/K,Z}^a\times_KK^{\alg}$ with respect to Zariski topology. We define a $\gk$-set $\mathcal{F}^a(L):=\pi_0^{\geom}(as_{L/K,Z}^a)$ and define the non-log ramification break by
\[
b(L/K):=\inf\{a\in\R;\#\mathcal{F}^a(L)=[L:K]\}\in\Q.
\]
If $L/K$ is Galois, then $\mathcal{F}^a(L)$ can be identified with a quotient of $G_{L/K}$. Moreover, the system $\{\mathcal{F}^a(L)\}_{a}$ of $\gk$-sets defines a filtration $\{G_{L/K}^a\}_{a\in\Q_{\ge 0}}$ of $G_{L/K}$ such that $\mathcal{F}^a(L)\cong G_{L/K}/G_{L/K}^a$ as $\gk$-sets.

There exists a log variation of this construction by considering the following log structure: Let $P\subset Z$ be a subset containing a uniformizer. Take a lift $g_j\in\oo_K\langle Z_0,\dots,Z_n\rangle$ of $z_j^{e_K}/\pi_K^{v_L(z_j)}$ for each $z_j\in P$ and take a lift $h_{i,j}\in\oo_K\langle Z_0,\dots,Z_n\rangle$ of $z_j^{v_L(z_i)}/z_i^{v_L(z_j)}$ for each pair $(z_i,z_j)\in P\times P$. For $a\in\Q_{>0}$, we define the log Abbes-Saito space by
\[
as^a_{L/K,Z,P}:=D^{n+1}\begin{pmatrix}
|\pi_K|^{-a}f\text{ for }f\in I_Z\\
|\pi_K|^{-a-v_L(z_i)}(X_i^{e_{L/K}}-\pi_Kg_i)\text{ for }i\in P\\
|\pi_K|^{-a-v_L(z_i)v_L(z_j)/e_{L/K}}(X_j^{v_L(z_i)}-X_i^{v_L(z_j)}h_{i,j})\text{ for }(i,j)\in P\times P
\end{pmatrix}
\]
as an affinoid subdomain of $D^{n+1}$. As before, we define the $\gk$-set $\mathcal{F}^a_{\log}(L):=\pi_0^{\geom}(as_{L/K,Z,P}^a)$ and define the log ramification break by
\[
b_{\log}(L/K):=\inf\{a\in\R;\#\mathcal{F}^a_{\log}(L)=[L:K]\}\in\Q.
\]
A similar procedure as before defines the log ramification filtration $\{G_{L/K,\log}^a\}_{a\in\Q_{\ge 0}}$ of $G_{L/K}$.

In this paper, we consider only the following simple Abbes-Saito spaces: Let notation be as above. Let $p_0,\dots,p_m$ be a system of generators of the kernel of the surjection $\oo_K\langle X_0,\dots,X_n\rangle\to\oo_L$. Assume that $z_0$ is a uniformizer of $L$ and $p_0=X_0^{e_{L/K}}-\pi_K g_0$ for some $g_0\in\oo_K\langle X_0,\dots,X_n\rangle$. In this case, we have a simple log structure: We put $P:=\{z_0\}$ and we can choose $g_0$ as a lift of $z_0^{e_{L/K}}/\pi_K$. We also choose $1$ as $h_{1,1}$. Hence, Abbes-Saito spaces are given by
\[
as_{L/K,Z}^a=D^{n+1}(|\pi_K|^{-a}p_j\text{ for }0\le j\le m),
\]
\[
as^a_{L/K,Z,P}=D^{n+1}(|\pi_K|^{-a-1}p_0,|\pi_K|^{-a}p_j\text{ for }1\le j\le m).
\]

Let $F/\Q_p$ be a finite extension and $V$ an $F$-representation of $\gk$ with finite local monodromy. We define Abbes-Saito's Artin and Swan conductors by
\begin{align*}
\art^{\AS}(V)&:=\sum_{a\in\Q_{\ge 0}}a\cdot\dim_{F}(V^{\cap_{b>a}G_K^{b}}/V^{G_K^a}),\\
\sw^{\AS}(V)&:=\sum_{a\in\Q_{\ge 0}}a\cdot\dim_{F}(V^{\cap_{b>a}G_{K,\log}^{b}}/V^{G_{K,\log}^a}).
\end{align*}

Note that the above construction does not depend on the auxiliary choices such as $Z$, $P$. Also, note that both Artin and Swan conductors are additive and compatible with unramified base change. When $k_K$ is perfect, the log (resp. non-log) ramification filtration is compatible with the usual upper numbering filtration (resp. shift by one). Moreover, our Artin and Swan conductors coincide with the classical Artin and Swan conductors when $k_K$ is perfect.

\begin{thm}[{Hasse-Arf Theorem, \cite[Theorem~4.5.14]{Xia2}}]\label{thm:mixed}
Assume that $K$ is of mixed characteristic. Let $F/\Q_p$ be a finite extension and $V\in\rep^{\fg}_F(G_K)$. Then, we have $\art(V)\in\Z$ if $K$ is not absolutely unramified; we have $\sw^{\AS}(V)\in\Z$ if $p\neq 2$ and $\sw^{\AS}(V)\in 2^{-1}\Z$ if $p=2$.
\end{thm}

Xiao gives more precise results in the equal characteristic case, as we will see in Theorem~\ref{thm:Xiao}.

\subsection{Overconvergent rings}\label{subsec:ocnotation}
In this subsection, we will recall basic definitions of overconvergent rings associated to complete valuation fields of characteristic~$p$ following \cite[\S~2,3]{mon} and \cite[\S~2]{Doc}.

\begin{construction}[{\cite[\S\S~2.1,\ 2.2]{Doc}}]\label{con:oc}
Let $(E,v)$ be a complete valuation field of characteristic~$p$. Assume that either $E$ is perfect or $v$ is a discrete valuation. We will construct overconvergent ring associated to $E$. We first consider the case where $E$ is perfect. Note that any element of $W(E)[p^{-1}]$ is uniquely expressed as $\sum_{k\gg-\infty}p^k[x_k]$ with $x_k\in E$. For $n\in\Z$, we define a ``partial valuation'' on $W(E)[p^{-1}]$ by
\[
v^{\le n}(\sum_{k\gg-\infty}p^k[x_k]):=\inf_{k\le n}v(x_k).
\]
For $r\in\R_{>0}$, we define
\[
w_{r}(x):=\inf_{n}\{rv^{\le n}(x)+n\},
\]
\[
W(E)_r:=\{x\in W(E);w_r(x)<\infty\}.
\]
Then, $W(E)_r[p^{-1}]$ is a subring of $W(E)[p^{-1}]$ and $w_r$ is a multiplicative valuation of $W(E)_r[p^{-1}]$. Moreover, we have $W(E)_r\subset W(E)_{r'}$ for $r'\le r$. We put $W_{\con}(E):=\varinjlim_{r\to 0}W(E)_r$.

Then, we consider the general case, i.e., we do not need to assume that $E$ is perfect in the following. Let $\Gamma$ be a Cohen ring of $E$ with a Frobenius lift $\varphi$. Then, we can obtain a Frobenius-compatible embedding $\Gamma\hookrightarrow W(E^{\mathrm{pf}})\hookrightarrow W(\hat{E}^{\alg})$, where $\hat{E}^{\alg}$ is the completion of $E^{\alg}$. By using this embedding, we can define $v^{\le n}$ and $w_r$ on $\Gamma$. Moreover, we define $\Gamma_r:=\Gamma\cap W(\hat{E}^{\alg})_r$ and $\Gamma_{\con}:=\varinjlim_{r\to 0}\Gamma_r=\Gamma\cap W_{\con}(\hat{E}^{\alg})$. We say that $\Gamma$ has enough $r$-units if the canonical map $\Gamma_r\to E$ is surjective. We say that $\Gamma$ has enough units if $\Gamma$ has enough $r$-units for some $r>0$. Note that if $E$ is perfect, then $\Gamma$ has enough $r$-units for any $r$. In general, by \cite[Proposition~3.11]{mon}, $\Gamma$ has enough $r$-units for all sufficiently small $r$. In the following, we fix $r_0$ such that $\Gamma$ has enough $r$-units for all $r\le r_0$. Note that $\Gamma_{r}$ for $r<r_0$ is a PID and $\Gamma_{\con}$ is a Henselian local ring with maximal ideal $(p)$, residue field $E$ and fraction field $\Gamma_{\con}[p^{-1}]$ (\cite[Lemma~2.1.12]{Doc}). We endow $\Gamma_r[p^{-1}]$ with the Fr\'echet topology defined by the family of valuations $\{w_s\}_{0<s\le r}$. Let $\Gamma_{\an,r}$ be the completion of $\Gamma_r[p^{-1}]$ with respect to the Fr\'echet topology and $\Gamma_{\an,\con}:=\varinjlim_{r\to 0}\Gamma_{\an,r}$. We can extends $v^{\le n}$ and $w_r$ to $v^{\le n},w_r:\Gamma_{\an,r}\to\mathbb{R}$ and we endow $\Gamma_{\an,r}$ (resp. $\Gamma_{\an,\con}$ ) with Fr\'echet topology defined by $\{w_s\}_{0<s\le r}$ (resp. the inductive limit topology of Fr\'echet topologies). Note that $\varphi(\Gamma_r)\subset\Gamma_{r/p}$, hence, $\varphi$ extends to a map $\varphi:\Gamma_{\an,r}\to\Gamma_{\an,r/p}$. In particular, $\Gamma_{\con}$ and $\Gamma_{\an,\con}$ are canonically endowed with endomorphisms $\varphi$. Also, note that $\Gamma_{\an,r}$ for all $r<r_0$, hence, $\Gamma_{\an,\con}$ are B\'ezout integral domains (\cite[Theorem~2.9.6]{Doc}).
\end{construction}

In the rest of this subsection, we will see explicit descriptions of $\Gamma_{\con}$, together with its finite \'etale extensions, by using rings of overconvergent power series ring.

\begin{notation}\label{notation:ocpower}
Let $\oo$ be a complete discrete valuation ring of mixed characteristic $(0,p)$. We denote by $\oo\{\{S\}\}$ the $p$-adic Hausdorff completion of $\oo((S)):=\oo[[S]][S^{-1}]$. For $r\in\Q_{>0}$, we define the ring of overconvergent power series over $\oo$ as
\[
\oo((S))^{\dagger,r}:=\{f\in\oo\{\{S\}\};f\text{ converges in }0<v_p(S)\le r\},\ \oo((S))^{\dagger}:=\cup_{r>0}\oo((S))^{\dagger,r}.
\]
Recall that $(\oo((S))^{\dagger},(\pi_{\oo}))$ is a Henselian discrete valuation ring (\cite[Proposition~2.2]{Matd}). We also define the Robba ring $\mathcal{R}$ associated to $\oo((S))^{\dagger}$ by
\[
\mathcal{R}:=\{f=\sum_{n\in\Z}a_nS^n;a_n\in\mathrm{Frac}(\oo),\ f\text{ converges on }0<v_p(S)\le r\text{ for some }r>0\}.
\]
\end{notation}

\begin{construction}\label{const:ocext}
We construct a realization of a finite \'etale extension of $\oo((S))^{\dagger}$ as an overconvergent power series ring. Let $\Gamma$ be a Cohen ring of a complete discrete valuation field $E$ of characteristic $p$. By fixing an isomorphism $f:\Gamma\cong\oo\{\{S\}\}$, where $\oo$ is a Cohen ring of $k_E$, we identify $\Gamma$ and $E$ with $\oo\{\{S\}\}$ and $k_E((S))$. Let $\Gamma'/\Gamma$ be a finite \'etale extension with $\Gamma'$ connected and $F/E$ the corresponding residue field extension. Then, $\Gamma'$ is again a Cohen ring of $F$. We identify $F$ with $k_F((T))$ and fix a Cohen ring $\oo'$ of $k_F$. We claim that there exists an isomorphism $f':\Gamma'\cong \oo'\{\{T\}\}$ such that $f'$ modulo $p$ is the identity and $f'(\oo[[S]])\subset\oo'[[T]]$ and $f':\oo[[S]]\to\oo'[[T]]$ is finite flat. We can write $S=T^{e_{F/E}}\bar{u}$ in $\oo_F$ with some $\bar{u}\in\oo_F^{\times}$. We fix $u\in\oo'[[T]]^{\times}$ a lifting of $\bar{u}$ with respect to the projection $\oo'[[T]]\to\oo_F$ and let $s':\Z[S_0]\to\oo'[[T]];S_0\mapsto T^{e_{F/E}} u$ be a ring homomorphism. Let $s:\Z[S_0]\to\oo[[S]]$ be the ring homomorphism sending $S_0$ to $S$. By the formal smoothness of $s$ (cf. \cite[\S~1A]{Ohk}), there exists a local ring homomorphism $\beta:\oo[[S]]\to\oo'[[T]]$:
\[\xymatrix{
\oo[[S]]\ar[r]\ar@{-->}[rrd]^{\beta}&\oo_E\ar[r]&\oo_F\\
\Z[S_0]\ar[u]^s\ar[rr]^{s'}&&\oo'[[T]]\ar[u].
}\]
By the local criteria of flatness and Nakayama's lemma, $\beta$ is finite flat. By the definition of $s$ and $s'$, $\beta$ induces $\beta:\oo((S))\to\oo'((T))$, hence, $\hat{\beta}:\oo\{\{S\}\}\to\oo'\{\{T\}\}$. Since $\hat{\beta}$ is finite \'etale with residue field extension $F/E$, there exists a canonical isomorphism $f':\Gamma'\cong\oo'\{\{T\}\}$, which satisfies the desired properties by the construction of $\beta$.

By the relation $S=T^{e_{F/E}}u$ with $u\in\oo'[[T]]^{\times}$, we have $f'(\oo((S))^{\dagger,r})\subset \oo'((T))^{\dagger,r/e_{F/E}}$. Passing $r\to\infty$, we obtain a flat morphism $f':\oo((S))^{\dagger}\to\oo'((T))^{\dagger}$. Finally, we prove the finiteness of $f':\oo((S))^{\dagger}\to\oo'((T))^{\dagger}$. We fix a basis $\omega_1,\dots,\omega_g$ of $\oo'[[T]]$ as an $\oo[[S]]$-module. Then, we have only to prove that $x\in\oo'((T^{e_{F/E}}))^{\dagger,r}$ is written as $\sum_i\omega_i\sum_na_{i,n}S^n$ with $\sum_na_{i,n}S^n\in\oo((S))^{\dagger,re_{F/E}}$. By the relation $Su^{-1}=T^{e_{F/E}}$ again, any element $x\in \oo'((T^{e_{F/E}}))^{\dagger,r}$ is written as $\sum_{n\in\Z}a_nS^{n}$ with $a_n\in\oo'[[T]]$ such that $|a_n||p|^{e_{F/E}nr}\to 0\ (n\to-\infty)$, where $|\cdot|$ is a norm of $\oo'[[T]]$ associated to the $p$-adic valuation. We write $a_n=\sum_ia_{n,i}\omega_i$. Then, we have $|a_n|=\sup_i{|a_{n,i}|}$, where $|\cdot|$ on the RHS is a norm of $\oo[[S]]$ associated to the $p$-adic valuation. Hence, $\sum_na_{n,i}S^{n}$ belongs to $\oo((S))^{\dagger,re_{F/E}}$, which implies the assertion.
\end{construction}

\begin{lem}[{\cite[Lemma~2.3.5, Corollary~2.3.7]{Doc}}]\label{lem:ocext}
Let $\Gamma$ be a Cohen ring of a complete discrete valuation field $E$ of characteristic~$p$ and $\varphi:\Gamma\to\Gamma$ a Frobenius lift. By fixing an isomorphism $f:\Gamma\cong\oo\{\{S\}\}$, we identify $\Gamma$ and $E$ with $\oo\{\{S\}\}$ and $k_E((S))$. Assume that $\varphi(S)\in\oo((S))^{\dagger}$. Then, we have
\[
\Gamma_{r}=\oo((S))^{\dagger,r},\ \Gamma_{\con}=\oo((S))^{\dagger}
\]
for all sufficiently small $r>0$.

Moreover, let $F/E$ be a finite separable extension and $\Gamma'/\Gamma$ the corresponding finite \'etale extension and $\varphi:\Gamma'\to\Gamma'$ the corresponding Frobenius lift extending $\varphi$. We fix an isomorphism $f':\Gamma'\cong\oo'\{\{T\}\}$ as in Construction~\ref{const:ocext}. Then, $f'$ induces isomorphisms
\[
\Gamma'_{r}\cong\oo'((T))^{\dagger,r/e_{F/E}},\ \Gamma'_{\con}\cong\oo'((T))^{\dagger}
\]
for all sufficiently small $r>0$.
\end{lem}
\begin{proof}
Let $\varphi$ be the Frobenius lift of $\oo'\{\{T\}\}$ obtained by identifying $\oo'\{\{T\}\}$ with $\Gamma'$. We have only to check that the assumption $\varphi(T)\in\oo'((T))^{\dagger}$ in \cite[Convension~2.3.1]{Doc} is satisfied. It follows from the integrally closedness of $\oo'((T))^{\dagger}$ in $\oo'\{\{T\}\}$, which is a consequence of Raynaud's criteria of integral closedness for Henselian pairs (\cite[Th\'eor\`eme~3~(b), Chapitre~XI]{Ray}).
\end{proof}

Finally, we define (pure) $\varphi$-modules over overconvergent rings.
\begin{dfn}[{\cite[Definition~4.6.1]{Doc}}]
Let $R$ be either $\Gamma[p^{-1}]$, $\Gamma_{\con}[p^{-1}]$, or $\Gamma_{\an,\con}$ (Construction~\ref{con:oc}) and let $\sigma:=\varphi^h$ for some $h\in\N_{>0}$. A $\sigma$-module over $R$ is a finite free $R$-module $M$ endowed with semi-linear $\sigma$-action such that $1\otimes \sigma:M\otimes_{R,\sigma}R\to M$ is an isomorphism. Assume that $E$ is algebraically closed. Then, any $\sigma$-module over $\Gamma[p^{-1}]$ or $\Gamma_{\an,\con}$ admits a Dieudonn\'e-Manin decomposition (\cite[Theorem~4.5.7]{Doc}) and we define the slope multiset of $M$ as the multiset of the $p$-adic valuation of the ``eigenvalues''. For a $\sigma$-module $M$ over $\Gamma_{\con}[p^{-1}]$, we define a slope multiset of $M$ as the slope multisets of $\Gamma\otimes_{\Gamma_{\con}[p^{-1}]} M$, which coincides with that of $\Gamma_{\an,\con}\otimes_{\Gamma_{\con}[p^{-1}]} M$. For a general $E$, we define the slope multiset after the base change $\Gamma\to W(\hat{E}^{\alg})$. A $\sigma$-module over $R$ is a pure of slope $s$ if the slope multiset consists of only $s$. If $M$ is a $\sigma$-module of pure of slope $0$, then we call $M$ \'etale.

Let $\varphi$ be a Frobenius lift of $\Gamma:=\oo\{\{S\}\}$ such that $\varphi(S)\subset \oo((S))^{\dagger}$. By regarding $\oo((S))^{\dagger}[p^{-1}]$ and $\mathcal{R}$ in Notation~\ref{notation:ocpower} as $\Gamma_{\con}[p^{-1}]$ and $\Gamma_{\an,\con}$ by using Lemma~\ref{lem:ocext}, we can give similar definitions for $R=\oo((S))^{\dagger}[p^{-1}]$ and $\mathcal{R}$.

When $R$ is one of the above rings, we denote the category of $\sigma$-modules (resp. \'etale $\sigma$-modules, $\sigma$-modules of pure of slope $s$) over $R$ by $\Mod_{R}(\sigma)$ (resp. $\Mod_{R}^{\et}(\sigma)$, $\Mod_{R}^s(\sigma)$).
\end{dfn}

\subsection{Differential Swan conductor}\label{subsec:diffconductor}
The aim of this subsection is to recall the definition of the differential Swan conductor. The following coordinate-free definition of the continuous K\"ahler differentials for overconvergent rings will be useful.

\begin{dfn}
Let $\Gamma$ be an absolutely unramified complete discrete valuation ring of mixed characteristic $(0,p)$. For a subring $R$ of $\Gamma$, we define $\Omega^1_R$ as the $R$-submodule of $\hat{\Omega}^1_{\Gamma}$ generated by the image of $R$ under $d:\Gamma\to\hat{\Omega}^1_{\Gamma}$.
\end{dfn}

\begin{lem}\label{lem:ocdiff}
Let $\Gamma:=\oo\{\{S\}\}$ and $\Gamma^{\dagger}:=\oo((S))^{\dagger}$, where $\oo$ is a Cohen ring of a field $k$ of characteristic $p$. Assume that $[k:k^p]=p^d<\infty$. Then, $\Omega^1_{\Gamma^{\dagger}}$ is the unique $\Gamma^{\dagger}$-submodule $\mathcal{M}$ of $\hat{\Omega}^1_{\Gamma}$ such that
\begin{enumerate}
\item $\mathcal{M}$ is of finite type over $\Gamma^{\dagger}$;
\item The image of $\Gamma^{\dagger}$ under $d:\Gamma\to\hat{\Omega}^1_{\Gamma}$ is contained in $\mathcal{M}$;
\item The canonical map $\Gamma\otimes_{\Gamma^{\dagger}}\mathcal{M}\to\hat{\Omega}^1_{\Gamma}$ is an isomorphism.
\end{enumerate}
Moreover, if $\varphi:\Gamma\to\Gamma$ is a Frobenius lift such that $\varphi(\Gamma^{\dagger})\subset\Gamma^{\dagger}$, then $\Omega^1_{\Gamma^{\dagger}}$ is stable under $\varphi:\hat{\Omega}^1_{\Gamma}\to\hat{\Omega}^1_{\Gamma}$.
\end{lem}

We omit the proof since it is elementary. Note that if $\{t_j\}\subset\oo$ is a lift of a $p$-basis of $k$, then $\Omega^1_{\Gamma^{\dagger}}$ is a free of rank $d+1$ with basis $dS,dt_1,\dots,dt_d$.

\begin{cor}\label{lem:bcomega}
Let notation be as in Lemma~\ref{lem:ocext}. Then, the canonical isomorphism $\Gamma'\otimes_{\Gamma}\hat{\Omega}^1_{\Gamma}\cong\hat{\Omega}^1_{\Gamma'}$ descends to a canonical isomorphism $\Gamma'_{\con}\otimes_{\Gamma_{\con}}\Omega^1_{\Gamma_{\con}}\cong\Omega^1_{\Gamma'_{\con}}$.
\end{cor}

\begin{notation}\label{notation:sw1}
In the rest of this section, let notation be as in Lemma~\ref{lem:ocdiff}. We fix a Frobenius lift $\varphi:\Gamma\to\Gamma$ satisfying $\varphi(\Gamma^{\dagger})\subset\Gamma^{\dagger}$. Let $\mathcal{R}$ be the Robba ring associated to $\Gamma^{\dagger}$ and assume that $\varphi(\mathcal{R})\subset\mathcal{R}$. We put $\Omega^1_{\mathcal{R}}:=\mathcal{R}\otimes_{\Gamma^{\dagger}}\Omega^1_{\Gamma^{\dagger}}$. Note that the canonical derivation $d:\Gamma^{\dagger}\to\Omega^1_{\Gamma^{\dagger}}$ extends to $d:\mathcal{R}\to\Omega^1_{\mathcal{R}}$.
\end{notation}

\begin{dfn}\label{dfn:phinabla}
A $\nabla$-module $M$ over $\mathcal{R}$ is a finite free module over $\mathcal{R}$ together with a connection $\nabla=\nabla_M:M\to M\otimes_{\mathcal{R}}\Omega^1_{\mathcal{R}}$ such that the composition of $\nabla_M$ with the map $M\otimes_{\mathcal{R}}\Omega^1_{\mathcal{R}}\to M\otimes_{\mathcal{R}}\wedge_{\mathcal{R}}^2\Omega^1_{\mathcal{R}}$ induced by $\nabla$ is the zero map. For $h\in\N_{>0}$, a $(\varphi^h,\nabla)$-module $M$ over $\mathcal{R}$ is a $\varphi^h$-module over $\mathcal{R}$ endowed with $\nabla$-module structure commuting with the action of $\varphi^h$. We call a $(\varphi^h,\nabla)$-module pure (resp. \'etale) if the underlying $\varphi^h$-module is pure (resp. \'etale). Similarly, we define notions of (\'etale or pure) $(\varphi^h,\nabla)$-modules over $\Gamma^{\dagger}$ and $\Gamma$. Denote by $\Mod^s_{R}(\varphi^h,\nabla)$ the category of pure $(\varphi^h,\nabla)$-modules over $R$, where $R=\Gamma$, $\Gamma^{\dagger}[p^{-1}]$ and $\mathcal{R}$.
\end{dfn}

\begin{thm}[{\cite[Theorem~3.4.6]{Sw}}]\label{thm:slopefil}
For a $(\varphi,\nabla)$-module $M$ over $\mathcal{R}$, there exists a canonical slope filtration
\[
0=\fil^0(M)\subset\dots \subset\fil^l(M)=M,
\]
whose graded pieces are $(\varphi,\nabla)$-modules of pure of slope $s_1<\dots<s_l$.
\end{thm}

\begin{construction}[{\cite[Definition~3.3.4]{Sw}}]\label{const:dswan}
Let $F/\Q_p$ be a finite unramified extension and $V\in\rep^{\fg}_F(G_E)$. Let $\Gamma^{\dagger,\ur}$ be the maximal unramified extension of $\Gamma^{\dagger}$. We put $\Omega^1_{\Gamma^{\dagger,\ur}}:=\varinjlim\Omega^1_{\Gamma_1^{\dagger}}$, where the limit runs all the finite \'etale extensions $\Gamma_1^{\dagger}/\Gamma^{\dagger}$ with $\Gamma_1^{\dagger}$ connected. We consider the connection
\[
\nabla:\Gamma^{\dagger,\ur}\otimes_{\oo_F}V\to\Omega^1_{\Gamma^{\dagger,\ur}}\otimes_{\oo_F}V;\lambda\otimes y\mapsto d\lambda\otimes y.\quad (*)
\]
Since $\Omega^1_{\Gamma^{\dagger,\ur}}\cong\Gamma^{\dagger,\ur}\otimes_{\Gamma^{\dagger}}\Omega^1_{\Gamma^{\dagger}}$ as $G_E$-modules by Corollary~\ref{lem:bcomega}, by taking $G_E$-invariants of $(*)$, we obtain a connection
\[
\nabla:D^{\dagger}(V)\to\Omega^1_{\Gamma^{\dagger}}\otimes_{\Gamma^{\dagger}}D^{\dagger}(V),
\]
where $D^{\dagger}(V):=(\Gamma^{\dagger,\ur}\otimes_{\oo_F}V)^{G_E}$ is a finite dimensional $\Gamma^{\dagger}[p^{-1}]$-module of rank $\dim_FV$. Thus, we obtain a rank-preserving functor
\[
D^{\dagger}:\rep^{\fg}_F(G_E)\to\Mod_{\Gamma^{\dagger}[p^{-1}]}(\nabla).
\]
By extending scalars, we also obtain a rank-preserving functor
\[
D^{\dagger}_{\rig}:\rep^{\fg}_F(G_E)\to\Mod_{\mathcal{R}}(\nabla).
\]
Note that if $V$ is endowed with a semi-linear action of $\varphi^h$ for $h\in\N$, then $D^{\dagger}(V)$ and $D^{\dagger}_{\rig}(V)$ are also endowed with semi-linear $\varphi^h$-actions.
\end{construction}

\begin{dfn}\label{dfn:diffswan}
For a $\nabla$-module $M$ over $\mathcal{R}$, let $\sw^{\nabla}(M)$ be the differential Swan conductor of $M$ defined in \cite[Definition~2.8.1]{Sw}.
\end{dfn}

Recall that the differential Swan conductor is defined in terms of the behavior of the logarithmic radius of convergence (\cite[Definition~2.3.20]{Xia}), which depends only on the Jordan-H\"older factors of a given $\nabla$-module by definition. In particular, we have

\begin{lem}[The additivity of the differential Swan conductor]\label{lem:diffadd}
Let $0\to M'\to M\to M''\to 0$ be an exact sequence of $\nabla$-modules over $\mathcal{R}$. Then, we have $\sw^{\nabla}(M)=\sw^{\nabla}(M')+\sw^{\nabla}(M'')$.
\end{lem}

The following is Xiao's Hasse-Arf Theorem in the characteristic $p$ case.

\begin{thm}[{\cite[Theorem~4.4.1, Corollary~4.4.3]{Xia}}]\label{thm:Xiao}
Let $V$ be an $F$-representation of $G_{E}$ of finite local monodromy. Then, we have
\[
\sw^{\AS}(V)=\sw^{\nabla}(D^{\dagger}_{\rig}(V)).
\]
Moreover, these invariants are non-negative integers.
\end{thm}

\subsection{Scholl's fields of norms}\label{subsec:scholl}
In this subsection, we recall some results of Scholl (\cite[\S~1.3]{Sch}), which is a generalization of Fontaine-Wintenberger's fields of norms. Throughout this subsection, let $K$ be a complete discrete valuation field of mixed characteristic $(0,p)$ with $[k_K:k_K^p]=p^d<\infty$.

\begin{dfn}\label{dfn:scholl}
Let $K_1\subset K_2\subset\dots$ be finite extensions of $K$ and put $K_{\infty}=\cup{K_n}$. We say that a tower $\mathfrak{K}:=\{K_n\}_{n>0}$ is strictly deeply ramified if there exists $n_0>0$ and an element $\xi\in\oo_{K_{n_0}}$ such that $0<v_p(\xi)\le 1$, and such that the following condition holds: For every $n\ge n_0$, the extension $K_{n+1}/K_n$ has degree $p^{d+1}$, and there exists a surjection $\Omega^1_{\oo_{K_{n+1}}/\oo_{K_n}}\to (\oo_{K_{n+1}}/\xi\oo_{K_{n+1}})^{d+1}$ of $\oo_{K_{n+1}}$-modules.
\end{dfn}

Let $\mathfrak{K}=\{K_n\}_{n>0}$ be a strictly deeply ramified tower. For $n\ge n_0$, we have $e_{K_{n+1}/K_n}=p$ and $k_{K_{n+1}}=k_K^{p^{-1}}$, and the Frobenius $\oo_{K_{n+1}}/\xi\oo_{K_{n+1}}\to\oo_{K_{n+1}}/\xi\oo_{K_{n+1}}$ induces a surjection $f_n:\oo_{K_{n+1}}/\xi\oo_{K_{n+1}}\to\oo_{K_n}/\xi \oo_{K_n}$. We can also choose a uniformizer $\pi_{K_n}$ of $K_n$ such that $\pi_{K_{n+1}}^p\equiv\pi_{K_n}\mod{\xi\oo_{K_n}}$. Then, we define $X^+:=X^+(\mathfrak{K},\xi,n_0):=\varprojlim_{n\ge n_0}\oo_{K_n}/\xi \oo_{K_n}$, where the transition maps are $\{f_n\}$. Let $\pr_n:X^+\to\oo_{K_n}/\xi \oo_{K_n}$ be the $n$-th projection for $n\ge n_0$. We put $\Pi:=(\pi_{K_n}\mod{\xi\oo_{K_n}})\in X^+$. Let $k_{\mathfrak{K}}:=\varprojlim_{n\ge n_0}k_{K_n}$ where the transition maps are the maps induced by $f_n$'s. Since $k_{K_{n+1}}= k_{K_n}^{p^{-1}}$, the projection $\pr_n:k_{\mathfrak{K}}\to k_{K_n}$ for all $n\ge n_0$ are isomorphisms. Moreover, $X^+$ is a complete discrete valuation ring of characteristic $p$, with uniformizer $\Pi$ and residue field $k_{\mathfrak{K}}$. The construction does not depend on $\xi$ and $n_0$, also $X^+$ is invariant after changing $\{K_n\}_n$ by $\{K_{n+m}\}_n$ for some $m$. Hence, we may denote $X^+(\mathfrak{K},\xi,n_0)$ by $X^+_{\mathfrak{K}}$ and denote the fractional field of $X^+_{\mathfrak{K}}$ by $X_{\mathfrak{K}}$. Note that if $K_n/K$ is Galois for all $n$, then $X^+_{\mathfrak{K}}$ and $X_{\mathfrak{K}}$ are canonically endowed with $G_{K_{\infty}/K}$-actions by construction.

\begin{example}[Kummer tower case]\label{ex:Kummer}
Let $K=\wtil{K}$ and $\{L_n\}$ be as in \S~\ref{subsec:Galois}. Then, $\{L_n\}$ is strictly deeply ramified (\cite[Example~6.2]{JNT}).
\end{example}

Let $L_{\infty}/K_{\infty}$ be a finite extension. We can choose a finite extension $L/K$ such that $L_{\infty}=LK_{\infty}$. Then, the tower $\mathfrak{L}:=\{L_n:=LK_n\}$ depends only on $L_{\infty}$ up to shifting, and is also strictly deeply ramified with respect to any $\xi'\in K_{n_0}$ with $0<v_p(\xi')<v_p(\xi)$ (\cite[Theorem~1.3.3]{Sch}). Note that if $L_n/K$ is Galois for all $n$, then $X^+_{\mathcal{L}}$ and $X_{\mathcal{L}}$ are canonically endowed with $G_{L_{\infty}/K}$-actions by construction.

\begin{thm}[{\cite[Theorem~1.3.4]{Sch}}]\label{thm:sch}
Let notation be as above. Denote the category of finite \'etale algebras over $K_{\infty}$ (resp. $X_{\mathfrak{K}}$) by $\mathbf{F\acute{E}t}_{K_{\infty}}$ (resp. $\mathbf{F\acute{E}t}_{X_{\mathfrak{K}}}$). Then, the functor
\[
X_{\bullet}:\mathbf{F\acute{E}t}_{K_{\infty}}\to \mathbf{F\acute{E}t}_{X_{\mathfrak{K}}};L_{\infty}\mapsto X_{\mathfrak{L}}
\]
is an equivalence of Galois categories. In particular, the corresponding fundamental groups are isomorphic, i.e., $G_{K_{\infty}}\cong G_{X_{\mathfrak{K}}}$. Moreover, the sequences $\{[L_n:K_n]\}_n$, $\{e_{L_n/K_n}\}_n$ and $\{[k_{L_n}:k_{K_n}]\}_n$ are stationary for sufficiently large $n$. Their limits are equal to $[X_{\mathfrak{L}}:X_{\mathfrak{K}}]$, $e_{X_{\mathfrak{L}}/X_{\mathfrak{K}}}$ and $[k_{X_{\mathfrak{L}}}:k_{X_{\mathfrak{K}}}]$.
\end{thm}

\subsection{$(\varphi,\Gamma_K)$-modules}\label{subsec:phigamma}
Throughout this subsection, let $K$ be a complete discrete valuation field of mixed characteristic $(0,p)$. In this subsection, we will recall about $(\varphi,\Gamma_K)$-modules in the Kummer tower case (\cite{And}). To avoid complications, especially verifying the assumption \cite[(2.1.2)]{Sch}, we will assume the following to work under the settings of \cite{And}, \cite{AB} and \cite{AB2}.

\begin{assumption}[{\cite[\S~1]{And}}]\label{ass:relative}
Let $\mathcal{V}$ be a complete discrete valuation field of mixed characteristic~$(0,p)$ with perfect residue field. Let $R_0$ be the $p$-adic Hausdorff completion of $\mathcal{V}[T_1,\dots,T_d][1/T_1\dots T_d]$ and $\wtil{R}$ a ring obtained from $R_0$ iterating finitely many times the following operations:
\begin{enumerate}
\item[(\'et)] The $p$-adic Hausdorff completion of an \'etale extension;
\item[(loc)] The $p$-adic Hausdorff completion of the localization with respect to a multiplicative system;
\item[(comp)] The Hausdorff completion with respect to an ideal containing $p$.
\end{enumerate}
We assume that there exists a finite flat morphism $\wtil{R}\to\oo_K$, which sends $T_j$ to $t_j$.
\end{assumption}

Note that $\wtil{R}$ is an absolutely unramified complete discrete valuation ring. Denote $\wtil{R}$ by $\oo_{\wtil{K}}$ and $\mathrm{Frac}(\wtil{R})$ by $\wtil{K}$. Let $L/\wtil{K}$ be a finite extension. In the rest of this subsection, we will use notation as in \S\S~\ref{subsec:Galois}, \ref{subsec:periods}. We also apply the results of \S~\ref{subsec:scholl} to Kummer tower $\{L_n\}_{n>0}$.

\begin{notation}[{\cite[\S~4.1]{AB}}]
We will denote
\[
\E_L^+:=X^+_{\mathfrak{L}},\ \E_L:=X_{\mathfrak{L}}.
\]
For any non-zero $\xi\in p\oo_{L_{\infty}}$, we put
\[
\wtil{\E}^+_L:=\varprojlim_{x\mapsto x^p}\oo_{L_{\infty}}/\xi\oo_{L_{\infty}},\ \wtil{\E}_L:=\mathrm{Frac}(\wtil{\E}^+_L)
\]
where both rings are independent of the choice of $\xi$. We also put
\[
\wtil{\A}_L^+:=W(\wtil{\E}_L^+),\ \wtil{\A}_L:=W(\wtil{\E}_L),\ \wtil{\B}_L:=\wtil{\A}_L[p^{-1}].
\]
\end{notation}

By definition, we have $\E_L^+\subset\wtil{\E}^+_L$, $\E_L\subset\wtil{\E}_L$ and $\wtil{\E}_L$ can be regarded as a closed subring of $\wtil{\E}$. In particular, $\wtil{\A}_L^+$, $\wtil{\A}_L$ and $\wtil{\B}_L$ can be regarded as subrings of $\wtil{\A}^+$, $\wtil{\A}$ and $\wtil{\B}$. Note that the completion of an algebraic closure of $\E_L$ coincides with $\wtil{\E}$. Moreover, $\wtil{\E}$ is perfect and $(\wtil{\E}_L,v_{\wtil{\E}})$ is a perfect complete valuation field, whose integer ring is $\wtil{\E}_L^+$. By using the $G_{\wtil{K}}$-actions on $\wtil{\E}$ and $\wtil{\A}$, we can write (\cite[Lemme~4.1]{AB})
\[
\wtil{\E}^+_L=(\wtil{\E}^+)^{H_L},\ \wtil{\E}_L=\wtil{\E}^{H_L},\ \wtil{\A}_L=\wtil{\A}^{H_L},\ \wtil{\B}_L=\wtil{\B}^{H_L}.
\]

\begin{lem}[{A special case of \cite[Proposition~4.42]{AB}}]\label{lem:plus}
We put $\A_{W(k_{\mathcal{V}})}^+:=W(k_{\mathcal{V}})[[\pi]]\subset\wtil{\A}^+$ where $\pi=[\varepsilon]-1\in\wtil{\A}^+$. Let $L/\wtil{K}$ be a finite extension. The weak topology of $\wtil{\A}_L\cong\wtil{\E}_L^{\N}$ is the product topology $\wtil{\E}_L^{\N}$, where $\wtil{\E}_L$ is endowed with the valuation topology. Then, there exists a unique subring $\A_L$ of $\wtil{\A}_L$ such that:
\begin{enumerate}
\item $\A_L$ is complete for the weak topology;
\item $p\wtil{\A}_L\cap \A_L=p\A_L$;
\item One has an commutative diagram
\[\xymatrix{
\A_L\ar@{->>}[r]\ar[d]&\E_L\ar[d]\\
\wtil{\A}_L\ar@{->>}[r]&\wtil{\E}_L
}\]
\item $[\varepsilon],[\wtil{t}_j]\in \A_L$ for all $j$;
\item There exists an $\A_{W(k)}^+$-subalgebra $\A_L^+$ of $\A_L$ and $r_L\in\Q_{>0}$ such that:
\begin{enumerate}
\item There exists $a\in\N$ such that $p/\pi^a\in \A_L^+$ and $\A^+_L/(p/\pi^a)\cong \E_L^+$;
\item If $\alpha,\beta\in\N_{>0}$ such that $\alpha/\beta<pr_L/(p-1)$, one has $\A_L^+\subset\wtil{\A}^+_L\{p^{\alpha}/\pi^{\beta}\}$, where $\wtil{\A}^+_L\{p^{\alpha}/\pi^{\beta}\}$ denotes the $p$-adic Hausdorff completion of $\wtil{\A}^+_L[p^{\alpha}/\pi^{\beta}]$;
\item $\A_L^+$ is complete for the weak topology.
\end{enumerate}
\end{enumerate}
Moreover, by the uniqueness, $\A_L$ is stable under the actions of $\varphi$ and $G_{L_{\infty}/K}$ if $L/\wtil{K}$ is Galois.
\end{lem}

\begin{dfn}
Let $\A$ be the $p$-adic Hausdorff completion of $\cup_{L/\wtil{K}}\A_L$, which is a subring of $\wtil{\A}$, stable under the actions of $G_K$ and $\varphi$. We also put $\B_L:=\A_L[p^{-1}]$ and $\B:=\A[p^{-1}]$.
\end{dfn}

\begin{rem}\label{rem:aplus}
\begin{enumerate}
\item As is remarked in \cite[\S~4.3]{AB}, $\A_L$ is the unique finite \'etale $\A_{\wtil{K}}$-algebra corresponding to $\E_L/\E_{\wtil{K}}$, in particular, $\A_L$ is a Cohen ring of $\E_L$.
\item The action of $\Gamma_{\wtil{K}}$ on $\A_{\wtil{K}}$ is determined by the action of $\Gamma_{\wtil{K}}$ on $\pi,[\wtil{t}_1],\dots,[\wtil{t}_d]$, since $\varepsilon-1,\wtil{t}_1,\dots,\wtil{t}_d$ forms a $p$-basis of $\E_{\wtil{K}}$. Explicit descriptions are given as follows:
\[
\gamma_a(\pi)=(1+\pi)^a-1,\ \gamma_a([\wtil{t}_j])=[\wtil{t}_j]\text{ for }a\in\Z_p^{\times},
\]
\[
\gamma_b(\pi)=\pi,\ \gamma_b([\wtil{t}_j])=(1+\pi)^{b_j}[\wtil{t}_j]\text{ for }b=(b_j)\in\Z_p^d.
\]
\end{enumerate}
\end{rem}

\begin{dfn}
For $h\in\N_{>0}$, an \'etale $(\varphi^h,\Gamma_L)$-module $M$ over $\B_L$ is an \'etale $\varphi^h$-module over $\B_L$ endowed with semi-linear continuous $\gk$-action commuting with the action of $\varphi^h$. Denote by $\Mod^{\et}_{\B_L}(\varphi^h,\Gamma_L)$ the category of \'etale $(\varphi^h,\Gamma_L)$-modules over $\B_L$.

For $V\in \rep_{\Q_{p^h}}(G_L)$, let $\D(V):=(\B\otimes_{\Q_{p^h}}V)^{H_L}$. For $M\in \Mod^{\et}_{\B_L}(\varphi^h,\Gamma_L)$, let $\V(M):=(\B\otimes_{\B_K}M)^{\varphi^h=1}$. 
\end{dfn}

\begin{thm}[{\cite[Theorem~7.11]{And} or \cite[Th\'eor\`eme~4.34]{AB}}]\label{thm:phigamma}
Let $h\in\N_{>0}$. Then, the functor $\D$ gives a rank-preserving equivalence of categories
\[
\D:\rep_{\Q_{p^h}}(G_L)\to \Mod^{\et}_{\B_L}(\varphi^h,\Gamma_L)
\]
with a quasi-inverse $\V$.
\end{thm}

\subsection{Overconvergence of $p$-adic representations}\label{subsec:oc}
In this subsection, we will recall the overconvergence of $p$-adic representations in \cite{AB}. We still keep the notation in \S~\ref{subsec:phigamma} and Assumption~\ref{ass:relative}.

\begin{dfn}
We apply Construction~\ref{con:oc} to $(\wtil{\E},v_{\wtil{\E}})$ with $\Gamma=\wtil{\A}$ and we denote
\[
\wtil{\A}^{\dagger,r}:=\Gamma_r,\ \wtil{\A}^{\dagger}:=\Gamma_{\con},\ \wtil{\B}^{\dagger,r}:=\Gamma_r[p^{-1}],\ \wtil{\B}^{\dagger}:=\Gamma_{\con}[p^{-1}],\ \wtil{\B}^{\dagger,r}_{\rig}:=\Gamma_{\an,r},\ \wtil{\B}_{\rig}^{\dagger}:=\Gamma_{\an,\con}.
\]
We define $v_{\wtil{\E}}^{\le n}$ and $w_r$ by the same way. For a finite extension $L/\wtil{K}$, we apply a similar construction to the following $(E,v_{\wtil{\E}})$ with $\Gamma$ and we denote:
\[
\begin{tabular}{|llllllll|}
\hline
$\Gamma$&$E$&$\Gamma_r$&$\Gamma_{\con}$&$\Gamma_r[p^{-1}]$&$\Gamma_{\con}[p^{-1}]$&$\Gamma_{\an,r}$&$\Gamma_{\an,\con}$ \\
\hline
$\A$ &$\E$&$\A^{\dagger,r}$ & $\A^{\dagger}$ &$\B^{\dagger,r}$ &$\B^{\dagger}$&$\B^{\dagger,r}_{\rig}$&$\B^{\dagger}_{\rig}$ \\
$\wtil{\A}_L$ &$\wtil{\E}_L$&$\wtil{\A}^{\dagger,r}_L$ & $\wtil{\A}^{\dagger}_L$ &$\wtil{\B}^{\dagger,r}_L$ &$\wtil{\B}^{\dagger}_L$&$\wtil{\B}^{\dagger,r}_{\rig,L}$&$\wtil{\B}^{\dagger}_{\rig,L}$ \\
$\A_L$ &$\E_L$&$\A^{\dagger,r}_L$ & $\A^{\dagger}_L$ &$\B^{\dagger,r}_L$ &$\B^{\dagger}_L$&$\B^{\dagger,r}_{\rig,L}$&$\B^{\dagger}_{\rig,L}$ \\
\hline
\end{tabular}
\]
By construction, we have $\wtil{\B}^{\dagger}=\cup_r{\wtil{\B}^{\dagger,r}}$, $\B^{\dagger}=\cup_r{\B^{\dagger,r}}$, $\wtil{\B}_K^{\dagger,r}=\wtil{\B}_K\cap \wtil{\B}^{\dagger,r}$ and $\wtil{\B}_K^{\dagger}=\cup_r{\wtil{\B}_K^{\dagger,r}}$, $\B_K^{\dagger,r}=\B_K\cap \B^{\dagger,r}$ and $\B_K^{\dagger}=\cup_r{\B_K^{\dagger,r}}$. We endow $\wtil{\B}^{\dagger,r}$ and $\wtil{\B}^{\dagger,r}_{\rig}$ ... etc. with the Fr\'echet topology defined by $\{w_s\}_{0<s\le r}$.
\end{dfn}

We can describe $\A_L^{\dagger}$ by using the ring of overconvergent power series.

\begin{lem}[{cf. \cite[Proposition~1.4]{Inv}}]\label{lem:ocexplicit}
Let $\oo$ be a Cohen ring of $k_{\E_{\wtil{K}}}$. Then, there exists an isomorphism $\A_{\wtil{K}}\cong\oo\{\{\pi\}\}$, which induces an isomorphism $\A_{\wtil{K}}^{\dagger,r}\cong\oo((\pi))^{\dagger}$ for all sufficiently small $r>0$. Similarly, there exists an isomorphism $\A_L\cong\oo'\{\{\pi'\}\}$, which induces $\A_L^{\dagger,r}\cong\oo'((\pi'))^{\dagger,r/e_{\E_L/\E_{\wtil{K}}}}$, where $\oo'$ is a Cohen ring of $k_{\E_{L}}$.
\end{lem}
\begin{proof}
Fix any isomorphism $\A_{\wtil{K}}\cong\oo\{\{\pi\}\}$ (Remark~\ref{rem:aplus}~(i)). Since $\varphi(\pi)=[\varepsilon]^p-1=(1+\pi)^p-1\in\oo\{\{\pi\}\}^{\dagger}$, the assertion follows from Lemma~\ref{lem:ocext}.
\end{proof}

\begin{notation}
By using the isomorphism in Lemma~\ref{lem:ocexplicit}, we can apply the results in \S~\ref{subsec:diffconductor}. In particular, for any finite extension $L/\wtil{K}$, we have a canonical continuous derivation
\[
d:\B_{\rig,L}^{\dagger}\to\Omega^1_{\B_{\rig,L}^{\dagger}},
\]
where $\Omega^1_{\B^{\dagger}_{\rig,L}}:=\B_{\rig,L}^{\dagger}\otimes_{\A^{\dagger}_L}\Omega^1_{\A^{\dagger}_L}$ is a free $\B^{\dagger}_{\rig,L}$-module with basis $d\pi,d[\wtil{t}_1],\dots,d[\wtil{t}_d]$. We may speak about $(\varphi,\nabla)$-modules over $\B_{\rig,L}^{\dagger}$ and the associated differential Swan conductors.
\end{notation}

\begin{dfn}
Let $h\in\N_{>0}$. An \'etale $(\varphi^h,\Gamma_L)$-module $M$ over $\B^{\dagger}_L$ is an \'etale $\varphi^h$-module over $\B^{\dagger}_L$ endowed with continuous semi-linear $\gk$-action commuting with $\varphi^h$-action. Denote by $\Mod^{\et}_{\B_L^{\dagger}}(\varphi^h,\Gamma_L)$ the category of \'etale $(\varphi^h,\Gamma_L)$-modules over $\B^{\dagger}_L$.

For $V\in \rep_{\Q_{p^h}}(G_L)$, let
\[
\D^{\dagger,r}(V):=(\B^{\dagger,r}\otimes_{\Q_{p^h}}V)^{H_L},\ \D^{\dagger}(V)=\cup_r\D^{\dagger,r}(V),
\]
\[
\D^{\dagger,r}_{\rig}(V):=\B^{\dagger,r}_{\rig,L}\otimes_{\B^{\dagger,r}_L}\D^{\dagger,r}(V),\ \D^{\dagger}_{\rig}(V)=\cup_r\D^{\dagger,r}_{\rig}(V).
\]
For $M\in \Mod^{\et}_{\B_L^{\dagger}}(\varphi^h,\Gamma_L)$, let $\V(M):=(\B^{\dagger}\otimes_{\B_L^{\dagger}}M)^{\varphi^h=1}$.
\end{dfn}

\begin{thm}[{\cite[Theorem~4.35]{AB}}]\label{thm:oc}
Let $h\in\N_{>0}$. The functor $\D^{\dagger}$ gives a rank-preserving equivalence of categories
\[
\D^{\dagger}:\rep_{\Q_{p^h}}(G_L)\to \Mod^{\et}_{\B_L^{\dagger}}(\varphi^h,\Gamma_L)
\]
with a quasi-inverse $\V$. Moreover, $\D^{\dagger}$ and $\V$ are compatible with $\D$ and $\V$ in Theorem~\ref{thm:phigamma}. Furthermore, for all sufficiently small $r$, $\D^{\dagger,r}(V)$ is free of rank $\dim_{\Q_{p^h}}V$ over $\B^{\dagger,r}_K$ and we have a canonical isomorphism $\B^{\dagger}_K\otimes_{\B^{\dagger,r}_K}\D^{\dagger,r}(V)\to\D^{\dagger}(V)$.
\end{thm}

The functor $\D^{\dagger}_{\rig}$ will be studied in \S~\ref{subsec:pure}.

\section{Adequateness of overconvergent rings}\label{sec:adeq}
In this section, we will prove the ``adequateness'', which assures the elementary divisor theorem, for overconvergent rings defined in \S~\ref{subsec:ocnotation}. The adequateness of overconvergent rings seems to be well-known to the experts: At least when the overconvergent ring is isomorphic the Robba ring, the adequateness follows from Lazard's results (\cite{Laz}) as in \cite[Proposition~4.12~(5)]{Inv}). Since the author could not find an appropriate reference, we give a proof.

\begin{dfn2.0.1}[{\cite[\S~2]{Hel}}]\label{dfn:Hel}
An integral domain $R$ is adequate if the following hold:
\begin{enumerate}
\item $R$ is a B\'ezout ring, that is, any finitely generated ideal of $R$ is principal;
\item For any $a,b\in R$ with $a\neq 0$, there exists a decomposition $a=a_1a_2$ such that $(a_1,b)=R$ and $(a_3,b)\neq R$ for any non-unit factor $a_3$ of $a_2$.
\end{enumerate}
\end{dfn2.0.1}

Recall that if $R$ is an adequate integral domain, then the elementary divisor theorem holds for free $R$-modules (\cite[Theorem~3]{Hel}). Precisely speaking, let $N\subset M$ be finite free $R$-modules of rank $n$ and $m$ respectively. Then, there exists a basis of $e_1,\dots,e_m$ (resp. $f_1,\dots,f_n$) of $M$ (resp. $N$) and non-zero elements $\lambda_1|\dots|\lambda_n\in R$ such that $f_i=\lambda_ie_i$ for $1\le i\le n$.

In the rest of this section, let notation be as in Construction~\ref{con:oc}. We fix $r_0>0$ such that $\Gamma$ has enough $r_0$-units and let $r\in (0,r_0)$ unless otherwise is mentioned. Recall that $\Gamma_{\an,r}$ is a B\'ezout integral domain.

\begin{dfn2.0.2}
We recall basic terminologies (\cite[\S~3.5]{mon}). For $x\in \Gamma_{\an,r}$ non-zero, we define the Newton polygon of $x$ as the lower convex hull of the set of points $(v^{\le n}(x),n)$, minus any segments of slope less than $-r$ on the left end and/or any segments of non-negative slope on the right end of the polygon. We define the slopes of $x$ as the negatives of the slopes of the Newton polygon of $x$. We also define the multiplicity of a slope $s\in (0,r]$ of $x$ as the positive difference in $y$-coordinates between the endpoints of the segment of the Newton polygon of slope $-s$, or $0$ if there is no such segment. If $x$ has only one slope $s$, we say that $x$ is pure of slope $s$.

A slope factorization of a non-zero element $x$ of $\Gamma_{\an,r}$ is a Fr\'echet-convergent product $x=\Pi_{1\le i\le n}x_i$ for $n$ a positive integer or $\infty$, where each $x_i$ is pure of slope $s_i$ with $s_1>s_2>\dots$ (cf. an explanation before \cite[Lemma~3.26]{mon}).
\end{dfn2.0.2}

Recall that the multiplicity is compatible with multiplication, i.e., the multiplicity of a slope $s$ of $xy$ is the sum of its multiplicities as a slope of $x$ and of $y$ (\cite[Corollary~3.22]{mon}). Also, recall that $x\in\Gamma_{\an,r}$ is a unit if and only if $x$ has no slopes (\cite[Corollary~2.5.12]{Doc}).

\begin{lem2.0.3}[{\cite[Lemma~3.26]{mon}}]
Every non-zero element of $\Gamma_{\an,r}$ has a slope factorization.
\end{lem2.0.3}

For simplicity, we denote $\Gamma_{\an,r}$ by $R$ in the rest of this subsection. The lemma below is an immediate consequence of B\'ezoutness of $R$ and the additivity of the multiplicity of a slope.
\begin{lem2.0.4}\label{lem:pureslope}
\begin{enumerate}
\item Let $x,y\in R$ such that $x$ is pure of slope $s$ and let $z$ be a generator of $(x,y)$. Then, $z$ is also pure of slope $s$ with multiplicity less than or equal to the multiplicity of slope $s$ of $x$. In particular, if the multiplicity of slope $s$ of $y$ is equal to zero, then $z$ is a unit and we have $(x,y)=R$.
\item Let $x,y\in R$ such that $x$ is a pure of slope $s$. Then, the decreasing sequence of the ideals $\{(x,y^n)\}_{n\in\N}$ is eventually stationary.
\end{enumerate}
\end{lem2.0.4}

\begin{lem2.0.5}[The uniqueness of slope factorizations]\label{lem:uniquefac}
Let $x\in R$ be a non-zero element. Let $x=\Pi_i{x_i}=\Pi_i{x'_i}$ be slope factorizations, whose slopes are $s_1>s_2>\dots$ and $s'_1>s'_2>\dots$. Let $m_i$ and $m'_i$ be the multiplicities of $s_i$ and $s'_i$ for $x_i$ and $x'_i$. Then, we have $s_i=s'_i$ and $x_i=x'_iu_i$ for some $u_i\in R^{\times}$. In particular, we have $m_i=m'_i$.
\end{lem2.0.5}
\begin{proof}
We can easily reduce to the case $i=1$. Since the multiplicity of slope $s_1$ of $\Pi_{i>1}x'_i$ is equal to zero, we have $(x_1,\Pi_{i>1}x'_i)=R$ by Lemma~2.0.4~(i). Hence, we have $(x_1,x)=(x_1,x_1\Pi_{i>1}x_i)=(x_1)$. By assumption, we have $s_1\neq s'_j$ except at most one $j$. Similarly as above, we have
\[
(x_1,x)=(x_1,x'_j\Pi_{i\neq j}x'_i)=(x_1,x'_j)=(x_1\Pi_{i>1}x_i,x'_j)=(x,x'_j)=(x'_j\Pi_{i\neq j}x'_i,x'_j)=(x'_j),
\]
i.e., $(x_1)=(x'_j)$.
Hence, there exists $u\in R^{\times}$ such that $x_1=x'_ju$. By the same argument, $x'_1=x_{l}u'$ for some $l$ and $u'\in R^{\times}$. Since $\{s_i\}$ and $\{s'_i\}$ are strictly decreasing, we must have $j=l=1$, which implies the assertion.
\end{proof}

\begin{lem2.0.6}\label{lem:adeq}
The integral domain $\Gamma_{\an,r}$ is adequate. In particular, the elementary divisor theorem holds over $\Gamma_{\an,r}$.
\end{lem2.0.6}
\begin{proof}
We have only to prove the condition~(ii) in Definition~2.0.1. Let $a,b\in R$ with $a\neq 0$. If $b=0$, then it suffices to put $a_1=1$, $a_2=a$. If $b$ is a unit, then it suffices to put $a_1=a$, $a_2=1$. Therefore, we may assume that $b$ is neither a unit nor zero. Let $b=\Pi_{i>0}b_i$ be a slope factorization with slopes $s_1>s_2>\dots$. By Lemma~2.0.4~(ii), there exists $z_i\in R$ such that $(a,b_i^n)=(z_i)$ for all sufficiently large $n$. By \cite[Proposition~3.13]{mon}, we may assume that $z_i$ admits a semi-unit decomposition: That is, $z_i$ is equal to a convergent sum of the form $1+\sum_{j<0}u_{i,j}p^j$, where $u_{i,j}\in R^{\times}\cup\{0\}$. As in the proof of \cite[Lemma~3.26]{mon}, we can prove that $\{z_1\dots z_i\}_{i>0}$ converges. Then, we claim that there exists $u_i\in R$ such that $a=z_1\dots z_i u_i$. We proceed by induction on $i$. By definition, we have $a=z_1u_1$ for some $u_1$. Assume that we have defined $u_i$. Since the multiplicity of slope $s_{i+1}$ of $z_j$ is equal to zero for $1\le j\le i$, we have $(z_j,z_{i+1})=R$ for $1\le j\le i$. Hence, we have $(z_{i+1})=(a,z_{i+1})=(z_1\dots z_{i}u_i,z_{i+1})=(u_i,z_{i+1})$, which implies $z_{i+1}|u_i$. Therefore, $u_{i+1}:=u_i/z_{i+1}$ satisfies the condition. By this proof, we can choose $u_i=u_1/(z_1\dots z_i)$. We put $a_1:=\lim_{i\to\infty}u_i=u_1/\Pi_{i>1}z_i$ and $a_2:=\Pi_{i>0}z_i$, which is a slope factorization of $a_2$. We prove that the factorization $a=a_1a_2$ satisfies the condition. We first prove $(a_1,b)=R$. By the uniqueness of slope factorizations, we have only to prove $(a_1,b_i)=R$ for all $i$. Fix $i\in\N_{>0}$. Then, for all sufficiently large $n\in\N$, we have
\[
(z_i)=(a,b_i^n)=(a,b_i^{n+1})=(a_1a_2,b_i^{n+1})\subset (a_1,b_i)(a_2,b_i^n)\subset (a_1,b_i)(z_i,b_i^n)=(a_1,b_i)(z_i).
\]
Since $z_i\neq 0$, we have $R\subset (a_1,b_i)$, which implies the assertion. Finally, we prove $(a_3,b)\neq R$ for any non-unit $a_3\in R$ dividing $a_2$. By replacing $a_3$ by any factor of a slope factorization of $a_3$, we may assume that $a_3$ is pure. By the uniqueness of slope factorizations, $a_3$ divides $z_i$ for some $i$. Since $z_i|b_i^n$ for sufficiently large $n$, we also have $a_3|b_i^n$. Hence, we have $(a_3,b_i)\neq R$, in particular, $(a_3,b)\neq R$.
\end{proof}

\section{Variations of Gr\"obner basis argument}\label{sec:grobner}
In this section, we will systematically develop a basic theory of Gr\"obner basis over various rings. Our theory generalizes the basic theory of Gr\"obner basis over fields (\cite{CLO}, particularly, $\S~2$). As a first application, we will prove the continuity of connected components of flat families of rigid analytic spaces over annulus (Proposition~\ref{prop:connected}~(iii)). As a second application, we also prove the ramification compatibility of Scholl's fields of norms (Theorem~\ref{thm:normcompat}).

The idea to use a Gr\"obner basis argument to study Abbes-Saito's rigid spaces of positive characteristic is due to \cite[\S~1]{Xia}. Some results of this section, particularly \S\S~\ref{subsec:grobnerregular}, \ref{subsec:grobnerannulus}, are already proved in \cite[\S~1]{Xia}, however we do not use Xiao's results; We will work under a slightly stronger assumption and deduce stronger results, with much clearer and simpler proofs, than Xiao's.

Note that this section is independent from the other parts of this paper except \S\S~\ref{subsec:AS}, \ref{subsec:scholl}.

\begin{notation3.0.1}\label{notation:multi}
Throughout this section, we will use multi-index notation: We denote $\und{n}=(n_1,\dots,n_l)\in\N^l$ and $|\und{n}|:=n_1+\dots+n_l$, $\und{X}^{\und{n}}=X_1^{n_1}\dots X_l^{n_l}$ for variables $\und{X}=(X_1,\dots,X_l)$. We also denote by $\und{X}^{\N}$ the set of monic monomials $\{\und{X}^{\und{n}}|\und{n}\in\N^l\}$.

In this section, when we consider a topology on a ring, we will use a norm $|\cdot|$ rather than a valuation.
\end{notation3.0.1}

\subsection{Convergent power series}\label{subsec:convergent}
In this subsection, we consider rings of strictly convergent power series over the ring of rigid analytic functions over annulus, which play an analogous role to Tate algebra in the classical situation. We also gather basic definitions and facts on these rings for the rest of this section.

\begin{dfn}
Let $R$ be a ring. For $f=\sum_{\und{n}}a_{\und{n}}\und{X}^{\und{n}}\in R[[\und{X}]]$ with $a_{\und{n}}\in R$, we call each $a_{\und{n}}\und{X}^{\und{n}}$ a term of $f$. If $f=a_{\und{n}}\und{X}^{\und{n}}$ with $a_{\und{n}}\in R$ (resp. $a_{\und{n}}=1$), then we call $f$ a (resp. monic) monomial.
\end{dfn}

\begin{dfn}[{\cite[Definition~1, 1.4.1]{BGR}}]
Let $(R,|\cdot|)$ be a normed ring. We define Gauss norm on $R[\und{X}]$ by $|\sum_{\und{n}}{a_{\und{n}}X^{\und{n}}}|:=\sup_{\und{n}}{|a_{\und{n}}|}$. A formal power series $f=\sum_{\und{n}}{a_{\und{n}}X^{\und{n}}}\in R[[\und{X}]]$ is strictly convergent if $|a_{\und{n}}|\to 0$ as $|\und{n}|\to\infty$. We denote the ring of strictly convergent power series over $R$ by $R\langle\und{X}\rangle$. The above norm $|\cdot|$ can be uniquely extended to $|\cdot|:R\langle\und{X}\rangle\to\mathbb{R}_{\ge 0}$. Note that if $R$ is complete with respect to $|\cdot|$, then $R\langle\und{X}\rangle$ is also complete with respect to $|\cdot|$ (\cite[Proposition~3, 1.4.1]{BGR}).
\end{dfn}

We recall basic facts on rings of strictly convergent power series. Let $R$ be a complete normed ring, whose topology is equivalent to the $\mathfrak{a}$-adic topology for an ideal $\mathfrak{a}$. Then, $R\langle\und{X}\rangle$ is canonically identified with the $\mathfrak{a}$-adic Hausdorff completion of $R[\und{X}]$. We further assume that $R$ is Noetherian. Then, $R\langle\und{X}\rangle$ is $R$-flat. Moreover, for any ideal $\mathfrak{b}$ of $R$, we have a canonical isomorphism
\[
R\langle\und{X}\rangle\otimes_{R}(R/\mathfrak{b})\cong (R/\mathfrak{b})\langle\und{X}\rangle,
\]
where the RHS means the $\mathfrak{a}$-adic Hausdorff completion of $(R/\mathfrak{b})[\und{X}]$.

For a complete discrete valuation ring $\oo$ with $F=\mathrm{Frac}(\oo)$, we denote by $\oo\langle\und{X}\rangle$ (resp. $F\langle\und{X}\rangle$) the rings of convergent power series over $\oo$ (resp. $F$).

\begin{lem}\label{lem:flatconseq}
Assume that $R$ is a complete normed Noetherian ring, whose topology is equivalent to the $\mathfrak{a}$-adic topology for some ideal $\mathfrak{a}$ of $R$. Let $I\subset R\langle\und{X}\rangle$ be an ideal such that $R\langle\und{X}\rangle/I$ is $R$-flat. Then, $I$ is also $R$-flat. Moreover, for any ideal $J\subset R$, we have $I\cap J\cdot R\langle\und{X}\rangle=JI$. In particular, if $f\in I$ is divisible by $s\in R$ in $R\langle\und{X}\rangle$, then $f/s\in I$.
\end{lem}
We omit the proof since it is an easy exercise of flatness.

\begin{notation}\label{notation:power}
In the rest of this subsection, we fix the notation as follows: Let $\oo$ be a Cohen ring of a field $k$ of characteristic $p$ and we fix a norm $|\cdot|$ on $\oo$ corresponding to the $p$-adic valuation. We denote 
\[
R^+:=\oo[[S]]\subset R:=\oo((S))
\]
and for $r\in\Q_{>0}$, we define a norm
\[
|\cdot|_r:R\to\mathbb{R}_{\ge 0};\sum_{n\gg-\infty}a_nS^n\mapsto\sup_n{|a_n||p|^{rn}},
\]
which is multiplicative (\cite[Proposition 2.1.2]{pde}). Recall that we have defined in Notation~\ref{notation:ocpower}
\[
R^{\dagger,r}=\left\{\sum_{n\in\Z}a_nS^n\in\oo\{\{S\}\};|a_nS^n|_r\to 0\text{ as }n\to-\infty\right\}.
\]
Note that we may canonically identify $R^{\dagger,r}/pR^{\dagger,r}$ with $k((S))$. We can extend $|\cdot|_r$ to $|\cdot|_r:R^{\dagger,r}\to\mathbb{R}_{\ge 0}$ by $|\sum_n{a_nS^n}|_r:=\sup_n{|a_nS^n|_r}$. We define subrings of $R^{\dagger,r}$ by
\[
R^{\dagger,r}_0:=\{f\in R^{\dagger,r};|f|_r\le 1\},
\]
\[
\mathcal{R}^{\dagger,r}_0:=R^{\dagger,r}_0\cap R=\{f\in R;|f|_r\le 1\}.
\]
Note that for $a\in\N$ and $b\in\N_{>0}$, $|p^a/S^b|_r\le 1$ if and only if $a/b\ge r$. Also, note that $R^{\dagger,r}=R^{\dagger,r}_0[S^{-1}]$ since $|S|_r<1$. We may regard $R^{\dagger,r}$ as the ring of rigid analytic functions on the annulus $[p^r,1)$, whose values at the boundary $|S|=1$ are bounded by $1$.
\end{notation}

\begin{lem}\label{lem:conv}
\begin{enumerate}
\item The $R^+$-algebra $\mathcal{R}^{\dagger,r}_0$ is finitely generated.
\item The topologies of $R^{\dagger,r}_0$ defined by $|\cdot|_r$ and defined by the ideal $(p,S)$ are equivalent.
\item The rings $R^{\dagger,r}_0$ and $R^{\dagger,r}$ are complete with respect to $|\cdot|_r$, and $\mathcal{R}^{\dagger,r}_0$ is dense in $R^{\dagger,r}_0$.
\item The rings $\mathcal{R}^{\dagger,r}_0$, $R^{\dagger,r}_0$, and $R^{\dagger,r}$ are Noetherian integral domains.
\end{enumerate}
\end{lem}
\begin{proof}
Let $a,b\in\N$ denote the relatively prime integers such that $r=a/b$.
\begin{enumerate}
\item It is straightforward to check that $\mathcal{R}^{\dagger,r}_0$ is generated as an $R^+$-algebra by $p^{\lfloor rb'\rfloor}/S^{b'}$ for $b'\in\{0,\dots,b\}$.
\item For $n\in\N$, we have
\[
\sup\{|x|_r;x\in (p,S)^nR^{\dagger,r}_0\}\le \{\inf (|p|,|S|_r)\}^n
\]
and the RHS converges to 0 as $n\to\infty$. Hence, the $(p,S)$-adic topology of $R^{\dagger,r}_0$ is finer than the topology defined by $|\cdot|_r$. To prove that the topology of $R^{\dagger,r}_0$ defined by $|\cdot|_r$ is finer than the $(p,S)$-adic topology, it suffices to prove that
\[
\{x\in R^{\dagger,r}_0;|x|_r\le |(pS)^n|_r\}\subset (p,S)^nR^{\dagger,r}_0
\]
for all $n\in\N$. Let $x=\sum_{m\in\Z}{a_mS^m}\in\mathrm{LHS}$ with $a_m\in\oo$. Then, we have $|a_mS^{m-n}|_r\le |p^n|\le 1$. Hence, $x=S^n\sum_{m\in\Z}a_mS^{m-n}\in S^n\cdot R^{\dagger,r}_0$, which implies the assertion.
\item If $f=\sum_{n\in\Z}a_nS^n\in R^{\dagger,r}_0$ with $a_n\in\oo$, then $\{\sum_{n\ge -m}a_nS^n\}_{m\in\N}\subset\mathcal{R}^{\dagger,r}_0$ converges to $f$, which implies the last assertion. Since $R^{\dagger,r}_0$ is an open subring of $R^{\dagger,r}$, we have only to prove a completeness for $R^{\dagger,r}_0$. Let $\{f_m\}_{m\in\N}\subset R^{\dagger,r}_0$ be a sequence such that $|f_m|_r\to 0$ as $m\to\infty$. We have only to prove that the limit $\sum_m{f_m}$ exists in $R^{\dagger,r}_0$ with respect to $|\cdot|_r$. Write $f_m=\sum_{n\in\Z}a_n^{(m)}S^n$ with $a_n^{(m)}\in\oo$. For $n\in\Z$, we have
\[
|a_n^{(m)}|\le\frac{|f_m|_r}{|S^n|_r}=|p|^{-nr}|f_m|_r,
\]
hence, $|a_n^{(m)}|\to 0$ as $m\to\infty$. Moreover, $a_n:=\sum_{m\in\N}{a_n^{(m)}}\in\oo$ converges to $0$ as $n\to-\infty$. Hence, the formal Laurent series $f:=\sum_{n\in\Z}{a_nS^n}$ belongs to $\oo\{\{S\}\}$. Since
\[
|a_nS^n|_r\le\sup_{m\in\N}{|a_n^{(m)}S^n|_r}\le\sup_{m\in\N}|f_m|_r\le 1,
\]
we have $f\in R^{\dagger,r}_0$. For $m\in\N$, we have
\begin{align*}
|f-(f_0+\dots+f_m)|_r&\le\sup_n{|a_nS^n-(a_n^{(0)}+\dots+a_n^{(m)})S^n|_r}\le\sup_n\sup_{l>m}{|a_n^{(l)}S^n|_r}\\
&=\sup_{l>m}\sup_n{|a_n^{(l)}S^n|_r}\le\sup_{l>m}|f_l|_r
\end{align*}
and the last term converges to $0$ as $m\to\infty$, which implies $f=\sum_{m}f_m$.
\item It follows from (i), (ii) and (iii).
\end{enumerate}
\end{proof}

\begin{dfn}
We define $R^+\langle\und{X}\rangle$ as the $(p,S)$-adic Hausdorff completion of $R^+[\und{X}]$. We also define $R^{\dagger,r}_0\langle\und{X}\rangle$ and $R^{\dagger,r}\langle\und{X}\rangle$ as the rings of strictly convergent power series over $R^{\dagger,r}_0$ and $R^{\dagger,r}$ with respect to $|\cdot|_r$. We endow $R^{\dagger,r}_0\langle\und{X}\rangle$ and $R^{\dagger,r}\langle\und{X}\rangle$ with the topology defined by the norm $|\cdot|_r$. By Lemma~\ref{lem:conv}~(iii), $R^{\dagger,r}_0\langle\und{X}\rangle$ and $R^{\dagger,r}\langle\und{X}\rangle$ are complete. By Lemma~\ref{lem:conv}~(ii), $R^{\dagger,r}_0\langle\und{X}\rangle$ can be regarded as the $(p,S)$-adic Hausdorff completion of $R^{\dagger,r}_0[\und{X}]$, hence, $R^{\dagger,r}_0\langle\und{X}\rangle$ and $R^{\dagger,r}\langle\und{X}\rangle=R^{\dagger,r}_0\langle\und{X}\rangle[S^{-1}]$ are Noetherian integral domains by Lemma~\ref{lem:conv}~(iv). Also, we may regard $R^+\langle\und{X}\rangle$ as a subring of $R^{\dagger,r}_0\langle\und{X}\rangle$.
\end{dfn}

The following lemma seems to be used implicitly in \cite[\S~1]{Xia}. The proof is due to Liang Xiao.
\begin{lem}[Liang Xiao]\label{lem:Xiao}
The canonical map $R^+\langle\und{X}\rangle\to R^{\dagger,r}\langle\und{X}\rangle$ is flat.
\end{lem}
\begin{proof}
We may regard $R^{\dagger,r}_0\langle\und{X}\rangle$ as the $(p,S)$-adic Hausdorff completion of $R^+\langle\und{X}\rangle\otimes_{R^+}R^{\dagger,r}_0$. Since $\mathcal{R}^{\dagger,r}_0$ is dense in $R^{\dagger,r}_0$ by Lemma~\ref{lem:conv}~(iii), $R^{\dagger,r}_0\langle\und{X}\rangle$ can be regarded as the $(p,S)$-adic Hausdorff completion of $R^+\langle\und{X}\rangle\otimes_{R^+}\mathcal{R}^{\dagger,r}_0$, which is Noetherian by Lemma~\ref{lem:conv}~(i). Hence, a canonical map
\[
\alpha:R^+\langle\und{X}\rangle\otimes_{R^+}\mathcal{R}^{\dagger,r}_0\to R^{\dagger,r}_0\langle\und{X}\rangle
\]
is flat. Since $\mathcal{R}^{\dagger,r}_0[S^{-1}]=R$ and $R^{\dagger,r}_0\langle\und{X}\rangle[S^{-1}]=R^{\dagger,r}\langle\und{X}\rangle$, the canonical map $\alpha[S^{-1}]$ is also flat, which implies the assertion.
\end{proof}

Next, we consider prime ideals corresponding to good ``points'' of the open unit disc $R^+=\oo[[S]]$.

\begin{dfn}\label{dfn:Eisenstein}
An Eisenstein polynomial in $R^+$ is a polynomial in $\oo[S]$ of the form $P(S)=S^e+a_{e-1}S^{e-1}+\dots+a_0$ with $a_i\in\oo$ such that $p|a_i$ for all $i$ and $p^2\nmid a_0$. We call $\mathfrak{p}\in \mathrm{Spec}(R^+)$ an Eisenstein prime ideal if $\mathfrak{p}$ is generated by an Eisenstein polynomial $P(S)$. Then, we put $\deg{(\mathfrak{p})}:=e$ if $e\neq 0$ and $\deg{(\mathfrak{p})}:=\infty$ if $e=0$. Note that we may regard $\kappa(\mathfrak{p}):=R/\mathfrak{p}R$ as a complete discrete valuation field with integer ring $R^+/\mathfrak{p}R^+$. We denote by $\pi_{\mathfrak{p}}\in \oo_{\kappa(\mathfrak{p})}$ the image of $S$, which is a uniformizer of $\oo_{\kappa(\mathfrak{p})}$. Note that $\deg(\mathfrak{p})<\infty$ if and only if the characteristic of $R/\mathfrak{p}$ is zero. For simplicity, we write $\kappa(p)$ and $S$ instead of $\kappa((p))$ and $\pi_{\kappa((p))}$.
\end{dfn}

\begin{lem}\label{lem:idealchange}
Let $\mathfrak{p}$ and $\mathfrak{q}$ be Eisenstein prime ideals of $R^+$. For $x\in R^+$, if
\[
\inf{(v_{\kappa(\mathfrak{p})}(x\mod{\mathfrak{p}}),v_{\kappa(\mathfrak{q})}(x\mod{\mathfrak{q}}))}<\inf{(\deg{\mathfrak{p}},\deg{\mathfrak{q}})},
\]
then we have $v_{\kappa(\mathfrak{p})}(x\mod{\mathfrak{p}})=v_{\kappa(\mathfrak{q})}(x\mod{\mathfrak{q}})$.
\end{lem}
\begin{proof}
Let $x\in R^+$ and $i\in\N$ such that $0\le i<\deg{\mathfrak{p}}$. Then, we have the following equivalences:
\[
v_{\kappa(p)}(x\mod{p})=i\Leftrightarrow x\in (p,S^i)\setminus (p,S^{i+1})\Leftrightarrow x\in (\mathfrak{p},S^i)\setminus (\mathfrak{p},S^{i+1})\Leftrightarrow v_{\kappa(\mathfrak{p})}(x\mod{\mathfrak{p}})=i,
\]
where the second equivalence follows from the fact $(\mathfrak{p},S^i)=(p,S^i)$, and the other equivalences follow by definition. By replacing $\mathfrak{q}$ by $\mathfrak{p}$, we obtain similar equivalences. As a result, $v_{\kappa(\mathfrak{p})}(x\mod{\mathfrak{p}})=i\Leftrightarrow v_{\kappa(\mathfrak{q})}(x\mod{\mathfrak{q}})=i$ for $x\in R^+$ and $i<\inf(\deg(\mathfrak{p}),\deg(\mathfrak{q}))$, which implies the assertion.
\end{proof}

The ring $R^{\dagger,r}\langle\und{X}\rangle$ can be considered as a family of Tate algebras:

\begin{lem}\label{lem:redp}
Let $\mathfrak{p}$ be an Eisenstein prime ideal of $R^+$ with $e=\deg(\mathfrak{p})$. Let $r\in\Q_{>0}$ such that $1/e\le r$. Then, there exists a canonical isomorphism
\[
R^{\dagger,r}\langle\und{X}\rangle/\mathfrak{p}R^{\dagger,r}\langle\und{X}\rangle\to\kappa(\mathfrak{p})\langle\und{X}\rangle.
\]
In particular, $\mathfrak{p}R^{\dagger,r}\neq R^{\dagger,r}$.
\end{lem}
\begin{proof}
We will briefly recall a result in \cite{Laz}. Let $F$ be a complete discrete valuation field of mixed characteristic $(0,p)$. Recall that $L_F[0,r]$ is the ring of Laurent series with variable $S$ and coefficients in $F$, which converge in the annulus $|p|^r\le |S|<1$ (\cite[1.3]{Laz}). For $r'\in\Q_{>0}$, a polynomial $P\in F[S]$ is said to be $r'$-extremal if all zeroes $x$ of $P$ in $F^{\alg}$ satisfy $v(x)=r'$ (\cite[2.7']{Laz}). Let $r'\le r$ be positive rational numbers and $P\in F[S]$ an $r'$-extremal polynomial. Then, for $f\in L_F[0,r]$, there exist a unique $g\in L_F[0,r]$ and a unique polynomial $Q\in F[S]$ of degree $<\deg{P}$ such that $f=Pg+Q$, which is a special case of \cite[Lemme 2]{Laz}. Note that if $f\in F[S]$ with $\deg(f)<\deg(P)$, then we have $g=0$ and $Q=f$ by the uniqueness. In particular, the canonical map $\delta:F[S]/P\cdot F[S]\to L_F[0,r]/P\cdot L_F[0,r]$ is an isomorphism.

We prove the assertion. We can easily reduce to the case $\underline{X}=\phi$. That is, we have only to prove that the canonical map
\[
R^{\dagger,r}/\mathfrak{p}R^{\dagger,r}\to\kappa(\mathfrak{p})
\]
is an isomorphism. The assertion is trivial when $\mathfrak{p}=(p)$. Hence, we may assume $\mathfrak{p}\neq (p)$. Since $p$ is invertible in $\kappa(\mathfrak{p})$, $p$ is also invertible in $R^{\dagger,r}/\mathfrak{p}R^{\dagger,r}$. Hence, we have $R^{\dagger,r}/\mathfrak{p}R^{\dagger,r}=R^{\dagger,r}[p^{-1}]/\mathfrak{p}R^{\dagger,r}[p^{-1}]$. Note that $R^{\dagger,r}[p^{-1}]$ coincides with the above $L_F[0,r]$ with $F:=\mathrm{Frac}(\oo)$ by definition. Let $P\in \oo[S]$ be an Eisenstein polynomial, which generates $\mathfrak{p}$. Then, $P$ is $1/e$-extremal by a property of Eisenstein polynomial. Hence, the assertion follows from isomorphisms
\[
L_F[0,r]/\mathfrak{p}L_F[0,r]\cong F[S]/P\cdot F[S]\cong (\oo[S]/P\cdot\oo[S])[p^{-1}]\cong (R^+/\mathfrak{p})[p^{-1}]=\kappa(\mathfrak{p}).
\]
Here, the first equality is given by Lazard's isomorphism $\delta$ with $r'=1/e$.
\end{proof}

\subsection{Gr\"obner basis argument over complete regular local rings}\label{subsec:grobnerregular}
In this subsection, we will develop a basic theory of Gr\"obner basis over complete regular local rings $R$, which generalizing that over fields. This is done in \cite[\S~1.1]{Xia} when $R$ is a $1$-dimensional complete regular local ring of characteristic $p$. We assume the classical theory of Gr\"obner basis over fields and our basic reference is \cite{CLO}.

Recall that the classical theory of Gr\"obner basis on $F[\und{X}]$ for a field $F$ can be regarded as a multi-variable version of Euclidean division algorithm of the $1$-variable polynomial ring $F[X]$: To obtain an appropriate division algorithm in $F[\und{X}]$, we need to fix a ``monomial order'' of $F[\und{X}]$ to define a leading term, which plays an analogue role of the na\"ive degree function in the $1$-variable case. Hence, we should first define a notion of leading terms over the ring of convergent power series.
\begin{dfn}\label{dfn:order}
A monomial order $\succeq$ on a commutative monoid $(M,+)$ is an well-order such that if $\alpha\succeq \beta$, then $\alpha+\gamma\succeq\beta+\gamma$. When $\alpha\succeq\beta$ and $\alpha\neq\beta$, we denote by $\alpha\succ\beta$.

In the following, we consider only in the case where $M$ is isomorphic to $\N^l$. Moreover, the reader may assume that $\succ$ is a lexicographic order: The lexicographic order $\succeq_{\lex}$ on $\N^l$ is defined by $(a_1,\dots,a_l)\succ_{\lex} (a'_1,\dots,a'_l)$ if $a_1=a'_1,\dots,a_i=a'_i,a_{i+1}>a'_{i+1}$. A lexicographic order is a monomial order (\cite[Proposition 4, Chapter 2, \S 2]{CLO}).

For convenience, we define a monoid $M\cup\{\infty\}$ by $\alpha+\infty=\infty$ for any $\alpha\in M\cup\{\infty\}$. We extend any monomial order $\succeq$ on $M$ to $M\cup\{\infty\}$ by $\infty\succ \alpha$ for any $\alpha\in M$.
\end{dfn}

\begin{construction}\label{const:lead}
Let $R$ be a complete regular local ring of Krull dimension $d$ with a fixed regular system of parameters $\{s_1,\dotsc,s_d\}$. We put $R_i:=R/(s_1,\dotsc,s_{i})R$, which is also a regular local ring. We denote the image of $s_{i+1},\dots,s_d$ in $R_i$ by $s_{i+1},\dots,s_d$ again and we regard these as a fixed regular system of parameters. Let $v_{s_i}:R_i\twoheadrightarrow\N\cup\{\infty\}$ be the multiplicative valuation associated to the divisor $s_i=0$. For a non-zero $f\in R$ and $0\le i\le d$, we define a non-zero $f^{(i)}\in R_{i}$ inductively as follows: Put $f^{(0)}:=f$. If we have defined $f^{(i)}$, then we define $f^{(i+1)}$ as the image of $f^{(i)}/s_{i+1}^{v_{s_{i+1}}(f^{(i)})}$ in $R_{i+1}$, which is non-zero by definition. We put $\und{v}_R(f):=(v_{s_1}(f^{(0)}),v_{s_2}(f^{(1)}),\dots,v_{s_d}(f^{(d-1)}))\in \N^d$ and $\und{v}_R(0):=\infty$. Thus, we obtain a map $\und{v}_R:R\to\N^d\cup\{\infty\}$. We also apply this construction to each $R_i$. Note that we have a formula
\begin{equation}\label{eq:val}
\und{v}_R(f)=(v_{s_1}(f),\und{v}_{R_1}(f^{(1)})).
\end{equation}
Also, note that $\und{v}_R$ is multiplicative, i.e., $\und{v}_R(fg)=\und{v}_R(f)+\und{v}_R(g)$, which follows by induction on $d$ and using the formula.

Let $R\langle\und{X}\rangle$ be the $\mathfrak{m}_R$-adic Hausdorff completion of $R[\und{X}]$. We fix a monomial order $\succeq$ on $\und{X}^{\N}\cong\N^l$. For any non-zero $f=\sum_{\und{n}}a_{\und{n}}\und{X}^{\und{n}}\in R\langle\und{X}\rangle$ with $a_{\und{n}}\in R$, we define $\und{v}_R(f):=\inf_{\succeq_{\lex}}\und{v}_R(a_{\und{n}})$, where $\succeq_{\lex}$ is the lexicographic order on $\N^d$, and $\und{\deg}_R(f):=\inf_{\succeq}\{\und{n}\in\N^l;\und{v}_{R}(a_{\und{n}})=\und{v}_R(f)\}$. We put $\und{\deg}_R(0):=\infty$. Note that when $f\neq 0$, we have a formula
\begin{equation}\label{eq:deg}
\und{\deg}_R(f)=\und{\deg}_R(f^{(0)})=\und{\deg}_R(f^{(1)})=\dots=\und{\deg}_R(f^{(d)}),
\end{equation}
which follows from (\ref{eq:val}). Also, note that $\und{\deg}_R$ is multiplicative. Indeed, by the above formula (\ref{eq:deg}), we can reduce to the case where $R$ is a field, which follows from \cite[Lemma~8, Chapter~2]{CLO}. Thus, we obtain a multiplicative map
\[
\und{v}_R\times\und{\deg}_R:R\langle\und{X}\rangle\to (\N^d\times\N^l)\cup\{\infty\},
\]
where $\infty$ in the RHS denotes $(\infty,\infty)$. We endow $\N^d\times\N^l$ with a total order $\succeq$ by
\[
(\und{a},\und{n})\succeq (\und{a}',\und{n}')\text{ if }\und{a}\preceq_{\lex}\und{a}'\text{ or }\und{a}=\und{a}'\text{ and }\und{n}\succeq\und{n}'
\]
and extend it to $(\N^d\times\N^l)\cup\{\infty\}$ as in Definition~\ref{dfn:order}. Note that this order is an extension of the fixed order on $\N^l=\{0\}\times\dots\{0\}\times\N^l$. As in the classical notation, we also define
\[
\LT_R(f):=\und{s}^{\und{v}_R(f)}\und{X}^{\und{\deg}_R(f)}\text{ for }f\neq 0,\ \LT_R(0):=0,
\]
where $\und{s}=(s_1,\dots,s_d)$. Note that $\LT_R$ is also multiplicative by the multiplicativities of $\und{v}_R$ and $\und{\deg}_R$. Also, we have a formula
\begin{equation}\label{eq:LT}
\LT_R(f)\equiv\LT_{R_i}(f\mod{(s_1,\dots,s_i)})\mod{(s_1,\dots,s_i)},\ \forall f\in R\langle\und{X}\rangle.
\end{equation}
Indeed, if $s_i|f^{(i-1)}$ for some $i$, then both sides are zero. If $s_i\nmid f^{(i-1)}$ for all $i$, then the formula follows from (\ref{eq:val}) and (\ref{eq:deg}). The map $\LT_R$ takes values in the subset $\und{s}^{\N}\und{X}^{\N}\cup\{0\}$ of $R\langle\und{X}\rangle$. We identify $\und{s}^{\N}\und{X}^{\N}\cup\{0\}$ with $(\N^d\times\N^l)\cup\{\infty\}$ as a monoid and consider the total order $\succeq$ on $\und{s}^{\N}\und{X}^{\N}\cup\{0\}$.
\end{construction}

When $R$ is a field, the above definition coincides with the classical definitions as in \cite[\S~2]{CLO}.

\begin{rem}
$\LT$ stands for ``leading term'' with respect to a given monomial order in the classical case $d=0$. To define an appropriate $\LT$ in the case of $d>0$, we should consider a suitable order on the coefficient ring $R$, which is defined by using an ordered regular system of parameters as above. Our definition is compatible with d\'evissage, namely, compatible with a parameter-reducing maps $R\to R_1\to\dots\to R_d$. This property enables us to reduce everything about Gr\"obner basis to the classical case under assuming a certain ``flatness'' as we will see below. 
\end{rem}

In the rest of this subsection, let notation be as in Construction~\ref{const:lead}. In particular, we fix a monomial order $\succeq$ on $\und{X}^{\N}$.

\begin{dfn}
For $I$ be an ideal of $R\langle\und{X}\rangle$, we denote by $\LT_R(I)$ the ideal of $R\langle\und{X}\rangle$ generated by $\{\LT_R(f);f\in I\}$. Assume that $R\langle\und{X}\rangle/I$ is $R$-flat. We say that $f_1,\dots,f_s\in I$ forms a Gr\"obner basis if $(\LT_R(f_1),\dots,\LT_R(f_s))=\LT_R(I)$. Note that a Gr\"obner basis always exists since $R\langle\und{X}\rangle$ is Noetherian.
\end{dfn}

Note that for monomials $f,f_1,\dots,f_s\in R\langle\und{X}\rangle$, we have $f\in (f_1,\dots,f_s)$ if and only if $f$ is divisible by some $f_i$. Indeed, any term of $g\in (f_1,\dots,f_s)$ is divisible by some $f_i$, which implies the necessity.

\begin{notation}
Let $I$ be an ideal of $R\langle\und{X}\rangle$ such that $R\langle\und{X}\rangle/I$ is $R$-flat. We denote $I_i:=I/(s_1,\dotsc,s_i)I$. We may identify $R\langle\und{X}\rangle\otimes_R R_i$ and $I\otimes_RR_i$ with $R_i\langle\und{X}\rangle$ and $I_i$ respectively. Note that $R_i\langle\und{X}\rangle/I_i$ is $R_i$-flat.
\end{notation}

\begin{lem}\label{lem:grobred}
Let $I$ be an ideal of $R\langle\und{X}\rangle$ such that $R\langle\und{X}\rangle/I$ is $R$-flat. For $f_1,\dots,f_s\in I$, the following are equivalent:
\begin{enumerate}
\item $f_1,\dots,f_s$ forms a Gr\"obner basis of $I$;
\item The images of $f_1,\dots,f_s$ forms a Gr\"obner basis of $I_i\subset R_i\langle\und{X}\rangle$ for some $i$.
\end{enumerate}
Moreover, when $f_1,\dots,f_s$ is a Gr\"obner basis of $I$, $f_1,\dots,f_s$ generates $I$.
\end{lem}
\begin{proof}
We prove the first assertion. We proceed by induction on $d=\dim{R}$. When $d=0$, there is nothing to prove. Assume the assertion is true for dimension $<d$. By the induction hypothesis, we have only to prove the equivalence between (i) and (ii) with $i=1$.

We first prove $(i)\Rightarrow (ii)$. Let $\bar{f}\in I_1$ be a non-zero element and $f\in I$ a lift of $\bar{f}$. By assumption, we have $\LT_R(f_j)|\LT_R(f)$ for some $j$. Then, $\LT_{R_1}(f_j\mod{s_1})|\LT_{R_1}(\bar{f})$ by the formula~(\ref{eq:LT}).

We prove $(ii)\Rightarrow (i)$. Let $f\in I$ be a non-zero element. By Lemma~\ref{lem:flatconseq}, we have $f^{(1)}=f/s_1^{v_{s_1}(f)}\in I$. By assumption, we have $\LT_{R_1}(f_j\mod s_1)|\LT_{R_1}(f^{(1)}\mod{s_1})$ for some $j$. Since $\LT_{R_1}(f^{(1)}\mod{s_1})\neq 0$, $s_1$ does not divides $f_j$, i.e., $v_{s_1}(f_j)=0$. By the formulas~(\ref{eq:val}) and (\ref{eq:deg}), $\LT_R(f_j)$ divides $\LT_R(f^{(1)})$, hence, divides $\LT_R(f)$, which implies the assertion.

We prove the last assertion. By Nakayama's lemma and (ii) with $i=d$, the assertion is reduced to the case where $R$ is a field. In this case, the assertion follows from \cite[Corollary~2, \S~6]{CLO}.
\end{proof}

\begin{rem}
By Lemma~\ref{lem:grobred}, $f_1,\dots,f_s$ is a Gr\"obner basis of $I$ if and only if $f_1\mod{\mathfrak{m}_R},\dots,f_s\mod{\mathfrak{m}_R}$ is a Gr\"obner basis of $I/\mathfrak{m}_RI$. In particular, the definition of Gr\"obner basis does not depend on the choice of a regular system of parameters $\{s_1,\dots,s_d\}$.
\end{rem}

We can generalize the classical division algorithm, which is a basic tool in Gr\"obner basis argument.

\begin{prop}[Division algorithm]\label{prop:div}
Let $I$ be an ideal of $R\langle\und{X}\rangle$ such that $R\langle\und{X}\rangle/I$ is $R$-flat. Let $f_1,\dots,f_s\in I$ be a Gr\"obner basis of $I$. Then, for any non-zero $f\in R\langle\und{X}\rangle$, there exist $a_i,r\in R\langle\und{X}\rangle$ for all $i$ such that
\[
f=\sum_{1\le i\le s}a_if_i+r
\]
with $\LT_R(f)\succeq \LT_R(a_if_i)$ if $a_if_i\neq 0$, and any non-zero term of $r$ is not divisible by any $\und{X}^{\und{\deg}_R(f_i)}$. Moreover, such $r$ is uniquely determined (but $a_i$'s are not uniquely determined), and $f\in I$ if and only if $r=0$.
\end{prop}
\begin{proof}
When $d=0$, i.e, $R$ is a field, the assertion is well-known (see \cite[Theorem 3, Chapter 2, \S~3]{CLO} for example). We prove the first assertion by induction on $d=\dim{R}$. Assume that the assertion is true for dimension $<d$. We may assume $s_1\nmid f_i$ for all $i$. Indeed, by Lemma~\ref{lem:grobred}, $\{f_i;s_1\nmid f_i\}$ forms a Gr\"obner basis of $I$. Moreover, any $\LT_R(f_j)$ is divisible by some $\LT_R(f_i)$ with $s_1\nmid f_i$. Therefore, if there exists a desired expression $f=\sum_{i:s_1\nmid f_i}a_if_i+r$ with respect to $\{f_i;s_1\nmid f_i\}$, then this expression is also a desired expression with respect to $f_1,\dots,f_s$. First, we construct $g_n\in R\langle\und{X}\rangle$ by induction on $n\in\N$; For $h\in R\langle\und{X}\rangle$, we denote by $\bar{h}$ its image in $R_1\langle\und{X}\rangle$. Put $g_0:=f$. Assume that $g_n$ has been defined. Put $g'_n:=g_n/s_1^{v_{s_1}(g_n)}$. By applying the induction hypothesis to $I_1=(\bar{f}_1,\dots,\bar{f}_s)$, we have $\bar{a}_{i,n},\bar{r}_n\in R_1\langle\und{X}\rangle$ such that
\[
\bar{g}'_n=\sum_i{\bar{a}_{i,n}\bar{f}_i}+\bar{r}_n
\]
such that $\LT_{R_1}(\bar{g}'_n)\succeq \LT_{R_1}(\bar{a}_{i,n}\bar{f}_i)$ if $\bar{a}_{i,n}\bar{f}_i\neq 0$, and any non-zero term of $\bar{r}_n$ is not divisible by any $\und{X}^{\und{\deg}_{R_1}(\bar{f}_i)}$. We choose lifts $a_{i,n}$ and $r_n$ in $R\langle\und{X}\rangle$ of $\bar{a}_{i,n}$ and $\bar{r}_n$ respectively such that any non-zero term of $a_{i,n}$ and $r_n$ is not divisible by $s_1$. Then, we put $g_{n+1}:=g_n-s_1^{v_{s_1}(g_n)}(\sum_i{a_{i,n}f_i}+r_n)$. By construction, we have $v_{s_1}(g_{n+1})>v_{s_1}(g_n)$, hence, $\{g_n\}$ converges $s_1$-adically to zero. Moreover, $a_i:=\sum_n{s_1^{v_{s_1}(g_n)}a_{i,n}}$ and $r:=\sum_n{s_1^{v_{s_1}(g_n)}r_n}$ converge $s_1$-adically and we have $f=\sum_i{a_if_i}+r$. We will check that $a_i$ and $r$ satisfy the condition. Since $s_1\nmid f_i$ and any non-zero term of $r_n$ is not divisible by $s_1$, any non-zero term of $r$ is not divisible by $\und{X}^{\und{\deg}_R(f_i)}$ for all $i$. We have $v_{s_1}(f_i)=0$ by assumption and $v_{s_1}(a_i)\ge v_{s_1}(f)$ by definition. If $v_{s_1}(a_i)>v_{s_1}(f)$, then we have $\und{v}_{R}(f)\preceq_{\lex}\und{v}_R(a_if_i)$, hence, $\LT_R(f)\succeq \LT_R(a_if_i)$. If $v_{s_1}(a_i)=v_{s_1}(f)$, then we have $a_i^{(0)}\equiv a_{i,0}\mod{s_1}$, hence, $\und{v}_{R}(f)\preceq\und{v}_R(a_if_i)$ by the formulas~(\ref{eq:val}), (\ref{eq:deg}) and the choice of $\bar{a}_{i,0}$. In particular, $\LT_R(f)\succeq \LT_R(a_if_i)$. Thus, we obtain the first assertion.

We prove the rest of the assertion. We first prove the uniqueness of $r$. Let $f=\sum a_if_i+r=\sum a'_if_i+r'$ be expressions satisfying the conditions. Then, we have $r-r'\in I$, hence, $\LT_R(r-r)\in \LT_R(I)$. Therefore, $r-r'$ is divisible by $\LT_R(f_i)$ for some $i$. Since any non-zero term of $r-r'$ is not divisible by any $\LT_R(f_i)$, we must have $r=r'$. We prove the equivalence $r=0\Leftrightarrow f\in I$. We have only to prove the necessity. Since $r\in I$, we have $\LT_R(r)\in \LT_R(I)$. Hence, $\LT_R(r)$ is divisible by $\LT_R(f_i)$ for some $i$. Since any non-zero term of $r$ is divisible by $\und{X}^{\und{\deg}_R(f_i)}$, we must have $r=0$.
\end{proof}

\begin{dfn}
We call the above expression $f=\sum a_if_i+r$ a standard expression (of $f$) and call $r$ the reminder of $f$ (with respect to $f_1,\dots,f_s$). Note that standard expressions are additive and compatible with scalar multiplications: That is, if $f=\sum_ia_if_i+r$ and $g=\sum_ia'_if_i+r'$ are standard expressions, then $f+g=\sum_i(a_i+a'_i)f_i+r+r'$ is also a standard expression of $f+g$, and $\lambda f=\sum_i\lambda a_i f_i+\lambda r$ is a standard expression of $\lambda f$ for $\lambda\in R$ by the formulas~(\ref{eq:val}) and (\ref{eq:deg}). The reminder of $f$ depends only on the class $f\mod{I}$ by Proposition~\ref{prop:div} and the above additive property. Therefore, we may call $r$ the reminder of $f\mod{I}$.
\end{dfn}

As in the classical case, we have the following.
\begin{lem}\label{lem:remchar}
Let $I$ be an ideal of $R\langle\und{X}\rangle$ such that $R\langle\und{X}\rangle/I$ is $R$-flat. Let $f_1,\dots,f_s\in I$ be a Gr\"obner basis of $I$. Let $f\in R\langle\und{X}\rangle$ be a non-zero element. For $r\in R\langle\und{X}\rangle$, the following are equivalent:
\begin{enumerate}
\item $r$ is the reminder of $f$;
\item $f-r\in I$ and any non-zero term of $r$ is not divisible by $\und{X}^{\und{\deg}(f_i)}$ for all $i$.
\end{enumerate}
\end{lem}
\begin{proof}
Since the assertion $(i)\Rightarrow (ii)$ is trivial, we prove the converse. By applying the division algorithm to $f-r$, we have $f-r=\sum{a_if_i}$ such that $\LT_R(f)\succeq \LT_R(a_if_i)$ if $a_if_i\neq 0$. This is nothing but to say that $r$ is the reminder of $f$.
\end{proof}

\begin{cor}
Let notation be as in Lemma~\ref{lem:remchar}. We regard $f_1\mod s_1,\dots,f_s\mod s_1$ as a Gr\"obner basis of $I_1$. For $f\in R\langle\und{X}\rangle$ with $s_1\nmid f$, denote by $r$ and $r'$ the reminders of $f$ and $f\mod s_1$. Then, we have $r\mod{s_1}\equiv r'$.
\end{cor}

Finally, we give a concrete example of Gr\"obner basis, which will appear in \S~\ref{sec:ram}.
\begin{prop}\label{prop:concrete}
Let $I=(f_1,\dots,f_s)\subset R\langle\und{X}\rangle$ be an ideal. Assume that there exists relatively prime monic monomials $T_1,\dots,T_s$, and units $u_1,\dots,u_s\in R^{\times}$ such that $\LT_R(f_i)=u_iT_i$ for $1\le i\le s$. Then, we have the following:
\begin{enumerate}
\item $R\langle\und{X}\rangle/I$ is $R$-flat;
\item $f_1,\dots,f_s$ is a Gr\"obner basis of $I$;
\item $f_1,\dots,f_s$ is a regular sequence in $R\langle\und{X}\rangle$.
\end{enumerate}
\end{prop}
\begin{proof}
We may assume that $\LT_R(f_1),\dots,\LT_R(f_s)$ are relatively prime monic monomials by replacing $f_i$ by $f_i/u_i$. We first note that in the case of $d=0$, the assertion is basic: Indeed, the condition (i) is automatically satisfied. The condition (ii) directly follows from \cite[Theorem 3 and Proposition 4, \S 2]{CLO}. The condition (iii) follows by applying \cite[Proposition 15.15]{Eis} with $F=S=R[\und{X}]$ and $M=0$, $h_j=f_j$, where $F$, $S$ and $M$, $h_j$'s are as in the reference. We prove the assertion by induction on $s$. In the case of $s=1$, we have only to prove the condition (i). We proceed by induction on $d$. By the local criteria of flatness and the induction hypothesis, we have only to prove that the multiplication by $s_1$ on $R\langle\und{X}\rangle/I$ is injective. Let $f\in R\langle\und{X}\rangle$ such that $s_1f\in I$. Write $s_1f=f_1h$ for some $h\in R\langle\und{X}\rangle$. By taking $v_{s_1}$, we have $s_1|h$ since $s_1\nmid f_1$. This implies $f_1|f$, i.e., $f\in I$. Thus, we finish the case of $s=1$. We assume that the assertion is true when the cardinality of $f_i$'s is $<s$. We proceed by induction on $d$. The case of $d=0$ has been done as above. Assume that the assertion is true for dimension $<d$. For $h\in R\langle\und{}X\rangle$, denote by $\bar{h}$ its image in $R_1\langle\und{X}\rangle$. By assumption, $s_1\nmid f_i$ for all $i$, hence, we can apply the induction hypothesis to $\bar{f}_1,\dots,\bar{f}_s\in I_1:=(\bar{f}_1,\dots,\bar{f}_s)\subset R\langle\und{X}\rangle$ by the formula (\ref{eq:LT}). Hence, $R_1\langle\und{X}\rangle/I_1$ is $R_1$-flat, $\bar{f}_1,\dots,\bar{f}_s$ are Gr\"obner basis of $I_1$, and $\bar{f}_1,\dots,\bar{f}_s$ is a regular sequence in $R_1\langle\und{X}\rangle$. The condition (ii) follows from Lemma~\ref{lem:grobred}. Then, we check the condition (i). By the local criteria of flatness, we have only to prove that the multiplication by $s_1$ on $R\langle\und{X}\rangle/I$ is injective. It suffices to prove $I\cap s_1\cdot R\langle\und{X}\rangle\subset s_1I$. Denote by $C_{\bullet}$ and $\bar{C}_{\bullet}$ Koszul complexes for $\{f_1,\dotsc,f_s\}$ and $\{\bar{f}_1,\dots,\bar{f}_s\}$ (\cite[18.D]{Mat}). Then, we have $\bar{C}_i=C_i/s_1C_i$ for $i\ge 1$ by definition and $\bar{C}_{\bullet}$ is exact since $\bar{f}_1,\dots,\bar{f}_s$ is a regular sequence. We also have a morphism of complexes $C_{\bullet}\to\bar{C}_{\bullet}$, whose the first few terms are
\[\xymatrix{
\hdots\ar[r]&C_2\ar[r]^{d_2}\ar[d]&C_1\ar@{->>}[r]^{d_1}\ar[d]&I\ar[r]\ar[d]&0\\
\hdots\ar[r]&\bar{C}_2\ar[r]^{\bar{d}_2}&\bar{C}_1\ar[r]^{\bar{d}_1}&I_1\ar[r]&0.
}\]
Let $f\in I\cap s_1\cdot R\langle\und{X}\rangle$. Then, there exists $a\in C_1$ such that $d_1(a)=f$. Since $\bar{d}_1(\bar{a})\equiv 0\mod{s_1}$, there exists $\bar{b}\in C_2$ such that $\bar{d}_2(\bar{b})=\bar{a}$. Let $b\in C_2$ be a lift of $\bar{b}$. Then, there exists $a'\in C_1$ such that $a-d_2(b)=s_1a'$. Therefore, we have $f=d_1(a-d_2(b))=s_1d_1(a')\in s_1I$. Thus, the condition (i) is proved. Finally, we check the condition (iii). We have only to prove that for $1\le i \le s$, if $f_if\in (f_1,\dots,f_{i-1})$ for some $f\in R\langle\und{X}\rangle$, then we have $f\in (f_1,\dots,f_{i-1})$. Note that $f_1,\dots,f_{i-1}$ is a Gr\"obner basis of $(f_1,\dots,f_{i-1})$ by induction hypothesis. Let $f=\sum_{1\le j< i} a_{j}f_{j}+r$ be a standard expression of $f$ with respect to $f_1,\dots,f_{i-1}$. It suffices to prove $r=0$. We suppose the contrary and we will deduce a contradiction. Any non-zero term of $r$ is not divisible by $\LT_R(f_j)$ for any $1\le j<i$, in particular, we have $\LT_R(f_j)\nmid \LT_R(r)$. By assumption, $f_if=f_i(\sum_{1\le j<i} a_jf_j)+f_ir\in (f_1,\dots,f_{i-1})$, hence, we have $f_ir\in (f_1,\dots,f_{i-1})$. In particular, there exists $1\le j <i$ such that $\LT_R(f_j)|\LT_R(f_ir)$. Since $\LT_R(f_i)$ and $\LT_R(f_j)$ are relatively prime, we have $\LT_R(f_j)|\LT_R(r)$, which is a contradiction. Thus, we obtain the assertion (iii).
\end{proof}

A remarkable feature of the reminder is the compatibility with the quotient norms:
\begin{lem}\label{lem:normcompat}
Let $I$ be an ideal of $R\langle\und{X}\rangle$ such that $R\langle\und{X}\rangle/I$ is $R$-flat. Let $f_1,\dots,f_s\in I$ be a Gr\"obner basis of $I$. Let $|\cdot|:R\to \mathbb{R}_{\ge 0}$ be any non-archimedean norm satisfying $|R|\le 1$ and $|\mathfrak{m}_R|<1$. We extend $|\cdot|$ to a norm on $R\langle\und{X}\rangle$ by $|\sum_{\und{n}}a_{\und{n}}\und{X}^{\und{n}}|:=\sup_{\und{n}}|a_{\und{n}}|<\infty$. If we denote by $|\cdot|_{\qt}:R\langle\und{X}\rangle/I\to\mathbb{R}_{\ge 0}$ the quotient norm of $|\cdot|$, then the reminder $r$ of $f\in R\langle\und{X}\rangle$ achieves the quotient norm of $f\mod{I}$, i.e.,
\[
|r|=|f\mod{I}|_{\qt}.
\]
\end{lem}
\begin{proof}
Let $f=\sum \lambda_{\und{n}}\und{X}^{\und{n}}$ with $\lambda_{\und{n}}\in R$. Let $\und{X}^{\und{n}}=\sum a_{\und{n},i}f_i+r_{\und{n}}$ be a standard expression of $\und{X}^{\und{n}}$. Let $a_i:=\sum_{\und{n}}\lambda_{\und{n}}a_{\und{n},i}$ and $r:=\sum_{\und{n}} \lambda_{\und{n}}r_{\und{n}}$, which converge since $\lambda_{\und{n}}\to 0$ as $|\und{n}|\to\infty$. Then, $f=\sum a_if_i+r$ is a standard expression of $f$ by Lemma~\ref{lem:remchar}. We have $|a_if_i|\le |a_i|\le \sup_{\und{n}}|\lambda_{\und{n}}a_{\und{n},i}|\le \sup_{\und{n}}|\lambda_{\und{n}}|=|f|$. Hence, we have $|r|\le |f|$. Since the reminder depend only on the class $f\mod{I}$, we have
\[
|f\mod{I}|_{\qt}=\inf_{g\in I}|f+g|\ge |r|\ge |f\mod{I}|_{\qt},
\]
which implies the assertion.
\end{proof}

\subsection{Gr\"obner basis argument over annulus}\label{subsec:grobnerannulus}
In this subsection, we will give an analogue of a Gr\"obner basis argument over rings of overconvergent power series. In this subsection, we use the notation as in \S~\ref{subsec:convergent} and \S~\ref{subsec:grobnerregular}. Also, let notation be as follows:

\begin{notation}\label{notation:annulus}
Let $\mathcal{O}$, $R^+$, and $R$ be as in Notation~\ref{notation:power}. Fix $\{p,S\}$ as a regular system of parameter of $R^+$. Let $I\subset R^+\langle\und{X}\rangle$ be an ideal such that $R^+\langle\und{X}\rangle/I$ is $R^+$-flat. For $r\in\Q_{>0}$, we endow $R^{\dagger,r}$ with the topology defined by the norm $|\cdot|_r$. We denote
\[
A:=R^+\langle\und{X}\rangle/I,\ I^{\dagger,r}:=I\otimes_{R^+\langle\und{X}\rangle}R^{\dagger,r}\langle\und{X}\rangle,\ A^{\dagger,r}:=A\otimes_{R\langle\und{X}\rangle}R^{\dagger,r}\langle\und{X}\rangle.
\]
(When $I=0$, $R^{\dagger,r}\langle\und{X}\rangle$ is denoted by $R\langle\und{X}\rangle^{\dagger,r}$ in this notation. However, we use this notation for simplicity.) Since $R^+\langle\und{X}\rangle\to R^{\dagger,r}\langle\und{X}\rangle$ is flat (Lemma~\ref{lem:Xiao}), we may identify $I^{\dagger,r}$ and $A^{\dagger,r}$ with $I\cdot R^{\dagger,r}\langle\und{X}\rangle$ and $R^{\dagger,r}\langle\und{X}\rangle/I^{\dagger,r}$. Since $R^+$ is an integral domain, $A$, hence, $A^{\dagger,r}$ are $R^+$-torsion free by flatness.

Denote $|\cdot|_{r,\qt}:A^{\dagger,r}\to\mathbb{R}_{\ge 0}$ by the quotient norm of $|\cdot|_{r}$. Note that $A^{\dagger,r}$ is complete with respect to $|\cdot|_{r,\qt}$ by \cite[Proposotion~3, 1.1.7]{BGR}.
\end{notation}

\begin{lem}[{cf. \cite[Lemma~1.1.22]{Xia}}]\label{lem:remannulus}
Let $f_1,\dots,f_s\in I$ be a Gr\"obner basis of $I$. For $f\in R^{\dagger,r}\langle\und{X}\rangle$, there exists a unique $\mathfrak{r}\in R^{\dagger,r}\langle\und{X}\rangle$ such that $f-\mathfrak{r}\in I^{\dagger,r}$ and any non-zero term of $\mathfrak{r}$ is not divisible by $\und{X}^{\und{\deg}_R(f_i)}$. Moreover, we have $|\mathfrak{r}|_{r'}=|f|_{r',\qt}$ for $r'\in\Q\cap (0,r]$, and $\mathfrak{r}=0$ if and only if $f\in I^{\dagger,r}$. We call $\mathfrak{r}$ the reminder of $f$ (with respect to $f_1,\dots,f_s$).
\end{lem}
\begin{proof}
We first construct $\mathfrak{r}$. Let $f=\sum_{\und{n}}{\lambda_{\und{n}}\und{X}^{\und{n}}}\in R^{\dagger,r}\langle\und{X}\rangle$ with $\lambda_{\und{n}}\in R^{\dagger,r}$. Let
\[
\und{X}^{\und{n}}=\sum_i{a_{{\und{n}},i}f_i}+r_{\und{n}}
\]
be the standard expression of $\und{X}^{\und{n}}$ in $R^+\langle\und{X}\rangle$ with respect to $f_1,\dots,f_s$. Since $\lambda_{\und{n}}\to 0$ as $|\und{n}|\to\infty$, the series
\[
a_i:=\sum_{\und{n}}{\lambda_{\und{n}}a_{{\und{n}},i}},\ \mathfrak{r}:=\sum_{\und{n}}{\lambda_{\und{n}}r_{\und{n}}}
\]
converge in $R^{\dagger,r}\langle\und{X}\rangle$ with respect to the topology defined by $|\cdot|_r$. Then, we have
\begin{equation}\label{eq:normdagger}
|\mathfrak{r}|_{r'}\le\sup_{\und{n}}{|\lambda_{\und{n}}r_{\und{n}}|_{r'}}\le\sup_{\und{n}}{|\lambda_{\und{n}}|_{r'}}=|f|_{r'}.
\end{equation}
Obviously, any non-zero term of $\mathfrak{r}$ is not divisible by any $\und{X}^{\und{\deg}_R(f_i)}$ and we have $f-\mathfrak{r}=\sum_i{a_if_i}\in I^{\dagger,r}$.

We prove the uniqueness of $\mathfrak{r}$. We suppose the contrary and deduce a contradiction. Let $\mathfrak{r}'\in R^{\dagger,r}\langle\und{X}\rangle$ be an element such that $f-\mathfrak{r}'\in I^{\dagger,r}$ and any non-zero term of $\mathfrak{r}'$ is not divisible by any $\und{X}^{\und{\deg}_R(f_i)}$. We choose $m\in\N$ such that $\delta:=S^m(\mathfrak{r}-\mathfrak{r}')$ belongs to $I^{\dagger,r}_0:=I\otimes_{R^+\langle\und{X}\rangle}R^{\dagger,r}_0\langle\und{X}\rangle$. If we write $\delta=p^n\delta'$ such that $\delta'\in R^{\dagger,r}_0\langle\und{X}\rangle$ is not divisible by $p$ in $R^{\dagger,r}_0\langle\und{X}\rangle$, then we have $\delta'\in I^{\dagger,r}_0$ by Lemma~\ref{lem:flatconseq}. We may identify $I^{\dagger,r}_0/pI^{\dagger,r}_0$ with $I/pI$ by Lemma~\ref{lem:redp}. We denote $\bar{\delta}':=\delta'\mod{pI^{\dagger,r}_0}\in I/pI$. We also denote $R_1^+:=R^+/pR^+$, which is a complete discrete valuation ring with uniformizer $S$. Then, any non-zero term of $\bar{\delta}'$ is not divisible by $\und{X}^{\und{\deg}_{R^+_1}(f_i\mod{p})}$. Hence, $\bar{\delta}'$ is the reminder of $0$ with respect to $f_1\mod{p},\dots,f_s\mod{p}$ in $R_1\langle\und{X}\rangle$. By Lemma~\ref{lem:remchar}, $\bar{\delta}'=0$, i.e., $\delta'\in\mod{pI^{\dagger,r}_0}$, which contradicts to $p\nmid \delta'$.

We prove $f=I^{\dagger,r}\Leftrightarrow \mathfrak{r}=0$. If $f\in I^{\dagger,r}$, then $0$ satisfies the required property for the reminder. Hence, $\mathfrak{r}=0$ by the uniqueness of the reminder. If $\mathfrak{r}=0$, then $f\in I^{\dagger,r}$ by definition.

We prove $|\mathfrak{r}|_{r'}=|f\mod{I^{\dagger,r}}|_{r',\qt}$. Let $\alpha\in I^{\dagger,r}$. Since $\mathfrak{r}$ satisfies the required condition for the reminder of $f+\alpha$, the reminder $f+\alpha$ is equal to $\mathfrak{r}$ by the uniqueness of the reminder. In particular, the reminder depends only on the class $f\mod{I^{\dagger,r}}$. Hence, the assertion follows from
\[
|f\mod{I^{\dagger,r}}|_{r',\qt}=\inf_{\alpha\in I^{\dagger,r}}{|\mathfrak{r}+\alpha|_{r'}}\ge |\mathfrak{r}|_{r'}\ge |f\mod{I^{\dagger,r}}|_{r',\qt},
\]
where the first equality follows from (\ref{eq:normdagger}) and the second inequality follows by definition.
\end{proof}

The following is an immediate consequence of the above lemma.
\begin{lem}\label{lem:remanu}
Let $f_1,\dots,f_s$ be a Gr\"obner basis of $I$. Let $f,g\in R^{\dagger,r}\langle\und{X}\rangle$ and $\mathfrak{r}$, $\mathfrak{r}'$ its reminders with respect to $f_1,\dots,f_s$. Then, we have the following:
\begin{enumerate}
\item The reminder of $f+g$ is equal to $\mathfrak{r}+\mathfrak{r}'$.
\item The reminder $\mathfrak{r}$ depends only on $f\mod{I^{\dagger,r}}$; One may call the reminder of $f$ the reminder of $f\mod{I^{\dagger,r}}$.
\item For $\lambda\in R^{\dagger,r}$, the reminder of $\lambda f$ is equal to $\lambda \mathfrak{r}$. Moreover, if $f\mod{I^{\dagger,r}}$ is divisible by $\lambda\in R^{\dagger,r}$, then $\mathfrak{r}$ is also divisible by $\lambda$.
\end{enumerate}
\end{lem}

\begin{cor}\label{cor:intersection}
Let $\mathfrak{a}\subsetneq R^{\dagger,r}$ be a principal ideal. Then, we have $\cap_{n\in\N}\mathfrak{a}^n\cdot A^{\dagger,r}=0$.
\end{cor}
\begin{proof}
Fix a Gr\"obner basis $f_1,\dots,f_s$ of $I$. Let $f\in\cap_{n\in\N}\mathfrak{a}^n\cdot A^{\dagger,r}$ and $\mathfrak{r}$ the reminder of $f$ with respect to $f_1,\dots,f_s$. By Lemma~\ref{lem:remanu}~(iii) and the assumption, $\mathfrak{r}\in\cap_{n\in\N}\mathfrak{a}^n=0$.
\end{proof}

\begin{rem}
One can prove that $R^{\dagger,r}$ is a principal ideal domain by using \cite[Proposition~2.6.5]{Doc}. We do not use this fact in this paper.
\end{rem}

\subsection{Continuity of connected components for families of affinoids}\label{subsec:continuity}
In this subsection, we will apply the previous results to prove a continuity of connected components of fibers of families of affinoids.

\begin{lem}\label{lem:conn}
Let $f:R\to S$ be a morphism of Noetherian rings and $\Idem(T)$ denote the set of idempotents for a ring $T$. If a canonical map $f_*:\Idem(R)\to \Idem(S)$ is surjective and $f_*^{-1}(\{0\})=\{0\}$, then $f^*:\pi_0^{\Zar}(S)\to\pi_0^{\Zar}(R)$ is bijective.
\end{lem}
\begin{proof}
We first recall a basic fact on commutative algebras: Let $A$ be a ring. Then, finite partitions of $\Spec(A)$ into non-empty open subspaces as a topological space correspond to finite sets of non-zero idempotents $e_1,\dots,e_n$ of $A$ such that $\sum_i{e_i}=1$ and $e_ie_j=0$ for all $i\neq j$. Precisely, $e_1,\dots,e_n$ corresponds to $\Spec(Ae_1)\sqcup\dots\sqcup\Spec(Ae_n)$ (for details, see \cite[Proposition 15, II, \S 4, $n^o$ 3]{Bou}).

Decompose $\Spec(R)$ into the connected components and choose the corresponding idempotents $e_1,\dots,e_n$ by the above fact. Since the non-zero idempotents $f(e_1),\dots,f(e_n)$ satisfies $\sum_{1\le i\le n}f(e_i)=1$ and $f(e_i)f(e_j)=0$ for $i\neq j$, we obtain a finite partition $\Spec(S)=\Spec(Sf(e_1))\sqcup\dots\sqcup\Spec(Sf(e_n))$. Hence, we have only to prove that $\Spec(Sf(e_i))$ is connected for all $1\le i\le n$. Let $e'\in\Idem(Sf(e_i))$. By regarding $e'\in\Idem(S)$, there exists $x\in \Idem(R)$ such that $e'=f(x)$. Since $xe_i\in\Idem(Re_i)$ and $\Spec(Re_i)$ is connected by definition, we have $xe_i=0$ or $e_i$. Since we have $e'=e'f(e_i)=f(x)f(e_i)=f(xe_i)$, we have $e'=0$ or $f(e_i)$. Hence, $Sf(e_i)$ has only trivial idempotents, which implies the assertion.
\end{proof}

\begin{notation}
In the rest of this subsection, unless otherwise is mentioned, let notation be as in Notation~\ref{notation:annulus} and Definition~\ref{dfn:Eisenstein}. For an Eisenstein prime ideal $\mathfrak{p}$ of $R^+$, fix a norm $|\cdot|_{\mathfrak{p}}$ of complete discrete valuation field $\kappa(\mathfrak{p})$ and denote
\[
A_{\kappa(\mathfrak{p})}:=(A/\mathfrak{p}A)[S^{-1}].
\]
We identify $R^+\langle\und{X}\rangle/\mathfrak{p}R^+\langle\und{X}\rangle$ with $\oo_{\kappa(\mathfrak{p})}\langle\und{X}\rangle$ and denote Gauss norm on $\kappa(\mathfrak{p})\langle\und{X}\rangle$ by $|\cdot|_{\mathfrak{p}}$. We also denote the quotient (resp. spectral) norm of $|\cdot|_{\mathfrak{p}}$ on $A/\mathfrak{p}A$ and $A_{\kappa(\mathfrak{p})}$ by $|\cdot|_{\mathfrak{p},\qt}$ (resp. $|\cdot|_{\mathfrak{p},\spe}$). For simplicity, we also denote $|f\mod{I/\mathfrak{p}I}|_{\mathfrak{p},\qt}$ (resp. $|f\mod{I/\mathfrak{p}I}|_{\mathfrak{p},\qt}$) by $|f|_{\mathfrak{p},\qt}$ (resp. $|f|_{\mathfrak{p},\qt}$) for $f\in \kappa(\mathfrak{p})\langle\und{X}\rangle$.

For $f=\sum_{\und{n}}a_{\und{n}}\und{X}^{\und{n}}\in\oo_{\kappa(\mathfrak{p})}\langle\und{X}\rangle$ with non-zero $a_{\und{n}}\in\oo_{\kappa(\mathfrak{p})}$, let $\wtil{a}_{\und{n}}\in R^+$ be a lift of $a_{\und{n}}$. Then, $\wtil{f}:=\sum_{\und{n}}\wtil{a}_{\und{n}}\und{X}^{\und{n}}\in R^+\langle\und{X}\rangle$ is called a minimal lift of $f$.
\end{notation}

We may apply Construction~\ref{const:lead} to $R=\oo_{\kappa(\mathfrak{p})}$ and $s_1=\pi_{\mathfrak{p}}$ with the same monomial order $\succeq$ for $\oo[[S]]$. Let $f_1,\dots,f_s$ be a Gr\"obner basis of $I$. Then, the images of $f_i$'s in $R^+/\mathfrak{m}_{R^+}[\und{X}]$ is a Gr\"obner basis by Lemma~\ref{lem:grobred}. Hence, the images of $f_i$'s in $\oo_{\kappa(\mathfrak{p})}\langle\und{X}\rangle$ is a Gr\"obner basis of $I/\mathfrak{p}I$ by Lemma~\ref{lem:grobred} again. In particular, if $\mathfrak{r}$ is the reminder of $f\in R^+\langle\und{X}\rangle$ with respect to $f_1,\dots,f_s$, then the image of $\mathfrak{r}$ in $\oo_{\kappa(\mathfrak{p})}\langle\und{X}\rangle$ is the reminder of $f\mod{\mathfrak{p}}$ with respect to $f_1\mod{\mathfrak{p}},\dots,f_s\mod{\mathfrak{p}}$.

By using our Gr\"obner basis argument, Lemma~\ref{lem:idealchange} can be converted into the following form:
\begin{lem}\label{lem:normtransfer} 
Let $c\in\N$ and let $\mathfrak{p}$, $\mathfrak{q}$ be Eisenstein prime ideals of $R^+$ such that $c<\inf{(\deg{\mathfrak{p}},\deg{\mathfrak{q}})}$. Assume that for $n\in\N$, we have
\[
|f^n|_{\mathfrak{p},\qt}\ge  |\pi_{\mathfrak{p}}|_{\mathfrak{p}}^c|f|^n_{\mathfrak{p},\qt},\ \forall f\in A_{\kappa(\mathfrak{p})}.
\]
Then, we have
\[
|f^n|_{\mathfrak{q},\qt}\ge  |\pi_{\mathfrak{q}}|_{\mathfrak{q}}^c|f|^n_{\mathfrak{q},\qt},\ \forall f\in A_{\kappa(\mathfrak{q})}.
\]
\end{lem}
\begin{proof}
We fix a Gr\"obner basis $f_1,\dots,f_s$ of $I$. We may regard $f_i\mod{\mathfrak{p}}$'s (resp. $f_i\mod{\mathfrak{q}}$'s) as a Gr\"obner basis of $I/\mathfrak{p}I$ (resp. $I/\mathfrak{q}I$). To prove the assertion, we may assume that $f\in A/\mathfrak{q}A$. Let $\mathfrak{r}\in \oo_{\kappa(\mathfrak{q})}\langle\und{X}\rangle$ be the reminder of $f$. We have $|f|_{\mathfrak{q},\qt}=|\mathfrak{r}|_{\mathfrak{q}}=|\pi_{\mathfrak{q}}|_{\mathfrak{q}}^m$ for some $m\in\N$. To prove the assertion, we may assume $|f|_{\mathfrak{q},\qt}=|\mathfrak{r}|_{\mathfrak{q}}=1$ by replacing $f$, $\mathfrak{r}$ by $f/\pi_{\mathfrak{q}}^m$, $\mathfrak{r}/\pi_{\mathfrak{q}}^m$.

Let $\wtil{\mathfrak{r}}\in R^+\langle\und{X}\rangle$ be a minimal lift of $\mathfrak{r}$ and let $\wtil{f}\in A$ denote the image of $\wtil{\mathfrak{r}}$. Denote by $\mathfrak{r}_n\in R^+\langle\und{X}\rangle$ the reminder of $\wtil{f}^n$. Then, we have
\[
|\mathfrak{r}_n\mod{\mathfrak{p}}|_{\mathfrak{p}}=|\wtil{f}^n\mod{\mathfrak{p}}|_{\mathfrak{p},\qt}\ge |\pi_{\mathfrak{p}}|^c_{\mathfrak{p}}|\wtil{f}\mod{\mathfrak{p}}|^n_{\mathfrak{p},\qt}
\]
by Lemma~\ref{lem:normcompat} and assumption. By $|\mathfrak{r}|_{\mathfrak{q}}=1$, the coefficient of some $\und{X}^{\und{n}}$ in $\mathfrak{r}$ belongs to $\oo_{\kappa(\mathfrak{p})}^{\times}$. Therefore, the coefficient of $\und{X}^{\und{n}}$ in $\wtil{\mathfrak{r}}$, hence, in $\wtil{\mathfrak{r}}\mod{\mathfrak{p}}$ are units. Therefore, we have
\[
|\wtil{f}\mod{\mathfrak{p}}|_{\mathfrak{p},\qt}=|\wtil{\mathfrak{r}}\mod{\mathfrak{p}}|_{\mathfrak{p}}=1,
\]
hence, $|\mathfrak{r}_n\mod{\mathfrak{p}}|_{\mathfrak{p}}\ge |\pi_{\mathfrak{p}}|^c_{\mathfrak{p}}$. By applying Lemma~\ref{lem:idealchange} to the coefficient $\lambda$ of $\mathfrak{r}_n$ such that $|\lambda\mod{\mathfrak{p}}|_{\mathfrak{p}}\ge |\pi_{\mathfrak{p}}|^c$, we have $|\mathfrak{r}_n\mod{\mathfrak{q}}|_{\mathfrak{q}}\ge |\pi_{\mathfrak{q}}|_{\mathfrak{q}}^c$. Since $\mathfrak{r}_n\mod{\mathfrak{q}}$ is the reminder of $f^n$, we have $|f^n|_{\mathfrak{q},\qt}=|\mathfrak{r}_n\mod{\mathfrak{q}}|_{\mathfrak{q}}\ge |\pi_{\mathfrak{q}}|_{\mathfrak{q}}^c$ by Lemma~\ref{lem:normcompat}, which implies the assertion.
\end{proof}

The following lemma can be considered as an analogue of Hensel's lemma.
\begin{lem}[{cf. \cite[Theorem~1.2.11]{Xia}}]\label{lem:connected}
Assume that there exists $c\in\mathbb{R}_{\ge 0}$ such that
\[
|\cdot|_{\mathfrak{p},\spe}\ge |\pi_{\mathfrak{p}}|^c|\cdot|_{\mathfrak{p},\qt}\text{ on }A_{\kappa(\mathfrak{p})}.
\]
Then, for all $\mathfrak{r}\in\Q_{>0}\cap [1/\deg{\mathfrak{p}},1/2c)$, there exists a canonical bijection
\[
\pi_0^{\Zar}(A_{\kappa(\mathfrak{p})})\to\pi_0^{\Zar}(A^{\dagger,r}).
\]
\end{lem}
\begin{proof}
Replacing $c$ by $\lfloor c\rfloor$, we may assume $c\in\N$. Denote by $\alpha$ the canonical map $\Idem(A^{\dagger,r})\to \Idem(A_{\kappa(\mathfrak{p})})$. By Lemma~\ref{lem:conn}, we have only to prove that we have $\alpha^{-1}(\{0\})=\{0\}$ and $\alpha$ is surjective. Let $e\in \Idem(A^{\dagger,r})$ such that $\alpha(e)=0$. Then, we have $e\in \mathfrak{p}\cdot A^{\dagger,r}$. Since $e=e^n$, we have $e\in \cap_{n\in\N}\mathfrak{p}^n\cdot A^{\dagger,r}=0$ by Corollary~\ref{cor:intersection}, which implies the first assertion. We will prove the surjectivity of $\alpha$. Let $e\in \Idem(A_{\kappa(\mathfrak{p})})$. Since $|e|_{\mathfrak{p},\spe}=1\ge |\pi_{\mathfrak{p}}|_{\mathfrak{p}}^c|e|_{\mathfrak{p},\qt}$ by assumption, we have $e\in \pi_{\mathfrak{p}}^{-c}A/\mathfrak{p}A$. Hence, we can choose $e'\in A$ such that $e\equiv S^{-c}e'\mod{\mathfrak{p}}$. Put $h_0:=S^{-2c}({e'}^2-S^ce')\in A[S^{-1}]$. Since
\[
{e'}^2-S^ce'\equiv (S^ce)^2-S^c\cdot S^ce\equiv S^{2c}(e^2-e)\equiv 0\mod{\mathfrak{p}},
\]
we have $h_0\in \mathfrak{p}S^{-2c}\cdot A$. Since $\mathfrak{p}\subset (p,S^e)R^+$, we obtain
\[
|h_0|_{r,\qt}\le\sup{(|S|^e,|p|)}|S|^{-2c}=|p^{1-2cr}|<1.
\]
We define sequences $\{f_n\}$ and $\{h_n\}$ in $A[S^{-1}]$ inductively as follows: Put $f_0:=S^{-c}e'$ and let $h_0$ be as above. For $n\ge 0$, we put
\[
f_{n+1}:=f_n+h_n-2h_nf_n,\ h_{n+1}:=f_{n+1}^2-f_{n+1}\in A[S^{-1}].
\]
Note that for $n\in\N$, we have
\[
f_{n+1}=-f_n^2(2f_n-3),\ f_{n+1}-1=-(f_n-1)^2(2f_n+1),
\]
hence, $h_{n+1}=f_n^2(f_n-1)^2(4f_n^2-4h_n-3)=h_n^2(4h_n-3)$. Then, we have
\[
|h_{n+1}|_{r,\qt}\le |h_n|^2_{r,\qt}\sup{(|h_n|_{r,\qt},1)}.
\]
Therefore, by induction on $n$, we have $|h_n|_r<1$, hence, $|h_{n+1}|_r\le |h_n|^3_r$. In particular, we have $|h_n|_r\to 0\ (n\to\infty)$. We also have
\[
\sup{(|f_{n+1}|_{r,\qt},1)}\le\sup{(|f_n|_{r,\qt},|h_n|_{r,\qt},|h_n|_{r,\qt}|f_n|_{r,\qt},1)}=\sup{(|f_n|_{r,\qt},1)},
\]
hence, $\sup{(|f_n|_{r,\qt},1)}\le\sup{(|f_0|_{r,\qt},1)}$. Therefore, we have
\[
|f_{n+1}-f_n|_{r,\qt}=|h_n(1-2f_n)|_{r,\qt}\le |h_n|_{r,\qt}\sup{(|f_n|_{r,\qt},1)}\le |h_n|_{r,\qt}\sup{(|f_0|_{r,\qt},1)},
\]
in particular, $\{f_n\}_n$ is a Cauchy sequence in $A^{\dagger,r}$ with respect to $|\cdot|_{r,\qt}$. Denote by $f:=\lim_{n\to\infty}f_n$. Since $f^2-f=\lim_{n\to\infty}{h_n}=0$, $f$ is an idempotent of $A^{\dagger,r}$. Since we have $h_n\in\mathfrak{p}\cdot A^{\dagger,r}$ by induction on $n$, $f\equiv f_0\equiv e\mod{\mathfrak{p}}$, i.e., $\alpha(f)=e$.
\end{proof}

\begin{prop}[Continuity of connected components]\label{prop:connected}
Assume that $A_{\kappa(p)}$ is reduced.
\begin{enumerate}
\item There exists $c\in\mathbb{R}_{\ge 0}$ such that
\[
|\cdot|_{(p),\spe}\ge |S|_{(p)}^c|\cdot|_{(p),\qt}\text{ on }A_{\kappa(p)}.
\]
We fix such a $c$ in the following.
\item Let $n\in\N_{\ge 2}$ and $\mathfrak{p}$ an Eisenstein prime ideal of $R^+$ such that $\deg{\mathfrak{p}}>nc$. Then, we have
\[
|\cdot|_{\mathfrak{p},\spe}\ge |\pi_{\mathfrak{p}}|_{\mathfrak{p}}^{\frac{nc}{n-1}}|\cdot|_{\mathfrak{p},\qt}\text{ on }A_{\kappa(\mathfrak{p})}.
\]
\item Let $\mathfrak{p}$ be an Eisenstein prime ideal of $R^+$ such that $\deg{\mathfrak{p}}>3c$. Then, for $r\in \Q_{>0}\cap [1/\deg{\mathfrak{p}},1/2c)$, there exists a canonical bijection
\[
\pi^{\Zar}_0(A_{\kappa(\mathfrak{p})})\to\pi^{\Zar}_0(A^{\dagger,r}).
\]
In particular, we have
\[
\#\pi_0(A_{\kappa(\mathfrak{p})})=\#\pi_0(A_{\kappa(p)})=\#\pi^{\Zar}_0(A^{\dagger,r}).
\]
\end{enumerate}
\end{prop}
\begin{proof}
\begin{enumerate}
\item By assumption, $|\cdot|_{(p),\spe}$ is equivalent to $|\cdot|_{(p),\qt}$ on $A_{\kappa(p)}$. Hence, there exists $\lambda\in\mathbb{R}_{>0}$ such that $|\cdot|_{\spe}\ge\lambda |\cdot|_{\qt}$. By $|1|_{\spe}=|1|_{\qt}=1$, we have $\lambda\le 1$. Hence, $c=\log_{|S|}{\lambda}\ge 0$ satisfies the condition.
\item By (i), we have
\[
|f^n|_{(p),\qt}\ge |f^n|_{(p),\spe}=|f|^n_{(p),\spe}\ge |S|_{(p)}^{nc}|f|^n_{(p),\qt},\ \forall f\in A_{\kappa(p)}.
\]
By Lemma~\ref{lem:normtransfer}, we obtain
\[
|f^n|_{\mathfrak{p},\qt}\ge |\pi_{\mathfrak{p}}|_{\mathfrak{p}}^{nc}|f|^n_{\mathfrak{p},\qt},\ \forall f\in A_{\kappa(\mathfrak{p})}.
\]
By using this inequality iteratively, we obtain
\[
|f^{n^i}|_{\mathfrak{p},\qt}\ge |\pi_{\mathfrak{p}}|_{\mathfrak{p}}^{nc+n^2c+\dots+n^ic}|f|^{n^i}_{\mathfrak{p},\qt}=|\pi_{\mathfrak{p}}|_{\mathfrak{p}}^{\frac{nc(n^i-1)}{n-1}}|f|^{n^i}_{\mathfrak{p},\qt},\forall f\in A_{\kappa(\mathfrak{p})}.
\]
Hence, for all $f\in A_{\kappa(\mathfrak{p})}$, we have $|f|_{\mathfrak{p},\spe}=\inf_{i\in\N}{|f^{n^i}|^{1/n^i}_{\mathfrak{p},\qt}}\ge |\pi_{\mathfrak{p}}|_{\mathfrak{p}}^{nc/(n-1)}|f|_{\mathfrak{p},\qt}$.
\item When $\mathfrak{p}=(p)$, the assertion follows from (i) and Lemma~\ref{lem:connected}. We consider the case of $\mathfrak{p}\neq (p)$. By applying Lemma~\ref{lem:connected} to the inequality in (ii) with $n=3$, we obtain the assertion for $r\in \Q\cap [1/\deg{\mathfrak{p}},1/3c)$. For general $r\in \Q\cap [1/\deg{\mathfrak{p}},1/2c)$, the assertion is reduced to the previous case by taking $\pi_0^{\Zar}$ of the following commutative diagram
\[\xymatrix{
A_{\kappa(p)}\ar@{=}[d]^{\id}&A^{\dagger,r}\ar[l]_{\can.}\ar[r]^{\can.}\ar[d]^{\can.}&A_{\kappa(\mathfrak{p})}\ar@{=}[d]^{\id}\\
A_{\kappa(p)}&A^{\dagger,\frac{1}{\deg{\mathfrak{p}}}}\ar[l]_{\can.}\ar[r]^{\can.}&A_{\kappa(\mathfrak{p})}.
}\]
\end{enumerate}
\end{proof}

\begin{rem}
In \cite[Theorem~1.2.11]{Xia}, Xiao proves $\#\pi_0(A_{\kappa(p)})=\#\pi^{\Zar}_0(A^{\dagger,r})$ under a slightly mild hypothesis on $A$ (cf. \cite[Hypothesis~1.1.10]{Xia}) by a similar idea. To generalize Xiao's result for Eisenstein prime ideals, it seems to be needed to assume that $A$ is flat over $R$.
\end{rem}

To obtain a geometric version of this proposition, we need the following lifting lemma.

\begin{lem}\label{lem:extlifting}
Let $\mathfrak{p}$ be an Eisenstein prime ideal of $R^+$ and $L/\kappa(\mathfrak{p})$ a finite extension. Let $\oo'$ be a Cohen ring of $k_L$ and put $R':=\oo'[[T]]$. Then, there exists a finite flat morphism $\alpha:R^+\to R'$ and an isomorphism $R'/\mathfrak{p}R'\cong \oo_L$ as $R^+/\mathfrak{p}$-algebras. Moreover, for any Eisenstein prime $\mathfrak{q}$ of $R^+$, $\mathfrak{q}R'$ is again an Eisenstein prime ideal with degree $e_{L/\kappa(\mathfrak{p})}\deg(\mathfrak{q})$.
\end{lem}
\begin{proof}
We can define $\alpha$ by a similar way to the definition of $\beta$ in Construction~\ref{const:ocext}: We fix an $\oo'$-algebra structure on $\oo_L$. Let $f:R'\to\oo_L$ be the local $\oo'$-algebra homomorphism, which maps $T$ to a uniformizer $\pi_L$ of $L$. Write $\pi_{\mathfrak{p}}=\pi_L^{e_{L/\kappa(\mathfrak{p})}}\bar{u}$ with $u\in\oo_L^{\times}$. Since $f$ is surjective by Nakayama's lemma, we can choose a lift $u\in (R')^{\times}$ of $\bar{u}$. Since $R^+$ is $p$-adically formally smooth over $\Z[S]$, we can define a morphism $\alpha:R^+\to R'$, which maps $S$ to ${T}^{e_{L/\kappa(\mathfrak{p})}}u$, by the lifting property.

We claim that $\mathfrak{p}R'$ is an Eisenstein prime. Let $P$ be an Eisenstein polynomial of $\oo[S]$, which generates $\mathfrak{p}$. We have $P\equiv T^{\deg(\mathfrak{p})e_{L/\kappa(\mathfrak{p})}}u\mod{pR'}$ for some unit $u\in R'$. By Weierstrass preparation theorem, there exists a distinguished polynomial $Q(T)$ of degree $\deg(\mathfrak{p})e_{L/\kappa(\mathfrak{p})}$ and a unit $U(T)\in R'$ such that $P=Q(T)U(T)$. By evaluating $T=0$, $Q(0)$ is equal to $p$ times a unit of $\oo'$, which implies the claim. In particular, $R'/\mathfrak{p}R'$ is a discrete valuation ring. Hence, the canonical surjection $R'/\mathfrak{p}R'\to\oo_L$ induced by $f$ is an isomorphism. By Nakayama's lemma and the local criteria of flatness, $\alpha$ is finite flat. The second assertion also follows from Weierstrass preparation theorem.
\end{proof}

The following is our main result of this subsection:
\begin{prop}[Continuity of geometric connected components]\label{prop:geomconnected}
Assume that $A_{\kappa(p)}$ is geometrically reduced.
\begin{enumerate}
\item If all connected components of $A_{\kappa(p)}$ are geometrically connected, then all connected components of $A_{\kappa(\mathfrak{p})}$ are also geometrically connected for all Eisenstein prime ideals $\mathfrak{p}$ of $R^+$ with $\deg{\mathfrak{p}}\gg 0$.
\item For all Eisenstein prime ideals $\mathfrak{p}$ of $R^+$ with $\deg{\mathfrak{p}}\gg 0$, we have
\[
\#\pi_0^{\geom}(A_{\kappa(\mathfrak{p})})=\#\pi_0^{\geom}(A_{\kappa(p)}).
\]
\end{enumerate}
\end{prop}
\begin{proof}
\begin{enumerate}
\item By assumption, there exists $c\in\mathbb{R}_{\ge 0}$ such that $|\cdot|_{(p),\spe}\ge |S|^c_{(p)}|\cdot|_{(p),\qt}$ on $A_{\kappa(p)}\otimes_{\kappa(p)}\kappa(p)^{\alg}$. We prove that any Eisenstein prime ideal $\mathfrak{p}$ of $R^+$ with $\deg(\mathfrak{p})>3c$ satisfies the condition. Let $L/\kappa(\mathfrak{p})$ be any finite extension. Let $R'$ be as in Lemma~\ref{lem:extlifting}. Since $R'$ is finite flat over $R^+$, we have $R^+\langle\und{X}\rangle\otimes_{R^+}R'\cong R'\langle\und{X}\rangle$ and $I':=I\otimes_{R^+\langle\und{X}\rangle}R'\langle\und{X}\rangle\cong I\cdot R'\langle\und{X}\rangle$. Hence, we can apply Proposition~\ref{prop:connected} to $R^+=R'$, $I=I'$ and $A=A':=A\otimes_{R^+}R'\cong R'\langle\und{X}\rangle/I'$: Note that $ce_{L/\kappa(\mathfrak{p})}$ can be taken as $c$ in Proposition~\ref{prop:connected}~(i). Therefore, by applying Proposition~\ref{prop:connected}~(iii), we have
\[
\#\pi_0^{\Zar}(A_{\kappa(\mathfrak{p})}\otimes_{\kappa(\mathfrak{p})}L)=\#\pi_0^{\Zar}(A'_{\kappa(\mathfrak{p}R')})=\pi_0^{\Zar}(A'_{\kappa(p)})=\#\pi_0^{\Zar}(A_{\kappa(p)})=\#\pi_0^{\Zar}(A_{\kappa(\mathfrak{p})}),
\]
where the third equality follows from assumption. Therefore, we have $\#\pi_0^{\geom}(A_{\kappa(\mathfrak{p})})=\#\pi_0(A_{\kappa(\mathfrak{p})})$, which implies the assertion.
\item Let $L/\kappa(p)$ be a finite extension such that all connected components of $A_{\kappa(p)}\otimes_{\kappa(p)}L$ are geometrically connected. Let $R'$ be a lifting of $\oo_L$ as in Lemma~\ref{lem:extlifting} and $A'$ as in the proof of (i). By applying (i) and Proposition~\ref{prop:connected}~(iii), we obtain the assertion.
\end{enumerate}
\end{proof}

\subsection{Application: Ramification compatibility of fields of norms}\label{sec:ram}
In this subsection, we prove Theorem~\ref{thm:normcompat}, which is the ramification compatibility of Scholl's equivalence in Theorem~\ref{thm:sch}, as an application of our Gr\"obner basis argument.

We first construct a characteristic zero lift of Abbes-Saito space in characteristic $p$.
\begin{lem}\label{lem:family}
Let $F/E$ be a finite extension of complete discrete valuation fields of characteristic $p$. Assume that the residue field extension $k_F/k_E$ is either trivial or purely inseparable. For $m\in\mathbb{N}$, we denote $\und{X}:=(X_0,\dots,X_m)$ and $\und{Y}:=(Y_0,\dots,Y_m)$.
\begin{enumerate}
\item (\cite[Notation 3.3.8]{Xia}) For some $m\in \mathbb{N}$, there exist a set of generators $\{z_0,\dots,z_m\}$ of $\oo_F$ with $z_0$ a uniformizer of $F$ as an $\oo_E$-algebra and a set of generators $\{p_0,\dots,p_m\}$ of the kernel of the $\oo_E$-algebra homomorphism $\oo_E\langle\und{X}\rangle\twoheadrightarrow\oo_F;X_j\mapsto z_j$ such that
\begin{align*}
p_0&=X_0^{e_{F/E}}+\pi_E\eta_0,\\
p_j&=X_j^{f_j}-\varepsilon_j+X_0\delta_j+\pi_E\eta_j\text{ for }1\le j\le m
\end{align*}
where $\delta_j,\eta_j\in\oo_E\langle\und{X}\rangle$, $\varepsilon_j\in\oo_E\langle X_0,\dots,X_{j-1}\rangle$ and $f_j\in\N$.
\item Let $\succeq$ be the lexicographic order on $\oo_E\langle\und{X}\rangle$ defined by $X_m\succ\dots\succ X_0$. By regarding $\pi_E$ as a regular system of parameter of $\oo_E$, we apply Construction~\ref{const:lead}. Then, we have $\LT_{\oo_E}(p_0^n)=X_0^{ne_{F/E}}$ for all $n\in\N$. Let $l,n\in\N_{>0}$ such that $p^ln\ge e_{F/E}$. Then, for $1\le j\le m$, there exists  $\theta_{j,l,n}\in \oo_E\langle\und{X}\rangle$ such that $\LT_{\oo_E}(p_j^{p^ln}-p_0^{\lfloor p^ln/e_{F/E}\rfloor}\theta_{j,l,n})=uX_j^{f_jp^ln}$ for some unit $u\in 1+\pi_E\oo_E$.
\item (cf. \cite[Example~1.3.4.]{Xia}). Fix an isomorphism $E\cong k_E((S))$. Let $\oo$ be a Cohen ring of $k_E$ and $R:=\oo[[S]]$ with a canonical projection $R\to\oo_E$. Fix a lift $P_j\in R\langle\und{X}\rangle$ of $p_j$ for all $j$. Let $\und{\alpha}\in\N^{m+1},\und{\beta}\in\N^{m+1}_{>0}$. Assume that $\lfloor\beta_j/e_{F/E}\rfloor\ge \beta_0$ for all $1\le j\le m$, and there exists $l\in\N_{>0}$ such that $p^l|\beta_j$ for all $1\le j\le m$. Then, the $R$-algebra
\[
A_{\und{\alpha},\und{\beta}}:=R\langle\und{X},\und{Y}\rangle/(S^{\alpha_j}Y_j-P_j^{\beta_j},\ 0\le j\le m).
\]
is $R$-flat. Moreover, the fiber of $A_{\und{\alpha},\und{\beta}}$ at any Eisenstein prime $\mathfrak{p}$ of $R$ is an affinoid variety, which gives rise to the following affinoid subdomain of $D^{m+1}_{\kappa(\mathfrak{p})}$
\[
D^{m+1}(|\pi_{\mathfrak{p}}|^{-\alpha_j/\beta_j}(P_j\mod{\mathfrak{p}}),\ 0\le j\le m).
\]
\end{enumerate}
\end{lem}
\begin{proof}
\begin{enumerate}
\item[(i)] See \cite[Construction~3.3.5]{Xia} for details.
\item[(ii)] Since the coefficient of $X_0^{ne_{F/E}}$ in $p_0^n$ is equal to $1$, the first assertion follows from $p_0^n\equiv X_0^{ne_{F/E}}\mod{\pi_E}$. We prove the second assertion. Put $\theta_{j,l,n}:=X_0^{p^ln-e_{F/E}\lfloor p^ln/e_{F/E}\rfloor}\delta_j^{p^ln}$. Since
\[
p_j^{p^ln}\equiv X_j^{p^lnf_j}-\varepsilon_j^{p^ln}+X_0^{p^ln}\delta_j^{p^ln}\equiv X_j^{p^lnf_j}-\varepsilon_j^{p^ln}+p_0^{\lfloor p^ln/e_{F/E}\rfloor}\theta_{j,l,n}\mod{\pi_E},
\]
we have $\LT_{k_E}(p_j^{p^ln}-p_0^{\lfloor p^ln/e_{F/E}\rfloor}\theta_{j,l,n}\mod{\pi_E})=\LT_{k_E}(X_j^{p^lnf_j}-\varepsilon_j^{p^ln}\mod{\pi_E})=X_j^{f_jp^ln}$, which implies the assertion.
\item[(iii)] The last assertion is trivial. We prove the first assertion. Let $\succeq$ be the lexicographic order on $\oo_E\langle\und{X},\und{Y}\rangle$ defined by $X_m\succ\dots\succ X_0\succ Y_m\succ\dots\succ Y_0$. By regarding $\{p,S\}$ as a regular system of parameter of $R$, we apply Construction~\ref{const:lead}. For $1\le j\le m$, we choose a lift of $\theta_{j,l,\beta_j/p^l}$ and denote by $\Theta_j$ for simplicity. Then, the ideal $(S^{\alpha_j}Y_j-P_j^{\beta_j},\ 0\le j\le m)$ is generated by $Q_0:=S^{\alpha_0}Y_0-P_0^{\beta_0}$ and
\[
Q_j:=S^{\alpha_j}Y_j-P_j^{\beta_j}-(S^{\alpha_0}Y_0-P_0^{\beta_0})P_0^{\lfloor\beta_j/e_{F/E}\rfloor-\beta_0}\Theta_j
\]
for $1\le j\le m$. We have only to prove that $\LT_{R/\mathfrak{m}_R}(-Q_j\mod{\mathfrak{m}_R})$ are relatively prime monic monomials by Proposition~\ref{prop:concrete}. We have $\LT_{R/\mathfrak{m}_R}(Q_0\mod{\mathfrak{m}_R})=-\LT_{R/\mathfrak{m}_R}(p_0^{\beta_0})=-X_0^{e_{F/E}\beta_0}$. Since
\[
Q_j\equiv -p_j^{\beta_j}+p_0^{\lfloor \beta_j/e_{F/E}\rfloor}\theta_{j,l,\beta_j/p^l}\mod{\mathfrak{m}_R},
\]
we have $\LT_{R/\mathfrak{m}_R}(Q_j\mod{\mathfrak{m}_R})=-X_j^{f_j\beta_j}$ by (ii), which implies the assertion.
\end{enumerate}
\end{proof}

In the rest of this subsection, let notation be as in Definition~\ref{dfn:scholl}.

\begin{lem}\label{lem:asfamily}
Fix an isomorphism $X_{\mathfrak{K}}\cong k_{\mathfrak{K}}((\Pi))$ and let $\oo$ be a Cohen ring of $k_{\mathfrak{K}}$ and $R:=\oo[[\Pi]]$.
\begin{enumerate}
\item There exists a surjective local ring homomorphism $\phi_n:R\to\oo_{K_n}$ for all sufficiently large $n$ making the diagram commutative
\[\xymatrix{
R\ar[d]^{\phi_n}\ar@{->>}[r]^{\can.}&X_{\mathfrak{K}}^+\ar@{->>}[d]^{\pr_n}\\
\oo_{K_n}\ar@{->>}[r]^(.35){\can.}&\oo_{K_n}/\xi\oo_{K_n},
}\]
and $\ker{(\phi_n)}$ is an Eisenstein prime ideal of $R$. We fix $\phi_n$ in the following and denote $\mathfrak{p}_n:=\ker{(\phi_n)}$.
\item Let $r\in\Q_{>0}$ and $L_{\infty}/K_{\infty}$ a finite extension and $\mathfrak{L}=\{L_n\}_{n>0}$ a corresponding strictly deeply ramified tower. Assume that the residue field extension of $X_{\mathfrak{L}}/X_{\mathfrak{K}}$ is either trivial or purely inseparable. Then, there exists a flat $R$-algebra $AS^r$ (resp. $AS^r_{\log}$) of the form $R\langle\und{X}\rangle/I$ for an ideal $I\subset R\langle\und{X}\rangle$, whose fiber at $(p)$, $\mathfrak{p}_n$ are isomorphic to Abbes-Saito spaces $as^r_{X_{\mathfrak{L}}/X_{\mathfrak{K}},\bullet}$, $as^r_{L_n/K_n,\bullet}$ (resp. $as^r_{X_{\mathfrak{L}}/X_{\mathfrak{K}},\bullet,\bullet}$, $as^r_{L_n/K_n,\bullet,\bullet}$) for all sufficiently large $n$.
\item Let notation and assumption be as in (ii). For all sufficiently large $n$, we have
\[
\#\mathcal{F}^r(X_{\mathfrak{L}})=\#\mathcal{F}^r(L_n),\ \#\mathcal{F}^r_{\log}(X_{\mathfrak{L}})=\#\mathcal{F}_{\log}^r(L_n).
\]
\end{enumerate}
\end{lem}
\begin{proof}
Denote $E:=X_{\mathfrak{K}}$ and $F:=X_{\mathfrak{L}}$.
\begin{enumerate}
\item For all sufficiently large $n$, the projection $\pr_n:\oo_E\to\oo_{K_n}/\xi\oo_{K_n}$ induces an isomorphism $\Phi_n:k_{\mathfrak{K}}\to k_{K_n}$ of the residue fields. Hence, we can choose an embedding $\oo\to \oo_{K_n}$ lifting $\Phi_n$. Let $\pi_{K_n}$ be a uniformizer of $\oo_{K_n}$, which is a lift of $\pr_n(\Pi)\in\oo_{K_n}/\xi\oo_{K_n}$. By the formally \'etaleness of the $\oo$-algebra homomorphism $\oo[\Pi]\to R;\Pi\mapsto\Pi$, we obtain $\phi_n$, which maps $\Pi$ to $\pi_{K_n}$. Since $\oo_{K_n}/\oo$ is totally ramified, the kernel of $\phi_n$ is generated by an Eisenstein polynomial.
\item Fix $\xi'\in\oo_{K_{\infty}}$ such that $0<v_p(\xi')< v_p(\xi)$ and $\{L_n\}_{n>0}$ is strictly deeply ramified with respect to $\xi'$. We denote the composite $\mathrm{can}\circ\mathrm{pr}_n:\oo_E\to\oo_{K_n}/\xi\oo_{K_n}\to\oo_{K_n}/\xi'\oo_{K_n}$ by $\mathrm{pr}_n$ again. Fix an expression $r=a/b$ with $a,b\in\N$. Also, fix $l\in\N$ such that $p^l\ge e_{F/E}$. Define $\und{\alpha},\und{\alpha}_{\log},\und{\beta},\und{\beta}_{\log}\in\N^l$ as $\alpha_0:=a$, $\alpha_{\log,0}:=a+b$, $\beta_0:=\beta_{\log,0}:=b$, $\alpha_j=\alpha_{\log,j}=ap^l$, $\beta_j=\beta_{\log,j}:=bp^l$ for $1\le j\le m$. Then, we can apply Lemma~\ref{lem:family} to the finite extension $F/E$: In the following, we use the notation as in the lemma. We will prove that $A_{\und{\alpha},\und{\beta}}$ (resp. $A_{\und{\alpha}_{\log},\und{\beta}_{\log}}$) satisfies the desired condition. We first consider in the non-log case. By Lemma~\ref{lem:family}~(iii), the fiber of $A_{\und{\alpha},\und{\beta}}$ at $(p)$ is isomorphic to $as^r_{F/E,Z}$, where $Z=\{z_0,\dots,z_m\}$. Recall that we have a canonical surjection $\pr_n:\oo_{F}\to\oo_{L_n}/\xi'\oo_{L_n}$ for all sufficiently large $n$. We choose a lift $z_j^{(n)}\in\oo_{L_n}$ of $\pr_n(z_j)\in\oo_{L_n}/\xi'\oo_{L_n}$. Then, $z_j^{(n)}$'s is a generator of $\oo_{L_n}$ as an $\oo_{K_n}$-algebra by Nakayama's lemma and $z_j^{(0)}$ is a uniformizer of $\oo_{L_n}$ by Lemma~\ref{lem:family}~(i). We consider the surjection $\varphi_n:\oo_{K_n}\langle\und{X}\rangle\to\oo_{L_n};X_j\mapsto z^{(n)}_j$ and choose a lift $p_j^{(n)}\in\ker{(\varphi_n)}$ of $\pr_n(p_j)\in\oo_{K_n}/\xi'\oo_{K_n}[\und{X}]$:
\[\xymatrix{
\oo_E\langle\und{X}\rangle\ar@{->>}[d]_{\pr_n}\ar@{->>}[rr]^{X_j\mapsto z_j}&&\oo_F\ar@{->>}[d]^{\pr_n}\\
\oo_{K_n}/\xi'\oo_{K_n}[\und{X}]\ar@{->>}[rr]^{X_j\mapsto \pr_n(z_j)}&&\oo_{L_n}/\xi'\oo_{L_n}\\
\oo_{K_n}\langle\und{X}\rangle\ar@{->>}[rr]^{\varphi_n;X_j\mapsto z_j^{(n)}}\ar@{->>}[u]^{\can.}&&\oo_{L_n}.\ar@{->>}[u]_{\can.}
}\]
By Nakayama's lemma, $p_j^{(n)}$'s is a generator of $\ker{(\varphi_n)}$. We may assume $v_{K_n}(\xi')\ge r$ by choosing $n$ sufficiently large. Since $\phi_n(P_j)\equiv p_j^{(n)}\mod{(\xi')}$, we have $|\phi_n(P_j)(\bm{x})|\le |\pi_{K_n}|^r$ if and only if $|p_j^{(n)}(\bm{x})|\le |\pi_{K_n}|^r$ for any $\bm{x}\in \oo_{\overline{K}}^{m+1}$. This implies that the fiber of $AS^r$ at $\mathfrak{p}_n$ is isomorphic to $as^r_{L_n/K_n,Z^{(n)}}$, where $Z^{(n)}=\{z^{(n)}_0,\dots,z^{(n)}_m\}$, which implies the assertion. In the log case, a similar proof works if we choose $n$ sufficiently large such that $v_{K_n}(\xi')\ge r+1$.
\item By applying Proposition~\ref{prop:geomconnected} to $AS^r$ and $AS^r_{\log}$, we obtain the assertion.
\end{enumerate}
\end{proof}

The following is the main theorem in this subsection. See \cite[\S~6]{Hat} for an alternative proof.

\begin{thm}\label{thm:normcompat}
Let $L_{\infty}/K_{\infty}$ be a finite separable extension and $\mathfrak{L}=\{L_n\}_{n>0}$ a corresponding strictly deeply ramified tower. Then, the sequence $\{b(L_n/K_n)\}_{n>0}$ (resp. $\{b_{\log}(L_n/K_n)\}_{n>0}$) converges to $b(X_{\mathfrak{L}}/X_{\mathfrak{K}})$ (resp. $b_{\log}(X_{\mathfrak{L}}/X_{\mathfrak{K}})$).
\end{thm}
\begin{proof}
Since the non-log and log ramification filtrations are invariant under base change, so are the non-log and log ramification breaks. Hence, we may assume that the residue field extension of $X_{\mathfrak{L}}/X_{\mathfrak{K}}$ is either trivial or purely inseparable by replacing $K_{\infty}$ and $L_{\infty}$ by its maximal unramified extensions. We first prove in the non-log case. Recall that we have $[X_{\mathfrak{L}}:X_{\mathfrak{K}}]=[L_n:K_n]$ for all sufficiently large $n$ by Theorem~\ref{thm:sch}. For $r\in\Q_{>0}$ such that $b(X_{\mathfrak{L}}/X_{\mathfrak{K}})<r$, we have $\#\mathcal{F}^r(L_n)=\#\mathcal{F}^r(X_{\mathfrak{L}})=[L_n:K_n]$ for all sufficiently large $n$ by Lemma~\ref{lem:asfamily}. Hence, we have $\limsup_n{b(L_n/K_n)}\le b(X_{\mathfrak{L}}/X_{\mathfrak{K}})$. For $r\in\Q_{>0}$ such that $b(X_{\mathfrak{L}}/X_{\mathfrak{K}})>r$, we have $\#\mathcal{F}^r(L_n)=\#\mathcal{F}^r(X_{\mathfrak{L}})<[L_n:K_n]$ for all sufficiently large $n$ by Lemma~\ref{lem:asfamily} and the definition of $\mathcal{F}^r$. Hence, we have $\liminf_n{b(L_n/K_n)}\ge b(X_{\mathfrak{L}}/X_{\mathfrak{K}})$. Therefore, we have $b(X_{\mathfrak{L}}/X_{\mathfrak{K}})\le\liminf_n{b(L_n/K_n)}\le\limsup_n{b(L_n/K_n)}\le b(X_{\mathfrak{L}}/X_{\mathfrak{K}})$, which implies the assertion. In the log case, the same argument replaced $b$ and $\mathcal{F}^r$ by $b_{\log}$ and $\mathcal{F}^r_{\log}$ works.
\end{proof}

The following representation version of Theorem~\ref{thm:normcompat} will be used in the proof of Main Theorem~\ref{thm:main}.
\begin{lem}\label{lem:normrep}
Let $F/\Q_p$ be a finite extension and $V\in\rep^f_{F}(G_{K_n})$ a finite $F$-representation for some $n$. We identify $G_{X_{\mathfrak{K}}}$ with $G_{K_{\infty}}$ by the equivalence in Theorem~\ref{thm:sch}.
\begin{enumerate}
\item For $m\ge n$, let $L_m$ (resp. $L_{\infty}$, $X'$) be the finite Galois extension corresponding to the kernel of the action of $G_{K_m}$ (resp. $G_{K_{\infty}}$, $G_{X_{\mathfrak{K}}}$) on $V$. Then, $L_{\infty}$ corresponds to $X'$ under the equivalence in Theorem~\ref{thm:sch} and $\{L_m\}_{m\ge n}$ is a strictly deeply ramified tower corresponding to $L_{\infty}$.
\item The sequences $\{\art^{\AS}(V|_{K_m})\}_{m\ge n}$ and $\{\sw^{\AS}(V|_{K_m})\}_{m\ge n}$ are eventually stationary and their limits are equal to $\art^{\AS}(V|_{X_{\mathfrak{K}}})$ and $\sw^{\AS}(V|_{X_{\mathfrak{K}}})$.
\end{enumerate}
\end{lem}
\begin{proof}
\begin{enumerate}
\item The first assertion is trivial. We prove the second assertion. Since $G_{L_n}\cap G_{K_m}=G_{L_m}$ for all $m\ge n$, we have $L_m=L_nK_m$. Therefore, $\{L_m\}$ is a strictly deeply ramified tower corresponding to $L'_{\infty}:=\cup_mL_m$. Hence, we have only to prove that $L_{\infty}=L'_{\infty}$. Let $\rho:G_{K_n}\to GL(V)$ be a matrix presentation of $V$. By the commutative diagram
\[\xymatrix{
1\ar[r]&G_{L_{\infty}}\ar[r]^{\inc.}&G_{K_{\infty}}\ar@{^(->}[d]^{\can.}\ar[r]^{\rho}&GL(V)\ar[d]^{\id}\\
1\ar[r]&G_{L_m}\ar[r]^{\inc.}&G_{K_m}\ar[r]^{\rho|_{G_{K_m}}}&GL(V),
}\]
where the horizontal sequences are exact, we obtain a canonical injection $G_{L_{\infty}}\hookrightarrow G_{L_m}$. Therefore, we have $L_m\subset L_{\infty}$, hence, $L'_{\infty}\subset L_{\infty}$. To prove the converse, we have only to prove $[L_{\infty}:K_{\infty}]\le [L'_{\infty}:K_{\infty}]$. Since $(K_{\infty}\cap L_n)/K_n$ is finite, we have $K_{\infty}\cap L_n=K_m\cap L_n$ for sufficiently large $m$, in particular,
\[
[L'_{\infty}:K_{\infty}]=[L_nK_{\infty}:K_{\infty}]=[L_n: K_{\infty}\cap L_n]=[L_n:K_m\cap L_n]=[L_nK_m:K_m]=[L_m:K_m].
\]
Then, the assertion follows from
\[
[L_{\infty}:K_{\infty}]=\#\rho(G_{K_{\infty}})\le \#\rho(G_{K_m})=[L_m:K_m].
\]
\item By Maschke's theorem, there exists an irreducible decomposition $V|_{X_{\mathfrak{K}}}=\oplus_{\lambda}V^{\lambda}$ with $V^{\lambda}\in\rep^f_F(G_{X_{\mathfrak{K}}})$. We choose $m_0\in\N$ such that the canonical map $G_{L_{\infty}/K_{\infty}}\to G_{L_m/K_m}$ is an isomorphism for all $m\ge m_0$. Then, $V^{\lambda}$ is $G_{K_m}$-stable for all $m\ge m_0$. Moreover, $V^{\lambda}|_{K_m}\in\rep^f_F(G_{K_m})$ is irreducible. For $m\ge m_0$, let $L_m^{\lambda}/K_m$ be the finite Galois extension corresponding to the kernel of the action of $G_{K_m}$ on $V^{\lambda}$. By (i), $\mathfrak{L}^{\lambda}=\{L_m^{\lambda}\}_{m\ge m_0}$ is a strictly deeply ramified tower and $X_{\mathfrak{L}^{\lambda}}$ corresponds to the kernel of the action of $G_{X_{\mathfrak{K}}}$ on $V^{\lambda}$. By the irreducibility of the action of $G_{K_m}$ (resp. $G_{X_{\mathfrak{K}}}$) on $V^{\lambda}$, we have
\[
\art^{\AS}(V^{\lambda}|_{K_m})=b(L_m^{\lambda}/K_m)\dim_F(V),
\]
\[
\art^{\AS}(V^{\lambda}|_{X_{\mathfrak{K}}})=b(X_{\mathfrak{L}^{\lambda}}/X_{\mathfrak{K}})\dim_F(V)
\]
for $m\ge m_0$. By applying Theorem~\ref{thm:normcompat} to each $\mathfrak{L}^{\lambda}$, we have $\lim_{m\to\infty}\art(V|_{K_m})=\art(V|_{X_{\mathfrak{K}}})$. Note that $K_m$ is not absolutely unramified for sufficiently large $m$. Indeed, the definition of strictly deeply ramifiedness implies that $K_{m+1}/K_m$ is not unramified. By Theorem~\ref{thm:mixed}, the convergence of $\{\art(V|_{K_m})\}$ implies that $\{\art(V|_{K_m})\}$ is eventually stationary, which implies the assertion for Artin conductor. The assertion for Swan conductor follows from the same argument by replacing $\art$, $b$ by $\sw$, $b_{\log}$.
\end{enumerate}
\end{proof}

\begin{rem}[Some Hasse-Arf property]
Let notation be as in Lemma~\ref{lem:normrep} and let $p=2$. By Theorem~\ref{thm:Xiao} and Lemma~\ref{lem:normrep}~(ii), $\sw(V|_{K_m})$ is an integer for all sufficiently large $m$ (cf. Theorem~\ref{thm:mixed}).
\end{rem}

\section{Differential modules associated to de Rham representations}\label{sec:diff}
In this section, we first construct $\N_{\dR}(V)$ for de Rham representations $V\in\rep_{\Q_p}(\gk)$ as a $(\varphi,\Gamma_K)$-module (\S~\ref{subsec:const}). Then, we prove that $\N_{\dR}(V)$ can be endowed with a $(\varphi,\nabla)$-module structure (\S~\ref{subsec:diffaction}). Then, we define Swan conductors of de Rham representations (\S~\ref{subsec:swan}) and we will prove that the differential Swan conductor of $\N_{\dR}(V)$ and Swan conductor of $V$ are compatible (\S~\ref{subsec:main}).

Throughout this section, let $K$ be a complete discrete valuation field of mixed characteristic $(0,p)$. Except for \S~\ref{subsec:swan}, we assume that $K$ satisfies Assumption~\ref{ass:relative}, and we use the notation of \S~\ref{subsec:Galois}.

\subsection{Calculation of horizontal sections}\label{subsec:horizontal}
When $k_K$ is perfect, $\N_{\dR}(V)$ is constructed by gluing a certain family of vector bundles over $K_n[[t]]$ for $n\gg 0$ (\cite[II~1]{Ber}). When $k_K$ is not perfect, $K_n[[t]]$ should be replaced by the ring of horizontal sections of $K_n[[u,t_1,\dots,t_d]]$ with respect to the connection $\nabla^{\geom}$, which will be studied in this subsection.

\begin{dfn}
\begin{enumerate}
\item We have a canonical $K_n$-algebra injection
\[
K_n[[t,u_1,\dots,u_d]]\to \B_{\dR}^+
\]
since $\B_{\dR}^+$ is a complete local $K^{\alg}$-algebra. We call the topology of $K_n[[t,u_1,\dots,u_d]]$ as a subring of $\B_{\dR}^+$ (endowed with the canonical topology) the canonical topology. Note that $K_n[[t,u_1,\dots,u_d]]$ is stable by $\gk$ and the $\gk$-action factors through $\Gamma_K$.
\item Let $F$ be a complete valuation field. The Fr\'echet topology on
\[
F[[X_1,\dots,X_n]]\cong \varprojlim_m{F[X_1,\dots,X_n]/(X_1,\dots,X_n)^m}
\]
is the inverse limit topology, where $F[X_1,\dots,X_n]/(X_1,\dots,X_n)^m$ is endowed with the (unique) topological $F$-vector space structure. Note that $F[[X_1,\dots,X_n]]$ is a Fr\'echet space, and that the $(X_1,\dots,X_n)$-adic topology of $F[[X_1,\dots,X_n]]$ is finer than the Fr\'echet topology.
\end{enumerate}
\end{dfn}

\begin{lem}
The Fr\'echet topology and the canonical topology of $K_n[[t,u_1,\dots,u_d]]$ are equivalent. In particular, $K_n[[t,u_1,\dots,u_d]]$ is a closed subring of $\B_{\dR}^+$.
\end{lem}
\begin{proof}
Denote $V_m:=K_n[t,u_1,\dots,u_d]/(t,u_1,\dots,u_d)^m$ and we identify $K_n[[t,u_1,\dots,u_d]]$ with $\varprojlim_mV_m$. If we endow $V_m$ with the (unique) topological $K_n$-vector space structure, then the resulting inverse limit topology is the Fr\'echet topology. We have a canonical injection $V_m\to \B_{\dR}^+/(t,u_1,\dots,u_d)^m$. If we endow $V_m$ with the subspace topology as a subset of $\B_{\dR}^+/(t,u_1,\dots,u_d)^m$, which is endowed with the canonical topology, then the resulting inverse limit topology is the canonical topology. Since $\B_{\dR}^+/(t,u_1,\dots,u_d)^m$ is $K_n$-Banach space by definition, $V_m$ endowed with this topology is a topological $K_n$-vector space, which implies the assertion.
\end{proof}

\begin{notation}
For $n\in\N$, we denote by $K_n[[t,u_1,\dots,u_d]]^{\nabla}$ the subring $K_n[[t,u_1,\dots,u_d]]^{\nabla^{\geom}=0}=\B_{\dR}^{\nabla+}\cap K_n[[t,u_1,\dots,u_d]]$ of $\B_{\dR}^{\nabla+}$. We call the subspace topology of $K_n[[t,u_1,\dots,u_d]]^{\nabla}$ as a subring of $\B_{\dR}^+$ (endowed with the canonical topology) the canonical topology. Note that $K_n[[t,u_1,\dots,u_d]]^{\nabla}$ is a closed subring of $\B_{\dR}^{\nabla+}$ since the connection $\nabla^{\geom}:\B_{\dR}^+\to \B_{\dR}^+\otimes_K \hat{\Omega}^1_K$ is continuous and $\B_{\dR}^{\nabla+}$ is closed in $\B_{\dR}^+$.
\end{notation}

\begin{lem}\label{lem:hori}
The ring $K_n[[t,u_1,\dots,u_d]]^{\nabla}$ is a complete discrete valuation ring with residue field $K_n$ and uniformizer $t$.
\end{lem}
\begin{proof}
We define a map
\[
f:K_n[t,u_1,\dots,u_d]\to K_n[[t,u_1,\dots,u_d]];x\mapsto \sum_{(n_1,\dots,n_d)\in\N^d}\frac{(-1)^{n_1+\dots+n_d}}{n_1!\dots n_d!}u_1^{n_1}\dots u_d^{n_d}\partial_1^{n_1}\circ\dots\circ\partial_d^{n_d}(x).
\]
It is easy to check that this is an abstract ring homomorphism such that $\mathrm{Im}(f)\subset K_n[[t,u_1,\dots,u_d]]^{\nabla}$, $f(tx)=tf(x)$ for all $x\in K_n[t,u_1,\dots,u_d]$ and $f(u_j)=0$ for all $j$. In particular, $f$ is $(t,u_1,\dots,u_d)$-adically continuous. Passing to the completion, we obtain a ring homomorphism $f:K_n[[t,u_1,\dots,u_d]]\to K_n[[t,u_1,\dots,u_d]]^{\nabla}$. Since $f$ is identity on $K_n[[t,u_1,\dots,u_d]]^{\nabla}$, $f$ is surjective and $f$ induces a surjection
\[
\bar{f}:K_n[[t]]\cong K_n[[t,u_1,\dots,u_d]]/(u_1,\dots,u_d)\to K_n[[t,u_1,\dots,u_d]]^{\nabla},
\]
where the first isomorphism is induced by the inclusion $K_n[[t]]\subset K_n[[t,u_1,\dots,u_d]]$. Since $\bar{f}(t)=t$ is non-zero, $\bar{f}$ is an isomorphism, which implies the assertion.
\end{proof}

\begin{lem}\label{lem:comparisontop}
The $t$-adic topology is finer than the canonical topology in $K_n[[t,u_1,\dots,u_d]]^{\nabla}$.
\end{lem}
\begin{proof}
Denote $K_n[[t,u_1,\dots,u_d]]^{\nabla}$ by $R$ and we identify $R$ with $\varprojlim_mR/t^mR$. If we endow $R/t^mR$ with the discrete topology, then the resulting inverse limit topology is the $t$-adic topology. By Lemma~\ref{lem:hori} and d\'evissage, the canonical map $R/t^mR\to K_n[t,u_1,\dots,u_d]/(t,u_1,\dots,u_d)^m$ is injective. If we endow $R/t^mR$ with the subspace topology as a subset of $K_n[t,u_1,\dots,u_d]/(t,u_1,\dots,u_d)^m$, which is endowed with the (unique) topological $K_n$-vector space structure, then the resulting inverse limit topology is the canonical topology. Since the discrete topology is the finest topology, we obtain the assertion.
\end{proof}

The map $f$ defined in the proof of Lemma~\ref{lem:hori} is continuous when $K=\wtil{K}$:

\begin{lem}\label{lem:contmap}
Let $\varphi:\oo_{\wtil{K}}\to\oo_{\wtil{K}}$ be the unique Frobenius lift characterized by $\varphi(t_j)=t_j^p$ for all $1\le j\le d$. Then, the map $f:\wtil{K}_n[[t,u_1,\dots,u_d]]\to \wtil{K}_n[[t,u_1,\dots,u_d]]^{\nabla}$ defined in the proof of Lemma~\ref{lem:hori} is continuous with respect to the Fr\'echet topologies.
\end{lem}
\begin{proof}
By the definition of $f$, we have only to prove the following claim: For all $m\in\N$ and $1\le j\le d$, we have
\[
\partial_j^m(\oo_{\wtil{K}})\subset m!\oo_{\wtil{K}}.
\]
We first note that by the commutativity of $d$ and $\varphi_*:\hat{\Omega}^1_{\oo_{\wtil{K}}}\to\hat{\Omega}^1_{\oo_{\wtil{K}}}$, we have
\begin{equation}\label{eq:der}
\partial_j\circ\varphi^i=p^it_j^{p^i-1}\varphi^i\circ\partial_j
\end{equation}
for all $i\in\N$ and $1\le j\le d$. We prove the claim. Fix $m$ and choose $i\in\N$ such that $v_p(m!)\le i$. Since $k_{\wtil{K}}=k_{\wtil{K}}^{p^i}[\bar{t}_1,\dots,\bar{t}_d]$, we have $\oo_{\wtil{K}}=\varphi^i(\oo_{\wtil{K}})[t_1,\dots,t_d]$ by Nakayama's lemma. By Leibniz rule, we have
\begin{equation}\label{eq:Leibniz}
\partial_j^m(\varphi^i(\lambda)t_1^{a_1}\dots t_d^{a_d})=\sum_{0\le m_0\le m}\binom{m}{m_0}\partial_j^{m_0}(\varphi^i(\lambda))t_1^{a_1}\dots\partial_j^{m-m_0}(t_j^{a_j})\dots t_d^{a_d}
\end{equation}
for $\lambda\in\oo_{\wtil{K}}$ and $a_1,\dots,a_d\in\N$. We have $\partial_j^{m_0}(\varphi^i(\lambda))\in p^i\oo_{\wtil{K}}\subset m!\oo_{\wtil{K}}$ unless $m_0=0$ by (\ref{eq:der}), and $\partial_j^m(t_j^{a_j})\in m!\oo_{\wtil{K}}$. Hence, the RHS of (\ref{eq:Leibniz}) belongs to $m!\oo_{\wtil{K}}$, which implies the claim.
\end{proof}

\subsection{Construction of $\N_{\dR}$}\label{subsec:const}
In this subsection, we will construct $\N_{\dR}(V)$, as a $(\varphi,\Gamma_K)$-module, for de Rham representations $V$. The idea of construction is similar to Berger's idea (\cite[\S~II]{Ber}), i.e., gluing a compatible family of vector bundles over $K_n[[t,u_1,\dots,u_d]]^{\nabla}$ to obtain a vector bundles over $\B^{\dagger,r}_{\rig}$.

\begin{notation}\label{notation:tacitly}
For $n\in\N$, put $r(n):=1/p^{n-1}(p-1)$. For $r\in\Q_{>0}$, let $n(r)\in\N$ be the least integer $n$ such that $r\ge r(n)$.

For each $K$, we fix $r_0$ such that $\A_K$ has enough $r_0$-units (Construction~\ref{con:oc}) and $\A_K^{\dagger,r}\cong\oo'((\pi'))^{\dagger,r/e_{K/\wtil{K}}}$ for all $r\in \Q_{>0}\cap (0,r_0)$ (Lemma~\ref{lem:ocexplicit}), where $\oo'$ is a Cohen ring of $k_{\E_K}$. In the rest of this section, let $r\in\Q_{>0}$, and when we consider $\A^{\dagger,r}_K$, $\B^{\dagger,r}_K$ and $\B^{\dagger,r}_{\rig,K}$, we tacitly assume $r\in\Q_{>0}\cap (0,r_0)$ unless otherwise is mentioned. Moreover, for $V\in\rep_{\Q_p}(G_K)$, we further choose $r_0$ sufficiently small (however, it depends on $V$) such that $\D^{\dagger,r}(V)$ admits a $\B^{\dagger,r}_K$-basis for all $r\in (0,r_0)$. Note that $\A^{\dagger,r}_K$, $\B^{\dagger,r}_K$ are PID's and $\B^{\dagger,r}_{\rig,K}$ is a B\'ezout integral domain.
\end{notation}

\begin{dfn}
Let $r>0$ and $n\in\N$ such that $n\ge n(r)$. For $x=\sum_{k\gg -\infty}p^k[x_k]\in \wtil{\B}^{\dagger,r}$, the sequence $\{\sum_{k\le N}p^k[x_k^{p^{-n}}]\}_{N\in\Z}$ converges in $\B_{\dR}^{\nabla+}$. Moreover, if we put
\[
\iota_n:\wtil{\B}^{\dagger,r}\to\B_{\dR}^{\nabla+};x\mapsto\sum_{k\gg-\infty}{p^k[x_k^{p^{-n}}]},
\]
then $\iota_n$ is a continuous ring homomorphism (see the proof of \cite[Lemme~7.2]{AB2} for details). Since $\B_{\dR}^{\nabla+}$ is Fr\'echet complete, $\iota_n$ extends to a continuous ring homomorphism
\[
\iota_n:\wtil{\B}^{\dagger,r}_{\rig}\to \B_{\dR}^{\nabla+}.
\]
We also denote by $\iota_n$ the restriction of $\iota_n$ to $\wtil{\B}^{\dagger,r}_{\rig,K}$ or $\B^{\dagger,r}_{\rig,K}$. Unless otherwise is mentioned, we also denote by $\iota_n$ the composite of $\iota_n$ and the inclusion $\B_{\dR}^{\nabla+}\subset \B_{\dR}^+$.
\end{dfn}

\begin{lem}\label{lem:unit}
For $x\in\B^{\dagger,r}_{\rig,K}$, the following are equivalent:
\[
x\in (\B^{\dagger,r}_{K})^{\times}\Leftrightarrow x\in (\B^{\dagger,r}_{\rig,K})^{\times}\Leftrightarrow x\text{ has no slopes}\Leftrightarrow  x\in (\wtil{\B}^{\dagger,r}_{K})^{\times}\Leftrightarrow x\in (\wtil{\B}^{\dagger,r}_{\rig,K})^{\times}.
\]
\end{lem}
\begin{proof}
Note that the slopes of $x$ as an element of $\B^{\dagger,r}_{\rig,K}$ or $\wtil{\B}^{\dagger,r}_{\rig,K}$ are the same by definition (cf. \S~\ref{sec:adeq}). Therefore, the assertion follows from \cite[Corollary~2.5.12]{Doc}.
\end{proof}

\begin{lem}\label{lem:kertheta}
Let $B$ be one of $\{\B^{\dagger,r}_K,\B^{\dagger,r}_{\rig,K},\wtil{\B}^{\dagger,r}_K,\wtil{\B}^{\dagger,r}_{\rig,K}\}$. Then, we have
\[
\ker{(\theta\circ\iota_n:B\to\mathbb{C}_p)}=\varphi^{n-1}(q)B
\]
for $n\ge n(r)$.
\end{lem}
\begin{proof}
Note that since $\wtil{\E}_K$ and $\wtil{\E}_{K\wtil{K}^{\mathrm{pf}}}$ are isomorphic, the associated analytic rings $\wtil{\B}^{\dagger,r}_{\rig,K}$ and $\wtil{\B}^{\dagger,r}_{\rig,K\wtil{K}^{\mathrm{pf}}}$ are isomorphic. Hence, in the case of $B=\wtil{\B}^{\dagger,r}_{\rig,K}$, it follows from \cite[Proposition~4.8]{Ber}. By regarding $\C_p$ as the completion of an algebraic closure of $\wtil{K}^{\mathrm{pf}}$ and applying \cite[Remarque~2.14]{Ber}, we have $\ker{(\theta\circ\iota_n:\wtil{\B}^{\dagger,r}\to\mathbb{C}_p)}=\varphi^{n-1}(q)\wtil{\B}^{\dagger,r}$. Since $(\wtil{\B}^{\dagger,r})^{H_K}=\wtil{\B}^{\dagger,r}_K$ and $\varphi^{n-1}(q)\in \wtil{\B}^{\dagger,r}_K$, we obtain the assertion for $B=\wtil{\B}^{\dagger,r}_K$. We will prove the assertion for $B=\B^{\dagger,r}_{\rig,K}$. Let $x\in\ker{(\theta\circ\iota_n:\B^{\dagger,r}_{\rig,K}\to\mathbb{C}_p)}$. Since $\B^{\dagger,r}_{\rig,K}$ is a B\'ezout integral domain, we have $(x,\varphi^{n-1}(q))=(y)$ for some $y\in\B^{\dagger,r}_{\rig,K}$. Let $y'\in\B^{\dagger,r}_{\rig,K}$ such that $\varphi^{n-1}(q)=yy'$. Since $y\in\ker{(\theta\circ\iota_n:\wtil{\B}^{\dagger,r}_{\rig,K}\to\mathbb{C}_p)}=\varphi^{n-1}(q)\wtil{\B}^{\dagger,r}_{\rig,K}$, we have $y=\varphi^{n-1}(q)y''$ for some $y''\in\wtil{\B}^{\dagger,r}_{\rig,K}$, hence, $y'y''=1$. By Lemma~\ref{lem:unit}, $y'$ is a unit in $\B^{\dagger,r}_{\rig,K}$. Hence, we have $x\in\varphi^{n-1}(q)\B^{\dagger,r}_{\rig,K}$ for any $x\in\ker{(\theta\circ\iota_n:\B^{\dagger,r}_{\rig,K}\to\mathbb{C}_p)}$, which implies the assertion. For $B=\B^{\dagger,r}_K$, a similar proof works since $\B^{\dagger,r}_K$ is a PID, hence, a B\'ezout integral domain.
\end{proof}

\begin{lem}
The image of $\B^{\dagger,r}_{\rig,K}$ under $\iota_n$ is contained in $K_n[[t,u_1,\dots,u_d]]$ for $n\ge n(r)$. In particular, $\iota_n$ induces a morphism $\iota_n:\B^{\dagger,r}_{\rig,K}\to K_n[[t,u_1,\dots,u_d]]^{\nabla}$ for $n\ge n(r)$.
\end{lem}
\begin{proof}
Since $\B^{\dagger,r}_{\rig,K}\subset\B^{\dagger,r(n)}_{\rig,K}$, we may assume $r=r(n)$. By \cite[Lemme~8.5]{AB2}, there exists a subring $\mathcal{A}_{R,(1,(p-1)p^{n-1})}$ of $\wtil{\A}$ such that $\A_K^{\dagger,r(n)}=\mathcal{A}_{R,(1,(p-1)p^{n-1})}[[\bar{\pi}]^{-1}]$ and $\iota_n(\B^{\dagger,r}_K)\subset K_n[[t,u_1,\dots,u_d]]$, which is proved in the proof of \cite[Proposition~8.6]{AB2}. Since $K_n[[r,u_1,\dots,u_d]]$ is closed in $\B_{\dR}^+$, we obtain the assertion.
\end{proof}

\begin{lem}\label{lem:surj}
For $h\in\N$ and $n\ge n(r)$, the morphism
\[
\pr_h\circ\iota_n:\B^{\dagger,r}_{\rig,K}\to K_n[[t,u_1,\dots,u_d]]^{\nabla}/t^hK_n[[t,u_1,\dots,u_d]]^{\nabla}
\]
is surjective.
\end{lem}
\begin{proof}
By $t\in K_n[[t,u_1,\dots,u_d]]^{\nabla}$ and Lemma~\ref{lem:hori}, we may assume $h=0$. Put $\theta_n:=\theta\circ\iota_n$. Let $\mathbb{A}^+_K\subset \mathbb{A}^{\dagger,r}_K$ be as in \cite[Proposition~4.42]{AB}. By the proof of \cite[Lemme~8.2]{AB2}, $\theta_n:\mathbb{A}_K^+\to\oo_{K_n}$ is surjective after taking modulo some power of $p$. Since $\mathbb{A}_K^+$ is Noetherian and $(p/\pi^a,p)$-adically Hausdorff complete, $\mathbb{A}_K^+$ is $p$-adically Hausdorff complete, which implies the surjectivity of $\theta_n:\mathbb{A}_K^+\to\oo_{K_n}$ by Nakayama's lemma.
\end{proof}

\begin{lem}\label{lem:dense}
For $n\ge n(r)$, the image of $\B^{\dagger,r}_{\rig,K}$ under $\iota_n$ is dense in $K_n[[t,u_1,\dots,u_d]]^{\nabla}$ with respect to the canonical topology.
\end{lem}
\begin{proof}
By Lemma~\ref{lem:comparisontop}, the assertion follows from Lemma~\ref{lem:surj}.
\end{proof}

\begin{lem}[{\cite[Corollary~2.8.5, Definition~2.9.5]{Doc}, see also \cite[Proposition~1.1.1]{Bpair}}]\label{lem:Robba}
Let $B$ be one of $\{\btrigdag,\ \wtil{\B}^{\dagger,r}_{\rig},\ \brigdagk, \B^{\dagger,r}_{\rig,K}\}$ and $M$ a $B$-submodule of a finite free $B$-module. Then, the following are equivalent:
\begin{enumerate}
\item $M$ is finite free;
\item $M$ is closed;
\item $M$ is finitely generated.
\end{enumerate}
\end{lem}

\begin{lem}\label{lem:elemrobba}
Let $B$ be either $\wtil{\B}^{\dagger,r}_{\rig}$ or $\B^{\dagger,r}_{\rig,K}$. If $I$ is a principal ideal of $B$, which divides $(t^h)$ for some $h\in\N$, then $I$ is generated by an element of the form $\Pi_{n\ge n(r)}(\varphi^{n-1}(q)/p)^{j_n}$ with $j_n\le h$.
\end{lem}
\begin{proof}
Note that we have a slope factorization $t=\pi\Pi_{n\ge 1}(\varphi^{n-1}(q)/p)$ in $\B^{\dagger,r}_{\rig,\Q_p}$ (see the proof of \cite[Proposition~I.~2.2]{Ber}). For $n<n(r)$, $\varphi^{n-1}(q)/p$ is a unit in $\B^{\dagger,r}_{\rig,\Q_p}$ and for $n\ge n(r)$, $\varphi^{n-1}(q)/p$ generates a prime ideal of $B$ by Lemma~\ref{lem:kertheta}. Hence, the assertion follows from the uniqueness of slope factorization (Lemma~\ref{lem:uniquefac}).
\end{proof}

\begin{lem}[The existence of a partition of unity]\label{lem:partition}
Let $n\in\N$ and $r>0$ such that $n\ge n(r)$. For $w\in\N_{>0}$, there exists $t_{n,w}\in\B^{\dagger,r}_{\rig,K}$ such that $\iota_n(t_{n,w})=1\mod{t^wK_n[[t,u_1,\dots,u_d]]^{\nabla}}$ and $\iota_m(t_{n,w})\in t^wK_m[[t,u_1,\dots,u_d]]^{\nabla}$ for $m\neq n$ such that $m\ge n(r)$.
\end{lem}
\begin{proof}
Since $\B^{\dagger,r}_{\rig,\Q_p}\subset\B^{\dagger,r}_{\rig,K}$ and $\Q_p(\zeta_{p^m})[[t]]\subset K_m[[t,u_1,\dots,u_d]]^{\nabla}$, we may assume $K=\Q_p$. Then, the assertion follows from \cite[Lemma~I.~2.1]{Ber}.
\end{proof}

\begin{lem}\label{lem:unique}
Let $B$ be either $\wtil{\B}^{\dagger,r}_{\rig}$ or $\B^{\dagger,r}_{\rig,K}$. For $n\ge n(r)$, denote $\iota_n:B:=\wtil{\B}^{\dagger,r}_{\rig}\to B_n:=\B_{\dR}^{\nabla+}$ in the first case and $\iota_n:B:=\B^{\dagger,r}_{\rig,K}\to B_n:=K_n[[t,u_1,\dots,u_d]]^{\nabla}$ in the second case. Let $D$ be a $\varphi$-module over $B$ of rank $d'$ and $D^{(1)}$ and $D^{(2)}$ two $B$-submodules of rank $d'$ stable by $\varphi$ of $D[t^{-1}]=B[t^{-1}]\otimes_B{D}$ such that
\begin{enumerate}
\item $D^{(1)}[t^{-1}]=D^{(2)}[t^{-1}]=D[t^{-1}]$;
\item $B_n\otimes_{\iota_n,B}D^{(1)}=B_n\otimes_{\iota_n,B}D^{(2)}$ for all $n\ge n(r)$.
\end{enumerate}
Then, we have $D^{(1)}=D^{(2)}$.
\end{lem}
\begin{proof}
Since $D^{(1)}+D^{(2)}$ is finite free by Lemma~\ref{lem:Robba} and satisfies the same condition as $D^{(2)}$, we may assume that $D^{(1)}\subset D^{(2)}$ by replacing $D^{(2)}$ by $D^{(1)}+D^{(2)}$. Then, the proof of \cite[Proposition~I.~3.4]{Ber} works by using the ingredients Lemma 2.0.13, and Lemma 4.2.10 instead of \cite[Proposition~I.~2.2]{Ber}.
\end{proof}

\begin{prop}[{cf. \cite[Th\'eor\`eme~II~1.2]{Ber}}]\label{prop:const}
Let $V\in \rep_{\dR}(\gk)$ be a de Rham representation with negative Hodge-Tate weights. Let $B$ be either $\wtil{\B}^{\dagger,r}_{\rig}$ or $\B^{\dagger,r}_{\rig,K}$. Let $B_n$ and $\iota_n:B\to B_n$ be as in Lemma~\ref{lem:unique}. Let $D_n:=(\B_{\dR}^+\otimes_K\D_{\dR}(V))^{\nabla^{\geom}=0}$ in the first case and $D_n:=(K_n[[t,u_1,\dots,u_d]]\otimes_K\D_{\dR}(V))^{\nabla^{\geom}=0}$ in the second case. Let $D:=\wtil{\B}^{\dagger,r}_{\rig}\otimes_{\Q_p} V$ in the first case and $D:=\D^{\dagger,r}_{\rig}(V)$ in the second case. Then, we have
\begin{enumerate}
\item There exists $h\in\N$ such that
\[
t^h B_n\otimes_{\iota_n,B}D \subset D_n \subset B_n\otimes_{\iota_n,B}D
\]
for all $n\ge n(r)$;
\item We define $\iota_n:D\to B_n\otimes_{\iota_n,B}D;x\mapsto 1\otimes x$ and
\[
\mathcal{N}:=\{x\in D;\iota_n(x)\in D_n\text{ for all }n\ge n(r)\}.
\]
Then, $\mathcal{N}$ is a finite free $B$-submodule of $D$, whose rank is equal to $\dim_{\Q_p}V$. Moreover, there exists a canonical isomorphism
\[
B_n\otimes_{\iota_n,B}\mathcal{N}\to D_n
\]
for all $n\ge n(r)$.
\end{enumerate}
\end{prop}
\begin{proof}
\begin{enumerate}
\item Since the inclusion $B_n\subset\B_{\dR}^+$ is faithfully flat by Lemma~\ref{lem:hori}, we have only to prove the assertion after tensoring $\B_{\dR}^+$ over $B_n$. We have the following isomorphisms:
\[
\B_{\dR}^+\otimes_{B_n}B_n\otimes_{\iota_n,B}D\cong\B_{\dR}^+\otimes_{\iota_n,\B^{\dagger,r}}\B^{\dagger,r}\otimes_{\B^{\dagger,r}}D^{\dagger,r}\cong\B_{\dR}^+\otimes_{\iota_n,\B^{\dagger,r}}\B^{\dagger,r}\otimes_{\Q_p}V=\B_{\dR}^+\otimes_{\Q_p}V,
\]
where $D^{\dagger,r}:=\wtil{\B}^{\dagger,r}\otimes_{\Q_p} V$ in the first case and $D^{\dagger,r}:=\D^{\dagger,r}(V)$ in the second case. Since $\B_{\dR}^+\otimes_{B_n}D_n\subset\B_{\dR}^+\otimes_{\Q_p}V$ by assumption and $\B_{\dR}^+\otimes_{B_n}D_n[t^{-1}]\cong \B_{\dR}^+\otimes_K\D_{\dR}(V)[t^{-1}]=\B_{\dR}\otimes_{\Q_p}V$, there exists $h\in\N$ such that
\[
t^h\B_{\dR}^+\otimes_{\Q_p}V\subset \B_{\dR}^+\otimes_{B_n}D_n \subset\B_{\dR}^+\otimes_{\Q_p}V,
\]
which implies the assertion.
\item Since $\mathcal{N}$ is a closed $B$-submodule of $D$ containing $t^hD$, $\mathcal{N}$ is free of rank $\dim_{\Q_p}V$ by Lemma~\ref{lem:Robba}. To prove the second assertion, we have only to prove the canonical map $B_n\otimes_{\iota_n,B}\mathcal{N}\to D_n/tD_n$ is surjective for all $n\ge n(r)$ since $B_n$ is a $t$-adically complete discrete valuation ring. Fix $n$ and let $x\in D_n$. Note that $\pr_{h+1}\circ\iota_n:B\to B_n/t^{h+1}B_n$ is surjective: Indeed, when $B=\B^{\dagger,r}_{\rig,K}$, it follows from Lemma~\ref{lem:surj}. When $B=\wtil{\B}^{\dagger,r}_{\rig}$, it is reduced to the case of $h=0$ and $\pr_1\circ\iota_n=\theta\circ\iota_n:\wtil{\B}^{\dagger,r}_{\rig}\to\C_p$ is surjective since $\wtil{\B}^+\subset\wtil{\B}^{\dagger,r}_{\rig}$. Hence, there exists $y\in D$ such that $\iota_{n}(y)-x\in t^{h+1} B_n\otimes_{\iota_n,B}D\in tD_n$. We put $z:=t_{n,h+1}y\in D$, where $t_{n,h+1}$ is as in Lemma~\ref{lem:partition}. By the property of $t_{\bullet,\bullet}$, we have
\[
\iota_{n}(z)-x=(\iota_{n}(t_{n,h+1})-1)\iota_n(y)+\iota_n(y)-x\in tD_{n}
\]
and for $m\neq n$,
\[
\iota_m(z)\in t^{h+1} B_n\otimes_{\iota_n,B}D\subset tD_n.
\]
These imply $z\in\mathcal{N}$, hence, we obtain the assertion.
\end{enumerate}
\end{proof}

\begin{dfn}
Under the setting of Proposition~\ref{prop:const}, we denote $\mathcal{N}$ by $\wtil{\N}^{\dagger,r}_{\rig}(V)$ in the first case and $\N_{\dR,r}(V)$ in the second case. For a de Rham representation $V$ with arbitrary Hodge-Tate weights, we put $\wtil{\N}^{\dagger,r}_{\rig}(V):=\wtil{\N}^{\dagger,r}_{\rig}(V(-n))(n)$ and $\N_{\dR,r}(V):=\N_{\dR,r}(V(-n))(n)$ for sufficiently large $n\in\N$, which are independent of the choice of $n$. We also put $\wtil{\N}^{\dagger}_{\rig}(V):=\cup_r \wtil{\N}^{\dagger,r}_{\rig}(V)$ and $\N_{\dR,r}(V):=\cup_r \N_{\dR,r}(V)$. Note that for $0<s\le r$, the canonical map $\B^{\dagger,s}_{\rig,K}\otimes_{\B^{\dagger,r}_{\rig,K}}\N_{\dR,r}(V)\to\N_{\dR,s}(V)$ is an isomorphism by Lemma~\ref{lem:unique} and Proposition~\ref{prop:const}. Hence, the canonical morphism $\B^{\dagger}_{\rig,K}\otimes_{\B^{\dagger,r}_{\rig}}\N_{\dR,r}(V)\to\N_{\dR}(V)$ is an isomorphism, in particular, $\N_{\dR}(V)$ is a finite free $\B^{\dagger}_{\rig,K}$-module of rank $\dim_{\Q_p}V$. Since the map $\varphi:\D^{\dagger,r}_{\rig}(V)\to\D^{\dagger,r/p}_{\rig}(V)$ induces a map $\varphi:\N_{\dR,r}(V)\to\N_{\dR,r/p}(V)$ by the formula $\iota_{n+1}\circ\varphi=\iota_n$, $\N_{\dR}(V)$ is stable under the $(\varphi,\Gamma_K)$-action of $\D^{\dagger}_{\rig}(V)$. By a similar reason, $\wtil{\N}_{\rig}^{\dagger}(V)$ is free of rank $\dim_{\Q_p}V$ and is stable under the $(\varphi,G_K)$-action of $\wtil{\B}^{\dagger}_{\rig}\otimes_{\Q_p}V$. Thus, we obtain a $(\varphi,\gk)$-module $\wtil{\N}^{\dagger}_{\rig}(V)$ over $\wtil{\B}^{\dagger}_{\rig}$ and a $(\varphi,\Gamma_K)$-module $\N_{\dR}(V)$ over $\B^{\dagger}_{\rig,K}$.
\end{dfn}

\subsection{Differential action of a $p$-adic Lie group}\label{subsec:Lie}
In this subsection, we recall basic facts on the differential action of a certain $p$-adic Lie group. Throughout this subsection, let $\mathcal{G}$ be a $p$-adic Lie group, which is isomorphic to an open subgroup of $(1+2p\Z_p)\ltimes\Z_p^d$ via a continuous group homomorphism $\eta:\mathcal{G}\hookrightarrow\Z_p^{\times}\ltimes\Z_p^d$. Denote $\eta(\gamma)=(\eta_0(\gamma),\dots,\eta_d(\gamma))\in\Z_p^{\times}\ltimes\Z_p^d$ for $\gamma\in\mathcal{G}$. Let
\[
\mathcal{G}_0:=\{\gamma\in\mathcal{G};\eta_j(\gamma)=0\text{ for all }j>0\},
\]
\[
\mathcal{G}_j:=\{\gamma\in\mathcal{G};\eta_0(\gamma)=1, \eta_{i}(\gamma)=0\text{ for all positive }i\neq j\}
\]
for $1\le j\le d$.

\begin{notation}\label{notation:Lie}
Let $(R,v)$ be a $\Q_p$-Banach algebra and $M$ a finite free $R$-module endowed with $R$-valuation $v$. Assume that $\mathcal{G}$ acts on $R$ and $M$ satisfying
\begin{enumerate}
\item The $\mathcal{G}$-action on $R$ is $\Q_p$-linear and the action of $\mathcal{G}$ on $M$ is $R$-semi-linear;
\item We have $v\circ\gamma(x)=v(x)$ for all $x\in R$ and $\gamma\in\mathcal{G}$;
\item There exists an open subgroup $\mathcal{G}_o\le_o \mathcal{G}$ such that
\[
v((\gamma-1)x)\ge v(x)+v(p)
\]
for all $\gamma\in\mathcal{G}_o$ and $x\in R$.
\item For any $x\in M$, there exists an open subgroup $\mathcal{G}_x\le_o \mathcal{G}_o$ such that
\[
v((\gamma-1)x)\ge v(x)+v(p)
\]
for all $\gamma\in\mathcal{G}_x$.
\end{enumerate}
\end{notation}

\begin{construction}\label{lem:op}
Let notation be as in Notation~\ref{notation:Lie}. We extend the construction of the differential operator $\nabla_V$ in \cite[\S~5.1]{Inv} in this setting. By assumption, there exists an open subgroup $\mathcal{G}_M\le_o\mathcal{G}_o$ such that
\[
v((\gamma-1)x)\ge v(x)+v(p)
\]
for all $x\in M$ and $\gamma\in\mathcal{G}_M$. Hence, we can apply Berger's argument to the $1$-parameter subgroup $\gamma^{\Z_p}$ for $\gamma\in\mathcal{G}_M$. Thus, we can define a continuous $\Q_p$-linear map
\[
\log(\gamma):M\to M;x\mapsto\log(\gamma)(x):=\sum_{n\ge 1}(-1)^{n-1}\frac{(\gamma-1)^n}{n}x
\]
for $\gamma\in\mathcal{G}_M$. Moreover, the operators
\[
\nabla_0(x):=\frac{\log(\gamma)(x)}{\log(\eta_0(\gamma))}\text{ for }\gamma\in \mathcal{G}_M\cap\mathcal{G}_0,
\]
\[
\nabla_j(x):=\frac{\log(\gamma)(x)}{\eta_j(\gamma)}\text{ for }\gamma\in\mathcal{G}_M\cap\mathcal{G}_j
\]
for $1\le j\le d$ are independent of the choice of $\gamma$.

Assume that $N$ satisfies the conditions in Notation~\ref{notation:Lie}. Then, $M\otimes_R N$ satisfies the conditions in Notation~\ref{notation:Lie}, and we have
\[
\log(\gamma\otimes\gamma)=\log(\gamma)\otimes \id_N+\id_M\otimes\log(\gamma)\text{ for }\gamma\in\mathcal{G}_M\cap\mathcal{G}_N
\]
in $\mathrm{End}_{\Q_p}{(M\otimes_RN)}$. By putting $(M,N)=(R,R)$ or $(M,R)$, $\nabla_j:R\to R$ is a continuous derivation and $\nabla_j:M\to M$ is a continuous derivation compatible with $\nabla_j:R\to R$, that is, $\nabla_j(\lambda x)=\nabla_j(\lambda)x+\lambda\nabla_j(x)$ for $\lambda\in R$ and $x\in M$.
\end{construction}

\begin{lem}\label{lem:Lierelation}
Let notation be as in Construction~\ref{lem:op}. In $\mathrm{End}_{\Q_p}(M)$, we have
\[
[\nabla_i,\nabla_j]=-[\nabla_j,\nabla_i]=
\begin{cases}
\nabla_j&\text{ if }i=0,\ 1\le j\le d\\
0&\text{ if } 1\le i,j\le d.
\end{cases}
\]
\end{lem}
\begin{proof}
Since $\mathcal{G}_i$ and $\mathcal{G}_j$ for $1\le i,j\le d$ are commutative, the assertion in the second case is trivial. We prove in the case of $i=0$ and $1\le j\le d$. Fix $x\in M$. We regard $\mathcal{G}$ as a subgroup of $GL_{d+1}(\Z_p)$ as in \S~\ref{subsec:Galois}. For sufficiently small $u_0,u_j\in\Z_p$, put $\gamma_0:=1+u_0E_{1,0}\in\mathcal{G}_0\cap\mathcal{G}_M$, $\gamma_j:=1+u_jE_{1,j}\in\mathcal{G}_j\cap\mathcal{G}_M$, where $E_{1,j}$ is the $(1,j+1)$-th elementary matrix in $M_{d+1}(\Z_p)$. Then, the assertion is equivalent to the equality
\[
\log(\gamma_0)\circ\log(\gamma_j)(x)-\log(\gamma_j)\circ\log(\gamma_0)(x)=\log(1+u_0)\log(\gamma_j)x.
\]
In the group ring $\Q_p[\mathcal{G}]$, we have
\begin{align*}
&\sum_{1\le i\le n}\frac{(-1)^{n-1}}{n}u_0^nu_jE_{1,j}\\
=&\sum_{1\le i\le n}\frac{(-1)^{n-1}}{n}(u_0E_{1,1})^n\sum_{1\le i\le n}\frac{(-1)^{n-1}}{n}(u_jE_{1,j})^n-\sum_{1\le i\le n}\frac{(-1)^{n-1}}{n}(u_jE_{1,j})^n\sum_{1\le i\le n}\frac{(-1)^{n-1}}{n}(u_0E_{1,1})^n.
\end{align*}
After taking the actions of both sides on $x$, the LHS converges to $\log(1+u_0)\log(\gamma_j)(x)$ and the RHS converges to $\log(\gamma_0)\circ\log(\gamma_j)(x)-\log(\gamma_j)\circ\log(\gamma_0)(x)$, which implies the assertion.
\end{proof}

In the following, we will use the Fr\'echet version of Construction~\ref{lem:op}.
\begin{construction}\label{const:diffoperator}
Let $(R,\{w_r\})$ be a Fr\'echet algebra and $M$ a finite free $R$-module endowed with $R$-valuations $\{w_r\}$. Assume that $\mathcal{G}$ acts on $R$ and $M$ and assume that the $\mathcal{G}$-actions on $(\hat{R}_r,w_r)$ and $(\hat{M}_r,w_r)$ for all $r$ satisfy the conditions in Notation~\ref{notation:Lie}, where $\hat{R}_r$ and $\hat{M}_r$ are the completions of $R$ and $M$ with respect to $w_r$. By applying Construction~\ref{lem:op} to each $\hat{R}_r$ and $\hat{M}_r$ and passing to limits, we obtain continuous derivations $\nabla_j:R\to R$ and $\nabla_j:M\to M$ for $0\le j\le d$, which are compatible with $\nabla_j:R\to R$, satisfying
\[
[\nabla_0,\nabla_j]=\nabla_j\text{ for }1\le j\le d,\ [\nabla_i,\nabla_j]=0\text{ for }1\le i,j\le d.
\]
Thus, the actions of $\nabla_0,\dots\nabla_d$ give rise to a differential action of the Lie algebra $\mathrm{Lie}(\mathcal{G})\cong\Q_p\ltimes\Q_p^d$.
\end{construction}

\subsection{Differential action and differential conductor of $\N_{\dR}$}\label{subsec:diffaction}
In \S~\ref{subsec:const}, we have constructed $\N_{\dR}(V)$ for de Rham representations $V$ as a $(\varphi,\Gamma_K)$-module. The aim of this subsection is to endow $\N_{\dR}(V)$ with a structure of $(\varphi,\nabla)$-module in the sense of Definition~\ref{dfn:phinabla} by using the results in \S~\ref{subsec:Lie}. As a consequence, we can define the differential Swan conductor of $\N_{\dR}(V)$ (Definition~\ref{dfn:diffswanndr}). Throughout this subsection, let $V$ denote a $p$-adic representation of $\gk$.

\begin{lem}\label{lem:cont}
There exists an open normal subgroup $\Gamma_K^o\le_o \Gamma_K$ and $r_K>0$ such that for all $0<r\le r_K$, there exists $c_r>0$ such that
\[
w_r((1-\gamma)x)\ge w_r(x)+c_r,\forall x\in\mathbb{B}_{K}^{\dagger,r},\forall\gamma\in \Gamma_K^o.
\]
\end{lem}
\begin{proof}
We may assume $x\in\A^{\dagger,r}_K$. Recall that the ring $\Lambda_{m,\oo_K}^{(i)}$ for $m\in\N$ (\cite[p.~82]{AB}) is a subring of $\wtil{\A}^{\dagger,r}_K$ containing $\A^{\dagger,r}_K$. Hence, we have only to prove a similar assertion for $\Lambda_{m,\oo_K}^{(i)}$. Then, the assertion follows from \cite[Prospotion~4.22]{AB} if we define $\Gamma_K^o$ as the closed subgroup of $\Gamma_K$ topologically generated by $\{\gamma_j^{p^{m}};0\le j\le d\}$ for sufficiently large $m$.
\end{proof}

By shrinking $\Gamma_K^o$, if necessary, we may assume that $\Gamma_K^o$ is an open subgroup of $(1+2p\Z_p)\ltimes\Z_p^d$ as in \S~\ref{subsec:Galois}. In the rest of this paper, we assume that $r_0$ in Notation~\ref{notation:tacitly} is sufficiently small such that $r_0\le r_K$.

\begin{lem}\label{lem:cont2}
For $x\in\wtil{\B}^{\dagger,r}$ and $c>0$, there exists an open subgroup $U_{x,c}\le_o\gk$ such that
\[
w_r((g-1)x)\ge c\text{ for all }g\in U_{x,c}.
\]
\end{lem}
\begin{proof}
We may assume that $x$ is of the form $[\bar{x}]$ with $\bar{x}\in\wtil{\E}$. Indeed, if we write $x=\sum_{k\gg -\infty}p^k[x_k]$ with $x_k\in\wtil{\E}$, then, by definition, there exists $N$ such that $w_r(p^k[x_k])\ge c$ for all $k\ge N$. We choose $U_{x,c}$ such that $w_r((g-1)(p^k[x_k]))\ge c$ for all $k\le N$ and all $g\in U_{x,c}$. Then, $U_{x,c}$ satisfies the condition.

Let $x=[\bar{x}]$ with $\bar{x}\in \wtil{\E}^{\times}$. Since the action of $\gk$ on $\wtil{\E}$ is continuous, there exists $U_{x,c}\le_o \gk$ such that $v_{\wtil{\E}}((g-1)\bar{x})\ge p^{\lfloor c\rfloor}c/r\ (>0)$ for all $g\in U_{x,c}$. We prove that $U_{x,c}$ satisfies the desired condition. We can write
\[
(g-1)[\bar{x}]=[(g-1)\bar{x}]+\sum_{k\ge 1}p^k[x_k]
\]
for some $x_k\in\wtil{\E}$. Since
\[
[\bar{x}]\left(\left[\frac{(g-1)\bar{x}}{\bar{x}}\right]+1\right)=(g(\bar{x}),-x_1^p,-x_2^{p^2},\dots),
\]
$x_k^{p^k}/\bar{x}$ is written as a $\Z$-coefficient polynomial without constant term of $(g-1)\bar{x}/\bar{x}$. Indeed, let $S_m\in\Z[X_0,\dots,X_m,Y_0,\dots,Y_m]$ for $m\in\N$ be a family of polynomials defining the additive law of the ring of Witt vectors (\cite[$n^o 3$, \S~1, IX]{Bou9}). Recall that $S_m$ is homogeneous of degree $p^m$, where $\deg(X_i)=\deg(Y_i)=p^i$. Since $S_0=X_0+Y_0$ and $\sum_{0\le i\le m}p^iS_i^{p^{m-i}}=\sum_{0\le i\le m}p^iX_i^{p^{m-i}}+\sum_{0\le i\le m}p^iY_i^{p^{m-i}}$, the coefficients of $X_0^{p^m}$ and $Y_0^{p^m}$ in $S_m$ for $m\ge 1$ are equal to zero, which implies the assertion. Hence, for $n\in\N$, we have
\[
v_{\wtil{\E}}^{\le n}((g-1)[\bar{x}])=\inf_{1\le k\le n}{\left\{v_{\wtil{\E}}((g-1)\bar{x}),v_{\wtil{\E}}(x_k)\right\}}\ge\inf_{1\le k\le n}{\left\{v_{\wtil{\E}}((g-1)\bar{x}),\frac{1}{p^k}v_{\wtil{\E}}((g-1)\bar{x})\right\}}=\frac{1}{p^n}v_{\wtil{\E}}((g-1)\bar{x}).
\]
Note that $v_{\wtil{\E}}^{\le n}((g-1)[\bar{x}])=\infty$ for $n\in\Z_{<0}$. Hence, $w_r((g-1)[\bar{x}])=\inf_{n\in\N}(rv_{\wtil{\E}}^{\le n}((g-1)[\bar{x}])+n)\ge\inf{(r\cdot \frac{1}{p^{\lfloor c\rfloor}}v_{\wtil{\E}}((g-1)\bar{x}),\lfloor c\rfloor)}\ge c$, which implies the assertion.
\end{proof}

\begin{lem}\label{lem:cont3}
Let $\{e_i\}$ be a $\B^{\dagger,r}_K$-basis of $\D^{\dagger,r}(V)$. We endow $\D^{\dagger,r}_{\rig}(V)$ with valuations $\{w_s\}_{0<s\le r}$ compatible with $\{w_s\}_{0<s\le r}$ associated to $\{e_i\}$. Then, the actions of $\Gamma_K^o$ on $\B^{\dagger,r}_{\rig,K}$ and $\D^{\dagger,r}_{\rig}(V)$ satisfy the conditions in Notation~\ref{notation:Lie}.
\end{lem}
\begin{proof}
The conditions~(i) and (ii) follow by definition. The condition~(iii) follows from the formula $\gamma^p-1=\sum_{1\le i\le p}\binom{p}{i}(\gamma-1)^i$ and Lemma~\ref{lem:cont}. To prove the condition~(iv), we may assume $x\in\D^{\dagger,r}(V)$. We choose a lattice $T$ of $V$ stable by $\gk$. Let $\{f_i\}$ be a basis of $T$ and we endow $\wtil{\B}^{\dagger,r}\otimes_{\Q_p} V$ with valuation $\{w'_s\}_{0<s\le r}$ compatible with $\{w_s\}_{0<s\le r}$ associated to the $\wtil{\B}^{\dagger,r}$-basis $\{1\otimes f_i\}$. By a canonical isomorphism $\B^{\dagger,r}\otimes_{\B^{\dagger,r}_K}\D^{\dagger,r}(V)\cong\B^{\dagger,r}\otimes_{\Q_p} V$ by Theorem~\ref{thm:oc}, we regard $\{1\otimes e_i\}$ as a $\wtil{\B}^{\dagger,r}$-basis of $\B^{\dagger,r}\otimes_{\Q_p} V$. Then, $w_s$ is equivalent to $w'_s$, therefore, we have only to prove that for any $x\in\wtil{\B}^{\dagger,r}\otimes_{\Q_p} V$ and $0<s\le r$, there exists an open subgroup $G_{K,s,x}^o\le_o\gk$ such that $w'_s((g-1)x)\ge w'_s(x)+w'_s(p)$ for all $g\in G_{K,s,x}^o$. We may assume that $x$ is of the form $\lambda\otimes v$ for $\lambda\in\wtil{\B}^{\dagger,r}$ and $v\in T$. Since the action of $\gk$ on $T$ is continuous, there exists an open subgroup $U\le_o \gk$ such that $U$ acts on $T/pT$ trivially. Applying Lemma~\ref{lem:cont2} by regarding $\lambda\in\wtil{\B}^{\dagger,s}$, there exists an open subgroup $U'\le_o\gk$ such that $w_s((g-1)\lambda)\ge w_s(\lambda)+w_s(p)$ for all $g\in U'$. If we put $G_{K,s,x}^o:=U\cap U'$, then the assertion follows from the formula
\[
(g-1)(\lambda\otimes v)=(g-1)(\lambda)\otimes g(v)+\lambda\otimes (g-1)v.
\]
\end{proof}

\begin{dfn}
By Lemma~\ref{lem:cont3}, we can apply Construction~\ref{const:diffoperator} to $\mathcal{G}=\Gamma_K$, $R=\B^{\dagger,r}_{\rig,K}$ and $M=\D^{\dagger,r}_{\rig}(V)$. Thus, we obtain continuous differentials operators $\nabla_j$ on $\D^{\dagger,r}_{\rig}(V)$ for $0\le j\le d$. The operator $\nabla_j$ induces a continuous differential operator on $\D^{\dagger}_{\rig}(V)$, which is denoted by $\nabla_j$ again. Since the actions of $\Gamma_K$ and $\varphi$ are commutative, $\nabla_j$ commutes with $\varphi$ by definition.
\end{dfn}

For a while, let $V=\Q_p$ and we regard $\D^{\dagger,r}_{\rig}(\Q_p)$ as $\B^{\dagger,r}_{\rig,K}$. Then, $\nabla_j$ can be regarded as a continuous derivation on $\B^{\dagger,r}_{\rig,K}$. In the following, we will describe this derivation explicitly.

\begin{construction}\label{const:derdR}
As in \cite[Propostion~4.3]{AB2}, the action of $\Gamma_{K}$ on $K_n[[t,u_1,\dots,u_d]]$ induces $K_n$-linear differentials
\[
\wtil{\nabla}_0:=\frac{\log(\gamma_0)}{\log(\eta_0(\gamma_0))}=t(1+\pi)\frac{\partial}{\partial\pi},
\]
\[
\wtil{\nabla}_j:=\frac{\log(\gamma_j)}{\eta_j(\gamma_j)}=-t[\wtil{t}_j]\frac{\partial}{\partial u_j}\text{ for }1\le j\le d
\]
for any sufficiently small $\gamma_0\in\Gamma_{K,0}$ and $\gamma_j\in\Gamma_{K,j}$. Note that these are continuous with respect to the canonical topology. Since the action of $\Gamma_{K}$ commutes with $\nabla^{\geom}$ by definition, $\wtil{\nabla}_j$ acts on $K_n[[t,u_1,\dots,u_d]]^{\nabla}$.
\end{construction}

We assume $K=\wtil{K}$ for a while. By an isomorphism $\A^{\dagger,r}_{K}\cong\oo((\pi))^{\dagger,r}$, we have derivations on $\A_{K}^{\dagger,r}$ (see \S~\ref{subsec:diffconductor})
\[
\partial_0:=\frac{\partial}{\partial\pi},\ \partial_1:=\frac{\partial}{\partial[\wtil{t}_1]},\dots,\ \partial_d:=\frac{\partial}{\partial[\wtil{t}_d]},
\]
which are continuous with respect the Fr\'echet topology defined by $\{w_s\}_{0<s\le r}$. By passing to completion, we obtain continuous derivations $\partial_j:\B^{\dagger,r}_{\rig,K}\to \B^{\dagger,r}_{\rig,K}$ for $0\le j\le d$. The derivation $\partial_j$ also extends to a derivation $\partial_j:\B^{\dagger}_{\rig,K}\to \B^{\dagger}_{\rig,K}$. By Lemma~\ref{lem:dense}, we may regard $\B^{\dagger,r}_{\rig,K}$ as a dense subring of $K_n[[t,u_1,\dots,u_d]]^{\nabla}$ by $\iota_n$. Hence, we can extend any continuous derivation $\partial$ on $\B^{\dagger,r}_{\rig,K}$ to a continuous derivation on $K_n[[t,u_1,\dots,u_d]]^{\nabla}$, which is denoted by $\iota_n(\partial)$. Note that we have a formula
\begin{equation}\label{eq:diff}
\iota_n(\partial)(\iota_n(x))=\iota_n(\partial(x))\text{ for }x\in\B^{\dagger,r}_{\rig,K}.
\end{equation}

\begin{lem}\label{lem:description}
For $n\ge n(r)$, we have
\[
\iota_n(t(1+\pi)\partial_0)=\wtil{\nabla}_0,\ \iota_n(t[\wtil{t}_j]\partial_j)=\wtil{\nabla}_j\text{ for }1\le j\le d.
\]
\end{lem}
\begin{proof}
Let $1\le j\le d$ and put $\delta_0:=\iota_n(t(1+\pi)\partial_0)-\wtil{\nabla}_0$ and $\delta_j:=\iota_n(t[\wtil{t}_j]\partial_j)-\wtil{\nabla}_j$. Let $f:K_n[[t,u_1,\dots,u_d]]\to K_n[[t,u_1,\dots,u_d]]^{\nabla}$ be the map defined in the proof of Lemma~\ref{lem:hori}, which is continuous by Lemma~\ref{lem:contmap}. Since $f$ induces a surjection on the residue fields by definition, $f(K_n[t])$ is a dense subring of $K_n[[t,u_1,\dots,u_d]]^{\nabla}$ by Lemma~\ref{lem:hori} and Lemma~\ref{lem:comparisontop}. Hence, we have only to prove $\delta_0\circ f(K_n[t])=\delta_j\circ f(K_n[t])=0$. We regard $\delta_0\circ f|_{K_n},\ \delta_j\circ f|_{K_n}\in \Der_{\cont}(K_n,K_n[[t,u_1,\dots,u_d]]^{\nabla})$, which is isomorphic to $\Hom_{K_n}(\hat{\Omega}^1_{K_n},K_n[[t,u_1,\dots,u_d]]^{\nabla})$ by Lemma~\ref{cor:kahler}. Since $\hat{\Omega}^1_{K_n}\cong K_n\otimes_K\hat{\Omega}^1_K$ has a $K_n$-basis $\{dt_i;1\le i\le d\}$ and we have $f(t)=t$ and $f(t_i)=[\wtil{t}_i]$ by definition, we have only to prove $\delta_0(t)=\delta_j(t)=0$ and $\delta_0([\wtil{t}_i])=\delta_j([\wtil{t}_i])=0$ for all $1\le i\le d$. By using the formula (\ref{eq:diff}),
\[
\iota_n(t(1+\pi)\partial_0)(t)=t=\wtil{\nabla}_0(t),\ \iota_n(t(1+\pi)\partial_0)[\wtil{t}_i]=0,
\]
\[
\iota_n(t[\wtil{t}_j]\partial_j)(t)=0=\wtil{\nabla}_j(t),\ \iota_n(t[\wtil{t}_j]\partial_j)[\wtil{t}_i]=\delta_{ij}t[\wtil{t}_j]
\]
for all $1\le i\le d$. Since $(\partial/\partial u_j)[\wtil{t}_i]=-(\partial/\partial u_j)u_i=-\delta_{ij}$ for all $1\le i\le d$, we obtain the assertion.
\end{proof}

In the rest of this section, we drop the assumptions $K=\wtil{K}$ and $V=\Q_p$.

\begin{cor}\label{lem:welldef}
The derivation
\[
d':\B^{\dagger}_{\rig,K}\to\Omega^1_{\B^{\dagger}_{\rig,K}};x\mapsto\nabla_0(x)\frac{1}{t(1+\pi)}d\pi+\sum_{1\le j\le d}\nabla_j(x)\frac{1}{t}d[\wtil{t}_j]
\]
coincides with the canonical derivation $d:\B^{\dagger}_{\rig,K}\to\Omega^1_{\B^{\dagger}_{\rig,K}}$.
\end{cor}
\begin{proof}
Since the canonical map $\B^{\dagger}_{\rig,\wtil{K}}\to\B^{\dagger}_{\rig,K}$ is finite \'etale by \cite[Proposition~2.4.10]{Doc}, we can reduce to the case $K=\wtil{K}$. Let notation be as in Lemma~\ref{lem:description}. Obviously, $\nabla_j$ extends to $\wtil{\nabla}_j$ by passing to completion. Since $\iota_n$ is injective, we have
\[
\nabla_0=t(1+\pi)\partial_0,\ \nabla_j=t[\wtil{t}_j]\partial_j\text{ for }1\le j\le d.
\]
as derivations of $\B^{\dagger,r}_{\rig,K}$ by Lemma~\ref{lem:description}, which implies the assertion.
\end{proof}

\begin{lem}\label{lem:ndr}
Let $V\in\rep_{\dR}(G_K)$.
\begin{enumerate}
\item We have $\nabla_j(\N_{\dR}(V))\subset t\N_{\dR}(V)$ for all $0\le j\le d$; We put $\nabla'_j:=t^{-1}\nabla_j$, which is a continuous differential operator on $\N_{\dR}(V)$.
\item We have
\[
[\nabla'_i,\nabla'_j]=0
\]
for all $0\le i,j\le d$.
\item We have
\[
\nabla'_j\circ \varphi=p\varphi\circ\nabla'_j
\]
for all $0\le j\le d$.
\end{enumerate}
\end{lem}
\begin{proof}
\begin{enumerate}
\item By Tate twist, we may assume that the Hodge-Tate weights of $V$ are sufficiently small. Let notation be as in Construction~\ref{const:derdR} and Proposition~\ref{prop:const} (with $B=\B_{\rig,K}^{\dagger,r}$). By regarding $t\N_{\dR,r}(V)$ and $t\D_{\dR}(V)$ as $\N_{\dR,r}(V(1))$ and $\D_{\dR}(V(1))$, we have only to prove that $\iota_n(\nabla_j(x))\in tD_n$ for all $n\ge n(r)$ and $x\in\N_{\dR,r}(V)$. For sufficiently small $\gamma_j\in\Gamma_{K,j}$, we have $\iota_n\circ\log(\gamma_j)(x)=\log(\gamma_j)(\iota_n(x))$ and $\iota_n(x)\in D_n\subset B_n\otimes_K\D_{\dR}(V)$. Since $\Gamma_K$ acts on $\D_{\dR}(V)$ trivially, $\log(\gamma_j)$ acts on $B_n\otimes_K\D_{\dR}(V)$ as $\log(\gamma_j)\otimes 1$. Since $\log(\gamma_j)(B_n)\subset tB_n$ (see Construction~\ref{const:derdR}), we have $\iota_n\circ\log(\gamma_j)(x)\in (B_n\otimes_K\D_{\dR}(V(1)))^{\nabla^{\geom}=0}=tD_n$, which implies the assertion.
\item It follows from a straightforward calculation using Lemma~\ref{lem:Lierelation}, $\nabla_0(t)=t$, and $\nabla_i(t)=\nabla_j(t)=0$.
\item Since $\nabla_j$ commutes with $\varphi$, we have $t\nabla'_j\circ\varphi=\nabla_j\circ\varphi=\varphi\circ\nabla_j=\varphi(t)\varphi\circ\nabla'_j=pt\varphi\circ\nabla'_j$. By dividing by $t$, we obtain the assertion since $\N_{\dR}(V)$ is torsion free.
\end{enumerate}
\end{proof}

\begin{dfn}\label{dfn:diffswanndr}
Let notation be as in Lemma~\ref{lem:ndr}. For $V\in \rep_{\dR}(\gk)$, we put
\[
\nabla:\N_{\dR}(V)\to\N_{\dR}(V)\otimes_{\B_{\rig,K}^{\dagger}}\Omega^1_{\B_{\rig,K}^{\dagger}};x\mapsto\nabla'_0(x)\otimes \frac{1}{1+\pi}d\pi+\sum_{1\le j\le d}\nabla'_j(x)\otimes d[\wtil{t}_j],
\]
which defines a $\nabla$-structure on $\N_{\dR}(V)$ by Corollary~\ref{lem:welldef}. Moreover, this $\nabla$-structure is compatible with the $\varphi$-structure on $\N_{\dR}(V)$ by Lemma~\ref{lem:ndr}~(iii) and $\varphi((1+\pi)^{-1}d\pi)=p(1+\pi)^{-1}d\pi$ and $\varphi(d[\wtil{t}_j])=pd[\wtil{t}_j]$. Thus, $\N_{\dR}(V)$ is endowed with $(\varphi,\nabla)$-module structure and we obtain the differential Swan conductor $\sw^{\nabla}(\N_{\dR}(V))$ of $\N_{\dR}(V)$. Note that the slope filtration of $\N_{\dR}(V)$ as a $(\varphi,\nabla)$-module (Theorem~\ref{thm:slopefil}) is $\Gamma_K$-stable by the commutativity of $\Gamma_K$- and $\varphi$-actions, and the uniqueness of the slope filtration (\cite[Theorem~6.4.1]{Sw}).
\end{dfn}

\subsection{Comparison of pure objects}\label{subsec:pure}
In this subsection, we will study ``pure'' objects in various categories.

\begin{notation}
Let $G$ be a topological group and $R$ a topological ring on which $G$ acts. Let $\phi:R\to R$ be a continuous ring homomorphism commuting with the action of $G$. A $(\phi,G)$-module over $R$ is a finite free $R$-module with continuous and semi-linear action of $G$ and a semi-linear endomorphism $\phi$, which are commutative. We denote the category of $(\phi,G)$-modules over $R$ by $\Mod_R(\phi,G)$. The morphisms in $\Mod_R(\phi,G)$ consist of $R$-linear maps commuting with $\phi$ and $G$.
\end{notation}

\begin{dfn}[{\cite[Definition~3.2.1]{Bpair}}]
Let $h\ge 1$ and $a\in\Z$ be relatively prime integers. Let $\rep_{a,h}(\gk)$ be the category, whose objects are $V_{a,h}\in\rep_{\Q_{p^h}}(G_K)$ endowed with a semi-linear Frobenius action $\varphi:V_{a,h}\to V_{a,h}$ such that $\varphi^h=p^a$, commuting with the $\gk$-action. The morphisms of this category are $\Q_{p^h}$-linear maps, commuting with $(\varphi,G_K)$-actions. Note that when $h=1$ and $a=0$, $\rep_{a,h}(\gk)=\rep_{\Q_p}(\gk)$.

Let $s:=a/h\in\Q$. We denote by $D_{[s]}$ the $\Q_p$-vector space $\oplus_{1\le i\le h}\Q_pe_i$ endowed with trivial $\gk$-action and $\varphi$-actions by $\varphi(e_i):=e_{i+1}$ if $i\neq h$ and $\varphi(e_h):=p^ae_1$. Then, $\Q_{p^h}\otimes_{\Q_p}D_{[s]}$ belongs to $\rep_{a,h}(\gk)$.
\end{dfn}

\begin{dfn}
For $s\in\Q$, we define
\[
\Mod_{\btrigdag}^s(\varphi,\gk),\ \Mod_{\brigdagk}^s(\varphi,\Gamma_K),\ \Mod_{\btdag}^s(\varphi,\gk),\ \Mod_{\bdagk}^s(\varphi,\Gamma_K)
\]
be the full subcategories of $\Mod_{\btrigdag}(\varphi,\gk)$, $\Mod_{\brigdagk}(\varphi,\Gamma_K)$, $\Mod_{\btdag}^s(\varphi,\gk)$ and $\Mod_{\bdagk}^s(\varphi,\Gamma_K)$, whose objects are pure of slope $s$ as $\varphi$-modules.
\end{dfn}

\begin{lem}\label{lem:plusdagger}
\begin{enumerate}
\item For any $r>0$, there exists a canonical injection
\[
\wtil{\B}^{\nabla+}_{\rig}\to\wtil{\B}^{\dagger,r}_{\rig},
\]
which is $(\varphi,\gk)$-equivariant; In the following, we regard $\wtil{\B}^{\nabla+}_{\rig}$ as a subring of $\wtil{\B}^{\dagger,r}_{\rig}$ and we endow $\wtil{\B}^{\nabla+}_{\rig}$ with a Fr\'echet topology induced by the family of valuations $\{w_r\}_{r>0}$.
\item For $h\in\N_{>0}$,
\[
(\wtil{\B}^{\nabla+}_{\rig})^{\varphi^h=1}=(\wtil{\B}^{\dagger,r}_{\rig})^{\varphi^h=1}=\Q_{p^h}.
\]
\end{enumerate} 
\end{lem}
\begin{proof}
By definition, $\wtil{\B}^{\nabla+}_{\rig}$ and $\wtil{\B}^{\dagger,r}_{\rig}$ depend only on $\C_p$ and do not depend on $K$. By regarding $\C_p$ as the $p$-adic completion of the algebraic closure of $K^{\mathrm{pf}}$, we can reduce to the perfect residue field case; The assertion (i) follows from (see \cite[Exemple~2.8~(2), Definition~2.16]{Inv}). The assertion (ii) for $\wtil{\B}^{\nabla+}_{\rig}$ is due to Colmez (\cite[Lemma~6.2]{Ohk}), and (ii) for $\wtil{\B}^{\dagger,r}_{\rig}$ is a consequence of \cite[Proposition~3.2]{Inv}.
\end{proof}

\begin{dfn}
For $s\in\Q$, an object $M\in\Mod_{\brig}(\varphi,\gk)$ is said to be pure of slope $s$ if $M$ is isomorphic to $(\brig\otimes_{\Q_p}D_{[s]})^m$ as a $\varphi$-module for some $m\in\N$. Denote by $\Mod^s_{\brig}(\varphi,\gk)$ the category of $(\varphi,\gk)$-modules over $\brig$, which are pure of slope $s$.
\end{dfn}

\begin{lem}\label{lem:diffpure}
Let notation be as in Notation~\ref{notation:ocpower} and Notation~\ref{dfn:phinabla}. For $s\in\Q$, the following forgetful functor is fully faithful:
\[
\Mod^s_{\mathcal{R}}(\varphi,\nabla)\to\Mod^s_{\mathcal{R}}(\varphi).
\]
\end{lem}
\begin{proof}
We consider the following commutative diagram
\[\xymatrix{
\Mod^{s}_{\Gamma[p^{-1}]}(\varphi,\nabla)\ar[r]^{\alpha_1}&\Mod^{s}_{\Gamma[p^{-1}]}(\varphi)\\
\Mod^{s}_{\Gamma^{\dagger}[p^{-1}]}(\varphi,\nabla)\ar[r]^{\alpha_2}\ar[u]^{\beta_1}\ar[d]_{\beta_2}&\Mod^{s}_{\Gamma^{\dagger}[p^{-1}]}(\varphi)\ar[d]^{\gamma_2}\ar[u]_{\gamma_1}\\
\Mod^{s}_{\mathcal{R}}(\varphi,\nabla)\ar[r]^{\alpha_3}&\Mod^{s}_{\mathcal{R}}(\varphi),
}\]
where $\alpha_{\bullet}$ is a forgetful functor, $\beta_{\bullet}$ and $\gamma_{\bullet}$ are base change functors. We first note that $\gamma_1$ (resp. $\gamma_2$) is fully faithful (resp. an equivalence) by \cite[Theorem~6.3.3~(a)]{Doc} (resp. \cite[Theorem~6.3.3~(b)]{Doc}). Let $M,N\in\Mod^s_{\Gamma^{\dagger}[p^{-1}]}(\varphi,\nabla)$ and denote by $\hat{M},\hat{N}$ the base changes of $M,N$ by the canonical map $\Gamma^{\dagger}[p^{-1}]\to\Gamma[p^{-1}]$. Then, we have
\[
\Hom_{\Mod^s_{\Gamma^{\dagger}[p^{-1}]}(\varphi,\nabla)}(M,N)=\Hom_{\Gamma^{\dagger}[p^{-1}]}(M,N)^{\varphi=1,\nabla=0}=\Hom_{\Gamma[p^{-1}]}(\hat{M},\hat{N})^{\varphi=1,\nabla=0},
\]
where the first equality follows by definition and the second equality follows from the fully faithfulness of $\gamma_1$. Therefore, $\beta_1$ is fully faithful. For the same reason, the fully faithfulness of $\gamma_2$ implies that of $\beta_2$. Note that $\alpha_1$ is an equivalence in the \'etale case, i.e., $s=0$ (\cite[Proposition~3.2.8]{Sw}). Let $M,N\in\Mod^s_{\Gamma[p^{-1}]}(\varphi,\nabla)$. Since $\Hom_{\Gamma[p^{-1}]}(M,N)\cong M\spcheck\otimes_{\Gamma[p^{-1}]}N$ can be regarded as an \'etale $(\varphi,\nabla)$-module over $\Gamma[p^{-1}]$, where $M\spcheck$ denotes the dual of $M$, we have
\[
\Hom_{\Mod^s_{\Gamma[p^{-1}]}(\varphi,\nabla)}(M,N)=\Hom_{\Gamma[p^{-1}]}(M,N)^{\varphi=1,\nabla=0}=\Hom_{\Gamma[p^{-1}]}(M,N)^{\varphi=1}=\Hom_{\Mod_{\Gamma[p^{-1}]}^s(\varphi)}(M,N),
\]
where the first and third equalities follow by definition and the second equality follows from the fully faithfulness of $\alpha_1$ in the \'etale case. Therefore, $\alpha_1$ is an equivalence. Since $\alpha_1$, $\beta_1$ and $\gamma_1$ are fully faithful, so is $\alpha_2$. Since $\alpha_2$, $\beta_2$ and $\gamma_2$ are fully faithful, so is $\alpha_3$.
\end{proof}

\begin{lem}\label{lem:pure}
Let $s\in\Q$ and let $h\in\N_{\ge 1},\ a\in\Z$ be the relatively prime integers such that $s=a/h$.
\begin{enumerate}
\item There exist equivalences of categories
\begin{align*}
\wtil{\D}_{\rig}^{\nabla+}&:\rep_{a,h}(\gk)\to \Mod^s_{\brig}(\varphi,\gk);V_{a,h}\mapsto\brig\otimes_{\Q_{p^h}}V_{a,h},\\
\wtil{\D}_{\rig}^{\dagger}&:\rep_{a,h}(\gk)\to \Mod^s_{\btrigdag}(\varphi,\gk);V_{a,h}\mapsto \btrigdag\otimes_{\Q_{p^h}}V_{a,h},\\
\mathbb{D}_{\rig}^{\dagger}&:\rep_{a,h}(\gk)\to \Mod^s_{\brigdagk}(\varphi,\Gamma_K);V_{a,h}\mapsto \brigdagk\otimes_{\bdagk}(\bdag\otimes_{\Q_{p^h}}V_{a,h})^{H_K},\\
\widetilde{\mathbb{D}}^{\dagger}&:\rep_{a,h}(\gk)\to \Mod^s_{\btdag}(\varphi,\gk);V_{a,h}\mapsto \btdag\otimes_{\Q_{p^h}}V_{a,h},\\
\mathbb{D}^{\dagger}&:\rep_{a,h}(\gk)\to \Mod^s_{\bdagk}(\varphi,\Gamma_K);V_{a,h}\mapsto (\bdag\otimes_{\Q_{p^h}}V_{a,h})^{H_K}.
\end{align*}
More precisely, quasi-inverses of $\widetilde{\mathbb{D}}_{\rig}^{\nabla+}$, $\widetilde{\mathbb{D}}_{\rig}^{\dagger}$ and $\widetilde{\mathbb{D}}^{\dagger}$ are given by $M\mapsto M^{\varphi^h=p^a}$.
\item We denote by $\alpha_i$ for $1\le i\le 5$ the following canonical morphisms of rings:
\[\xymatrix{
\bdagk\ar[r]^{\alpha_1}\ar[d]^{\alpha_2}&\brigdagk\ar[d]^{\alpha_4}&\\
\brigdag\ar[r]^{\alpha_3}&\btrigdag&\brig\ar[l]_{\alpha_5},
}\]
where the left square is commutative. Then, $\alpha_i$'s induce the following base change functors $\alpha_{\bullet}^*$:
\[\xymatrix{
\Mod^s_{\bdagk}(\varphi,\Gamma_K)\ar[r]^{\alpha_1^*}\ar[d]^{\alpha_2^*}&\Mod^s_{\brigdagk}(\varphi,\Gamma_K)\ar[d]^{\alpha_4^*}&\\
\Mod^s_{\btdag}(\varphi,\gk)\ar[r]^{\alpha_3^*}&\Mod^s_{\btrigdag}(\varphi,\gk)&\Mod^s_{\brig}(\varphi,\gk)\ar[l]_{\alpha_5^*},
}\]
where the left square is commutative. Moreover, the functors $\alpha_{\bullet}^*$'s are compatible with the functor defined in (i) , i.e., $\alpha_1^*\circ\mathbb{D}^{\dagger}=\mathbb{D}^{\dagger}_{\rig}$ etc. In particular, $\alpha_{\bullet}^*$'s are equivalences.
\end{enumerate}
\end{lem}
\begin{proof}
\begin{enumerate}
\item We prove the assertion for $\wtil{\D}_{\rig}^{\nabla+}$. Let $\mathcal{D}:=\wtil{\D}_{\rig}^{\nabla+}$ and $\mathcal{V}$ the converse direction functor as above. Let $V\in\rep_{a,h}(G_K)$. Then, there exists a functorial morphism $V\to\mathcal{V}\circ\mathcal{D}(V)$, which is bijective by Lemma~\ref{lem:plusdagger}~(ii). Hence, we have a natural equivalence $\mathcal{V}\circ \mathcal{D}\simeq\id$. Let $M\in \Mod^s_{\brig}(\varphi,\gk)$. Then, there exists a functorial morphism $\mathcal{D}\circ\mathcal{V}(M)\to M$, which is bijective by an isomorphism $M\cong (\brig\otimes_{\Q_p}D_{[s]})^m$ as $\varphi$-modules and Lemma~\ref{lem:plusdagger}~(ii). Hence, we have a natural equivalence $\mathcal{D}\circ\mathcal{V}\simeq\id$.

The assertions for $\wtil{\D}_{\rig}^{\dagger}$ and $\wtil{\D}^{\dagger}$ follow similarly: Instead of using an isomorphism $M\cong (\brig\otimes_{\Q_p}D_{[s]})^m$, we use Kedlaya's Dieudonn\'e-Manin decomposition theorems over $\wtil{\B}_{\rig}^{\dagger}$ (\cite[Proposition~4.5.3, 4.5.10 and Definition~4.6.1]{Doc}) and $\wtil{\B}^{\dagger}$ (\cite[Theorem~6.3.3~(b)]{Doc}), which assert that any object $M\in\Mod^s_{\wtil{\B}^{\dagger}_{\rig}}(\varphi)$ (resp. $\Mod^s_{\wtil{\B}^{\dagger}}(\varphi)$) is isomorphic to a direct sum of $\wtil{\B}^{\dagger}_{\rig}\otimes_{\Q_p}D_{[s]}$ (resp. $\wtil{\B}^{\dagger}\otimes_{\Q_p}D_{[s]}$).

We prove the assertion for $\D^{\dagger}$. For $M\in \Mod^s_{\bdagk}(\varphi,\Gamma_K)$, let $\mathcal{V}(M):=(\bdag\otimes_{\bdagk}M)^{\varphi^h=p^a}$. We will check that $\mathcal{V}$ gives a quasi-inverse of $\D^{\dagger}$. Let $V_{a,h}\in\rep_{a,h}(\gk)$. By forgetting the action of $\varphi$ on $V_{a,h}$ and applying Theorem~\ref{thm:oc} to $V=V_{a,h}$, we obtain a canonical bijection $\B^{\dagger}\otimes_{\B^{\dagger}_K}\D^{\dagger}(V_{a,h})\to\B^{\dagger}\otimes_{\Q_{p^h}}V_{a,h}$. Since this map is $\varphi$-equivariant, we have canonical isomorphisms $\mathcal{V}\circ \D^{\dagger}(V_{a,h})\cong (\B^{\dagger})^{\varphi^h=1}\otimes_{\Q_{p^h}}V_{a,h}\cong V_{a,h}$ by Lemma~\ref{lem:plusdagger}~(ii). Thus, we obtain a natural equivalence $\mathcal{V}\circ \D^{\dagger}\simeq \id$. We prove $\D^{\dagger}\circ \mathcal{V}\simeq\id$. Let $M\in \Mod^s_{\bdagk}(\varphi,\Gamma_K)$. By \cite[Proposition~6.3.5]{Doc}, there exists a $\mathbb{A}^{\dagger}_K$-lattice $N$ of $M$ such that $p^{-a}\varphi^h$ maps some basis of $N$ to another basis of $N$. Let $M'$ denote $M$ with $\varphi^h$-action given by $x\mapsto p^{-a}\varphi^h(x)$ and the same $\Gamma_K$-action as $M$. By the existence of the above lattice $N$, we have $M'\in \Mod^{\et}_{\bdagk}(\varphi^h,\Gamma_K)$. Since we have $\gk$-equivariant isomorphisms $\mathcal{V}(M)=(\bdag\otimes_{\bdagk}M)^{\varphi^h=p^a}\cong(\bdag\otimes_{\bdagk}M')^{\varphi^h=1}=\V(M')$, the assertion follows from the \'etale case (Theorem~\ref{thm:oc}).

Finally, we prove the assertion for $\D^{\dagger}_{\rig}$. By the base change equivalence (\cite[Theorem~6.3.3~(b)]{Doc})
\[
\alpha_1^*:\Mod^s_{\bdagk}(\varphi)\to \Mod^s_{\brigdagk}(\varphi),
\]
we also have the base change equivalence $\alpha_1^*:\Mod^s_{\B_{K}^{\dagger}}(\varphi,\Gamma_K)\to \Mod^s_{\brigdagk}(\varphi,\Gamma_K)$. Hence, the assertion follows from the case of $\D^{\dagger}$.
\item To check that the well-definedness of $\alpha_{\bullet}^*$'s, we have only to prove that pure objects are preserved by base change. For $\alpha_1$ and $\alpha_3$, it follows from \cite[Theorem~6.3.3~(b)]{Doc}. For $\alpha_2$, $\alpha_4$, it follows by definition: Precisely $M\in \Mod_{\bdagk}(\varphi)$ (resp. $\Mod_{\brigdagk}(\varphi)$) is pure if $\wtil{\B}^{\dagger}\otimes_{\bdagk}M$ (resp. $\btrigdag\otimes_{\brigdagk}M$) is pure by \cite[Definition~4.6.1 and 6.3.1]{Doc}. For $\alpha_5$, it follows from \cite[Proposition~4.5.10 and Definition~4.6.1]{Doc}.

The commutativity of the diagram is trivial. The compatibility follows by definition.
\end{enumerate}
\end{proof}

\subsection{Swan conductor for de Rham representations}\label{subsec:swan}
In this subsection, we will define Swan conductor of de Rham representations. In this subsection, Assumption~\ref{ass:relative} is not necessary since we do not use the results of \cite{AB}.

We first recall the canonical slope filtration associated to a Dieudonn\'e-Manin decomposition.
\begin{dfn}[{\cite[Remarque~3.3]{Col}}]\label{dfn:slopefil}
A $\varphi$-module $M$ over $\brig$ is a finite free $\brig$-module with semi-linear $\varphi$-action. A $\varphi$-module $M$ over $\brig$ admits a Dieudonn\'e-Manin decomposition if there exists an isomorphism $f:M\cong\oplus_{1\le i\le m}\brig\otimes_{\Q_p}D_{[s_i]}$ of $\varphi$-modules over $\brig$ with $s_1\le \dots\le s_m\in\Q$. We define the slope multiset of $M$ as the multiset of cardinality $\mathrm{rank}(M)$, consisting of $s_i$ included with multiplicity $\dim_{\Q_p}D_{[s_i]}$. Let $s'_1<\dots<s'_{r'}$ be the distinct elements in the slope multiset of $M$. Then, we define $\fil^0_f(M):=0$ and $\fil^i_f(M):=f^{-1}(\oplus_{j;s_j\le s'_i}\brig\otimes_{\Q_p}D_{[s_j]})$ for $1\le i\le r'$. Note that the filtration and the slope multiset are independent of the choice of $f$ above.
\end{dfn}

\begin{dfn}\label{dfn:swan}
Let $V\in \rep_{\dR}(\gk)$. First, we assume that the Hodge-Tate weights of $V$ are negative. By assumption, we have $\D_{\dR}(V)=(\B_{\dR}^+\otimes_{\Q_p}V)^{\gk}$. As in \cite[Proposition~5.3]{Ohk}, we define
\[
\nrig(V):=\{x\in\brig\otimes_{\Q_p}V;\iota_n(x)\in (\B_{\dR}^+\otimes_K \D_{\dR}(V))^{\nabla^{\geom}=0}\text{ for all }n\in\Z\},
\]
where $\iota_n:\brig\otimes_{\Q_p}V\to\B_{\dR}^+\otimes_{\Q_p}V;x\otimes v\mapsto \varphi^{-n}(x)\otimes v$. Since $\nrig(V)$ admits a Dieudonn\'e-Manin decomposition due to Colmez (\cite[Proposition~6.2]{Ohk}), $\nrig(V)$ is endowed with a canonical slope filtration $\fil^{\bullet}(\nrig(V))$ of $\varphi$-modules by Definition~\ref{dfn:slopefil}. Let $s_1<\dots<s_r$ be the distinct elements in the slope multiset of $\nrig(V)$. Write $s_i=a_i/h_i$ with $a_i\in\Z$, $h_i\in\N_{>0}$ relatively prime. By the uniqueness of slope filtrations, $\fil^i$ is $\gk$-stable and the graded piece $\gr^i(\nrig(V))$ belongs to $\Mod_{\brig}^{s_i}(\varphi,\gk)$. Hence, by Lemma~\ref{lem:pure}, there exists a unique $\mathcal{V}_i\in \rep_{a_i,h_i}(\gk)$, up to isomorphisms, such that $\gr^i(\nrig(V))\cong \brig\otimes_{\Q_{p^{h_i}}}\mathcal{V}_i$. It is proved in Step~1 of the proof of Main Theorem in \cite{Ohk} that the inertia $I_K$ acts on $\mathcal{V}_i$ via a finite quotient, i.e., $\mathcal{V}_i\in\rep_{\Q_{p^{h_i}}}^{\fg}(G_K)$ (in the reference, $\fil^i$ and $\mathcal{V}_i$ are denoted by $\mathcal{M}_i$ and $W_i$). Hence, we can define
\[
\sw(V):=\sum_i{\sw^{\AS}(\mathcal{V}_i)}.
\]
In the general Hodge-Tate weights case, we define $\nrig(V):=\nrig(V(-n))(n)$ and $\sw(V):=\sw(V(-n))$ for sufficiently large $n$. The definition is independent of the choice of $n$ since the above construction is compatible with Tate twist.
\end{dfn}

\begin{rem}
As in \cite{swan}, we should consider an appropriate contribution of ``monodromy action'' to define Artin conductor. To avoid complication, we do not define Artin conductor for de Rham representations in this paper.
\end{rem}

The lemma below easily follows from Hilbert~90.

\begin{lem}\label{lem:sw}
Let $V\in \rep_{\dR}(\gk)$.
\begin{enumerate}
\item If $L$ is the $p$-adic completion of an unramified extension of $K$, then we have $\sw(V|_L)=\sw(V)$.
\item Assume $V\in\rep_{\Q_p}^f(\gk)$. Then, we have $\sw(V)=\sw^{\AS}(V)$.
\end{enumerate}
\end{lem}

Though the following result will not be used in the proof of Main Theorem, we remark that when $k_K$ is perfect, our definition is compatible with the classical definition.

\begin{lem}[Compatibility of usual Swan conductor in the perfect residue field case]\label{lem:swpf}
Assume that $k_K$ is perfect. Then, we have $\sw(V)=\sw(\D_{\pst}(V))$ (see \cite[0.4]{swan} for the definition of $\D_{\pst}$).
\end{lem}
\begin{proof}
Let notation be as in Definition~\ref{dfn:swan}. By Tate twist, we may assume that all Hodge-Tate of $V$ are negative. By $\sw(\D_{\pst}(V))=\sw(\D_{\pst}(V|_{K^{\ur}}))$ and Lemma~\ref{lem:sw}~(i), we may assume that $k_K$ is algebraically closed by replacing $K$ by $K^{\ur}$. Since $\B_{\dR}^+\otimes_K\D_{\dR}(V)$ is a lattice of $\B_{\dR}^+\otimes_{\Q_p}V$, we may identify $\wtil{\N}_{\rig}^{\nabla+}(V)[t^{-1}]$ with $\wtil{\B}_{\rig}^{\nabla+}\otimes_{\Q_p}V[t^{-1}]$. By the $p$-adic monodromy theorem, there exists a finite Galois extension $L/K$ such that $\D_{\st,L}(V):=(\B_{\st}\otimes_{\Q_p}V)^{G_L}$ has dimension $\dim_{\Q_p}V$. Moreover, we may assume that $G_L$ acts on each $\mathcal{V}_i$ trivially. Put $D_i:=(\B_{\st}\otimes_{\wtil{\B}_{\rig}^{\nabla+}}\fil^i(\wtil{\N}_{\rig}^{\nabla+}(V)))^{G_L}$, which forms an increasing filtration of $\D_{\st,L}(V)$. Then, we have canonical morphisms
\[
D_i/D_{i+1}\hookrightarrow (\B_{\st}\otimes_{\wtil{\B}_{\rig}^{\nabla+}}\gr^i(\wtil{\N}_{\rig}^{\nabla+}(V)))^{G_L}\cong (\B_{\st}\otimes_{\Q_{p^{h_i}}}\mathcal{V}_i)^{G_L}\cong W(k_L)[p^{-1}]\otimes_{\Q_{p^{h_i}}}\mathcal{V}_i,
\]
where the first injection is an isomorphism by counting dimensions. By the additivity of Swan conductor, we have $\sw(\D_{\pst}(V))=\sw(\D_{\st,L}(V))=\sum_i\sw(D_i/D_{i+1})=\sum_i\sw(\mathcal{V}_i)=\sw(V)$.
\end{proof}

\subsection{Main Theorem}\label{subsec:main}
The aim of this subsection is to prove the following theorem, which generalizes Marmora's formula in Remark~\ref{rem:Mar}:

\begin{thm}[Main Theorem]\label{thm:main}
Let $V$ be a de Rham representation of $\gk$. Then, the sequence $\{\sw(V|_{{K_n}})\}_{n>0}$ is eventually stationary and we have
\[
\sw^{\nabla}(\N_{\dR}(V))=\lim_n{\sw(V|_{K_n})}.
\]
\end{thm}

\begin{rem}\label{rem:Mar}
When $k_K$ is perfect, we explain that our formula coincides with the following Marmora's formula (\cite[Th\'eor\`eme~1.1]{Mar}):
\[
\mathrm{Irr}(\N_{\dR}(V))=\lim_{n\to\infty}\sw(\D_{\mathrm{pst}}(V|_{K_n})).
\]
Here, the LHS means the irregularity of $\N_{\dR}(V)$ regarded as a $p$-adic differential equation. By Lemma~\ref{lem:swpf}, the RHS is equal to the RHS in Main Theorem~\ref{thm:main}. Therefore, we have only to prove $\mathrm{Irr}(D)=\sw^{\nabla}(D)$ for a $(\varphi,\nabla)$-module $D$ over the Robba ring. Since $D$ is endowed with a slope filtration and both irregularity and the differential Swan conductor are additive, we may assume that $D$ is \'etale by d\'evissage. Let $V$ be the corresponding $p$-adic representation of finite local monodromy. Then, the differential Swan conductor $\sw^{\nabla}(D)$ coincides with the usual Swan conductor of $V$ (\cite[Proposition~3.5.5]{Sw}). On the other hand, $\mathrm{Irr}(D)$ coincides with the usual Swan conductor of $V$ (\cite[Theorem~7.2.2]{Tsu}), which implies the assertion.
\end{rem}

We will deduce Main Theorem~\ref{thm:main} from Lemma~\ref{lem:normrep}~(ii) by d\'evissage. In the following, we use the notation as in Definition~\ref{dfn:swan}.
\begin{lem}\label{lem:diffgr}
Let $V$ be a de Rham representation of $\gk$ with Hodge-Tate weights $\le 0$.
\begin{enumerate}
\item The $(\varphi,\gk)$-modules
\[
\wtil{\B}_{\rig}^{\dagger}\otimes_{\B^{\dagger}_{\rig,K}}\ndr(V),\ \wtil{\B}_{\rig}^{\dagger}\otimes_{\wtil{\B}^{\nabla+}_{\rig}}\nrig(V)
\]
coincide with each other in $\wtil{\B}_{\rig}^{\dagger}\otimes_{\Q_p}V$. Moreover, the two filtrations induced by the slope filtrations of $\N_{\dR}(V)$ and $\wtil{\N}_{\rig}^{\nabla+}(V)$ also coincide with each other.
\item Let notation be as in Construction~\ref{const:dswan}. Then, there exists a canonical isomorphism
\[
\gr^i(\N_{\dR}(V))\cong D^{\dagger}_{\rig}(\mathcal{V}_i|_{\E_K})
\]
as $(\varphi,\nabla)$-modules over $\B^{\dagger}_{\rig,K}$.
\end{enumerate}
\end{lem}
\begin{proof}
\begin{enumerate}
\item We prove the first assertion. By Lemma~\ref{lem:unique} (with $B=\wtil{\B}^{\dagger,r}_{\rig}$), we have only to prove that $D^{(1)}:=\wtil{\B}_{\rig}^{\dagger,r}\otimes_{\B^{\dagger,r}_{\rig,K}}\N_{\dR,r}(V)$, $D^{(2)}:=\wtil{\B}_{\rig}^{\dagger,r}\otimes_{\wtil{\B}^{\nabla+}_{\rig}}\nrig(V)$ and $D:=\wtil{\B}_{\rig}^{\dagger,r}\otimes_{\Q_p}V$ satisfy the conditions in the lemma. We have $\N_{\dR,r}(V)[t^{-1}]=\D^{\dagger,r}_{\rig}(V)[t^{-1}]$ by definition and
\[
\wtil{\B}^{\dagger,r}_{\rig}\otimes_{\B^{\dagger,r}_{\rig,K}}\D^{\dagger,r}_{\rig}(V)\cong\wtil{\B}^{\dagger,r}_{\rig}\otimes_{\B^{\dagger,r}}\B^{\dagger,r}\otimes_{\B^{\dagger,r}}\D^{\dagger,r}(V)\cong\wtil{\B}^{\dagger,r}_{\rig}\otimes_{\B^{\dagger,r}}\B^{\dagger,r}\otimes_{\Q_p}V\cong\wtil{\B}^{\dagger,r}_{\rig}\otimes_{\Q_p}V.
\]
Since we have $\wtil{\N}_{\rig}^{\nabla+}(V)[t^{-1}]=\wtil{\B}_{\rig}^{\nabla+}[t^{-1}]\otimes_{\Q_p}V$ by definition, we have a canonical isomorphism $\wtil{\B}_{\rig}^{\dagger,r}\otimes_{\wtil{\B}^{\nabla+}_{\rig}}\nrig(V)[t^{-1}]\cong\wtil{\B}^{\dagger,r}_{\rig}[t^{-1}]\otimes_{\Q_p}V$, which implies the condition (i). By Proposition~\ref{prop:const}~(ii), we have a canonical isomorphism $\B_{\dR}^+\otimes_{\iota_n,\B^{\dagger,r}_{\rig,K}}\N_{\dR,r}(V)\cong\B_{\dR}^+\otimes_K\D_{\dR}(V)$. On the other hand, there exist canonical isomorphisms
\[
\B_{\dR}^+\otimes_{\wtil{\B}_{\rig}^{\nabla+}}\wtil{\N}^{\nabla+}_{\rig}(V)\cong\B_{\dR}^+\otimes_{\B_{\dR}^{\nabla+}}(\B_{\dR}^+\otimes_K\D_{\dR}(V))^{\nabla^{\geom}=0}\cong\B_{\dR}^+\otimes_K\D_{\dR}(V),
\]
where the first isomorphism follows from \cite[Proposition~5.3~(ii)]{Ohk} and the second isomorphism follows from \cite[Proposition~5.4]{Ohk}. Since the canonical map $\B_{\dR}^{\nabla+}\to\B_{\dR}^+$ is faithfully flat, the condition~(ii) is verified. The second assertion follows from the uniqueness of the slope filtration (\cite[Theorem~6.4.1]{Doc}).
\item By (i), there exists canonical isomorphisms
\[
\wtil{\B}^{\dagger}_{\rig}\otimes_{\B^{\dagger}_{\rig,K}}\gr^i(\N_{\dR}(V))\cong\wtil{\B}^{\dagger}_{\rig}\otimes_{\wtil{\B}^{\nabla+}_{\rig}}\gr^i(\wtil{\N}_{\rig}^{\nabla+}(V))\cong\wtil{\B}^{\dagger}_{\rig}\otimes_{\Q_{p^{h_i}}}\mathcal{V}_i
\]
as $(\varphi,\gk)$-modules. By Lemma~\ref{lem:pure}, we obtain a canonical isomorphism $\gr^i(\N_{\dR}(V))\cong\D^{\dagger}_{\rig}(\mathcal{V}_i)$ as $(\varphi,\Gamma_K)$-modules. Since $\mathcal{V}_i$ is of finite local monodromy, so is $\mathcal{V}_i|_{\E_K}$. Hence, $\dim_{\B^{\dagger}_K}D^{\dagger}(\mathcal{V}_i|_{\E_K})=\dim_{\Q_{p^h}}\mathcal{V}_i$, in particular, a canonical injection $D^{\dagger}(\mathcal{V}_i|_{\E_K})\hookrightarrow (\B^{\dagger}\otimes_{\Q_{p^{h_i}}}\mathcal{V}_i)^{H_K}$ is an isomorphism. Therefore, we have canonical isomorphisms $D^{\dagger}_{\rig}(\mathcal{V}_i|_{\E_K})\cong\D^{\dagger}_{\rig}(\mathcal{V}_i)\cong\gr^i(\N_{\dR}(V))$ as (pure) $\varphi$-modules over $\B^{\dagger}_{\rig,K}$, hence, the assertion follows from Lemma~\ref{lem:diffpure}.
\end{enumerate}
\end{proof}

\begin{rem}
One can prove that there exist canonical isomorphisms
\[
\wtil{\B}^{\dagger}_{\rig}\otimes_{\B^{\dagger}_{\rig,K}}\N_{\dR}(V)\cong\wtil{\B}^{\dagger}_{\rig}\otimes_{\wtil{\B}^{\nabla+}_{\rig}}\wtil{\N}^{\nabla+}_{\rig}(V)\cong\N^{\dagger}_{\rig}(V).
\]
\end{rem}

\begin{lem}\label{lem:swangr}
We have
\[
\sw^{\nabla}(\ndr(V))=\sum_{1\le i\le r}\sw^{\AS}(\mathcal{V}_i|_{\E_K}).
\]
\end{lem}
\begin{proof}
We have
\[
\sw^{\nabla}(\ndr(V))=\sum_{1\le i\le r}{\sw^{\nabla}(\gr^i{(\ndr(V))})}=\sum_{1\le i\le r}\sw^{\nabla}(D^{\dagger}_{\rig}(\mathcal{V}_i|_{\E_K}))=\sum_{1\le i\le r}\sw^{\AS}(\mathcal{V}_i|_{\E_K}),
\]
where the first equality follows from the additivity of the differential Swan conductor (Lemma~\ref{lem:diffadd}), the second one follows from Lemma~\ref{lem:diffgr}~(ii), and the third one follows from Xiao's comparison theorem (Theorem~\ref{thm:Xiao}).
\end{proof}

\begin{proof}[Proof of Main Theorem~\ref{thm:main}]
By Lemma~\ref{lem:swangr} and the definition of Swan conductor (Definition~\ref{dfn:swan}), we have only to prove $\sw^{\AS}(\mathcal{V}_i|_{\E_K})=\sw^{\AS}(\mathcal{V}_i|_{K_n})$ for all sufficiently large $n$, which follows from Lemma~\ref{lem:normrep}~(ii).
\end{proof}

\section{Appendix: List of notation}
The following is a list of notation in order defined.

\begin{itemize}
\item[\S~\ref{subsec:der}]: $\hat{\Omega}^1_K$, $\partial_j$, $\partial/\partial t_j$.
\item[\S~\ref{subsec:Galois}]: $\wtil{K}_n$, $\wtil{K}_{\infty}$, $\Gamma_{\wtil{K}}$, $H_{\wtil{K}}$, $\gamma_a$, $\gamma_b$, $\eta=(\eta_0,\dots,\eta_d)$, $\mathfrak{g}$, $L_n$, $L_{\infty}$, $\Gamma_{L}$, $H_{L}$, $\Gamma_{L,j}$.
\item[\S~\ref{subsec:periods}]: $\wtil{\E}^{(+)}$, $v_{\wtil{\E}}$, $\wtil{\A}^{(+)}$, $\wtil{\B}^{(+)}$, $\varepsilon$, $\wtil{t}_j$, $\pi$, $q$, $\A_{\inf}$, $\B_{\dR}^{(+)}$, $u_j$, $t$, $\D_{\dR}(\,\cdot\,)$, $\nabla^{\geom}$, $\B_{\dR}^{\nabla(+)}$, $\A_{\cris}$, $\B_{\cris}$, $\wtil{\B}^{\nabla+}_{\rig}$.
\item[\S~\ref{subsec:AS}]: $as_{L/K,Z}^a$, $\mathcal{F}^a(L)$, $b(L/K)$, $as_{L/K,Z,P}^a$, $\mathcal{F}_{\log}^a(L)$, $b_{\log}(L/K)$, $\art^{\AS}(\,\cdot\,)$, $\sw^{\AS}(\,\cdot\,)$.
\item[\S~\ref{subsec:ocnotation}]: $v^{\le n}$, $w_r$, $W(E)_r$, $W_{\con}(E)$, $\Gamma_r$, $\Gamma_{\con}$, $\Gamma_{\an,r}$, $\Gamma_{\an,\con}$, $\oo\{\{S\}\}$, $\oo((S))^{\dagger,r}$, $\oo((S))^{\dagger}$, $\mathcal{R}$, $\Mod_{\bullet}(\sigma)$, $\Mod_{\bullet}^{\et}(\sigma)$, $\Mod_{\bullet}^s(\sigma)$.
\item[\S~\ref{subsec:diffconductor}]: $\Omega^1_{R}$, $\Omega^1_{\mathcal{R}}$, $d:\mathcal{R}\to\Omega^1_{\mathcal{R}}$, $\Mod^s_{\bullet}(\varphi^h,\nabla)$ $D$, $D^{\dagger}$, $\sw^{\nabla}(\,\cdot\,)$.
\item[\S~\ref{subsec:scholl}]: $X^{(+)}_{\mathfrak{K}}=X^{(+)}(\mathfrak{K},\xi,n_0)$.
\item[\S~\ref{subsec:phigamma}]: $\E_L^{(+)}$, $\wtil{\E}^{(+)}_K$, $\wtil{\A}^{(+)}_L$, $\wtil{\B}_L$, $\A$, $\B_L$, $\B$, $\Mod^{\et}_{\B_L}(\varphi^h,\Gamma_L)$, $\D(\,\cdot\,)$, $\V(\,\cdot\,)$.
\item[\S~\ref{subsec:oc}]: $\wtil{\A}^{\dagger,r}$, $\wtil{\A}^{\dagger}$, $\wtil{\B}^{\dagger,r}$, $\wtil{\B}^{\dagger}$, $\wtil{\B}^{\dagger,r}_{\rig}$, $\wtil{\B}^{\dagger}_{\rig}$, $\A^{\dagger,r}$, $\A^{\dagger}$, $\B^{\dagger,r}$, $\B^{\dagger}$, $\B^{\dagger,r}_{\rig}$, $\B^{\dagger}_{\rig}$, $\wtil{\A}^{\dagger,r}_L$, $\wtil{\A}^{\dagger}_L$, $\wtil{\B}^{\dagger,r}_L$, $\wtil{\B}^{\dagger}_L$, $\wtil{\B}^{\dagger,r}_{\rig,L}$, $\wtil{\B}^{\dagger}_{\rig,L}$, $\A^{\dagger,r}_L$, $\A^{\dagger}_L$, $\B^{\dagger,r}_L$, $\B^{\dagger}_L$, $\B^{\dagger,r}_{\rig,L}$, $\B^{\dagger}_{\rig,L}$, $\D^{\dagger,r}(\,\cdot\,)$, $\D^{\dagger}(\,\cdot\,)$, $\D^{\dagger,r}_{\rig}(\,\cdot\,)$, $\D^{\dagger}_{\rig}(\,\cdot\,)$.
\item[\S~\ref{subsec:convergent}]: $R\langle\und{X}\rangle$, $\oo((S))^{\dagger,r}_0$, $|\cdot|_r$, $\oo[[S]]\langle\und{X}\rangle$, $\oo((S))_0^{\dagger,r}\langle\und{X}\rangle$, $\oo((S))^{\dagger,r}\langle\und{X}\rangle$, $\deg(\mathfrak{p})$, $\kappa(\mathfrak{p})$, $\kappa(p)$, $\pi_{\mathfrak{p}}$.
\item[\S~\ref{subsec:grobnerregular}]: $\succeq$, $\succ$, $\succeq_{\lex}$, $\und{v}_R$, $\und{\deg}_R$, $\LT_R(\,\cdot\,)$, $|\cdot|_{\qt}$.
\item[\S~\ref{subsec:grobnerannulus}]: $A$, $I^{\dagger,r}$, $A^{\dagger,r}$, $|\cdot|_{r,\qt}$.
\item[\S~\ref{subsec:continuity}]: $\Idem(\,\cdot\,)$, $|\cdot|_{\mathfrak{p},\qt}$, $|\cdot|_{\mathfrak{p},\spe}$ $A_{\kappa(\mathfrak{p})}$.
\item[\S~\ref{sec:ram}]: $AS^r$, $AS^r_{\log}$.
\item[\S~\ref{subsec:horizontal}]: $K_n[[u_1,\dots,u_d]]^{\nabla}$.
\item[\S~\ref{subsec:const}]: $\iota_n$, $t_{n,w}$, $\N_{\dR,r}(\,\cdot\,)$, $\N_{\dR}(\,\cdot\,)$, $\wtil{\N}_{\rig}^{\dagger,r}(\,\cdot\,)$, $\wtil{\N}^{\dagger}_{\rig}(\,\cdot\,)$.
\item[\S~\ref{subsec:Lie}]: $\nabla_j$.
\item[\S~\ref{subsec:diffaction}]: $\wtil{\nabla}_j$, $\nabla'_j$.
\item[\S~\ref{subsec:pure}]: $\rep_{a,h}(G_K)$, $D_{[s]}$, $\Mod^s_{\wtil{\B}^{\dagger}_{\rig}}(\varphi,G_K)$, $\Mod^s_{\B^{\dagger}_{\rig,K}}(\varphi,\Gamma_K)$, $\Mod^s_{\wtil{\B}^{\dagger}}(\varphi,G_K)$, $\Mod^s_{\B^{\dagger}_{K}}(\varphi,\Gamma_K)$, $\Mod^s_{\wtil{\B}^{\nabla+}_{\rig}}(\varphi,G_K)$, $\wtil{\D}^{\nabla+}_{\rig}(\,\cdot\,)$, $\wtil{\D}^{+}_{\rig}(\,\cdot\,)$, $\D^{\dagger}_{\rig}(\,\cdot\,)$, $\wtil{\D}^{\dagger}(\,\cdot\,)$, $\D^{\dagger}(\,\cdot\,)$.
\item[\S~\ref{subsec:swan}]: $\wtil{\N}^{\nabla+}_{\rig}(\,\cdot\,)$, $\mathcal{V}_i$, $\sw(\,\cdot\,)$.
\end{itemize}

\subsection*{Acknowledgement}
The author would like to thank Atsushi Shiho and Takeshi Tsuji for encouragement. The author thank Liang Xiao for the proof of Lemma~\ref{lem:Xiao}. The author also thanks Adriano Marmora for helpful discussion. The author thanks the referee for helpful advices. The author is supported by Research Fellowships of Japan Society for the Promotion of Science for Young Scientists.




\begin{thebibliography}{}
\bibitem[And06]{And}
F.~Andreatta, Generalized ring of norms and generalized $(\varphi,\Gamma)$-modules, Ann. Sci. Ecole Norm. Sup. (4) 39 (2006), no. 4, 599-647. 
\bibitem[AB08]{AB}
F.~Andreatta-O.~Brinon, Surconvergence des repr\'esentations $p$-adiques; Le cas relatif, Ast\'erisque No. 319 (2008), 39-116.
\bibitem[AB10]{AB2}
F.~Andreatta-O.~Brinon, $B_{\mathrm{dR}}$-repr\'esentations dans le cas relatif, Ann. Sci. Ec. Norm. Super. (4) 43 (2010), no. 2, 279-339. 
\bibitem[AS02]{AS}
A.~Abbes-T.~Saito, Ramification of local fields with imperfect residue fields, Amer. J. Math. 124 (2002), no. 5, 879-920.
\bibitem[AS03]{AS2}
A.~Abbes-T.~Saito, Ramification of local fields with imperfect residue fields. II, Doc. Math. 2003, Extra Vol., 5-72.
\bibitem[Ber02]{Inv}
L.~Berger, Repr\'esentations $p$-adiques et \'equations diff\'erentielles, Invent. Math. 148 (2002), no. 2, 219-284. 
\bibitem[Ber08a]{Ber}
L.~Berger, \'Equations diff\'erentielles $p$-adiques et $(\varphi,N)$-modules filtr\'es, Ast\'erisque No. 319 (2008), 13-38. 
\bibitem[Ber08b]{Bpair}
L.~Berger, Construction de $(\varphi,\Gamma)$-modules: repr\'esentations $p$-adiques et $B$-paires, Algebra Number Theory 2 (2008), no. 1, 91-120.
\bibitem[BGR84]{BGR}
S.~Bosch, U.~G\"untzer, R.~Remmert,
Non-Archimedean analysis. A systematic approach to rigid analytic geometry. Grundlehren der Mathematischen Wissenschaften, 261. Springer-Verlag, Berlin, 1984. xii+436 pp.
\bibitem[Bou98]{Bou}
N.~Bourbaki, Commutative algebra. Chapters 1-7. Elements of Mathematics (Berlin). Springer-Verlag, Berlin, 1998. xxiv+625 pp.
\bibitem[Bou06]{Bou9}
N.~Bourbaki, \'El\'ements de math\'ematique. Alg\`ebre commutative. Chapitres 8 et 9. Reprint of the 1983 original. Springer, Berlin, 2006. ii+200 pp.
\bibitem[CC98]{CC1}
F. Cherbonnier-P.~Colmez, Repr\'esentations $p$-adiques surconvergentes, Invent. Math.~133 (1998), 581--611.
\bibitem[Col08a]{Col}
P.~Colmez, Espaces vectoriels de dimension finie et repr\'esentations de de Rham.  Ast\'erisque No. 319 (2008), 117-186.
\bibitem[Col08b]{swan}
P.~Colmez, Conducteur d'Artin d'une repr\'esentation de de Rham, Ast\'erisque No. 319 (2008), 187-212.
\bibitem[CLO97]{CLO}
D.~Cox-J.~Little-D.~O'Shea, Ideals, varieties, and algorithms. An introduction to computational algebraic geometry and commutative algebra. Second edition. Undergraduate Texts in Mathematics. Springer-Verlag, New York, 1997. xiv+536 pp. 
\bibitem[Eis95]{Eis}
D.~Eisenbud, Commutative algebra. With a view toward algebraic geometry. Graduate Texts in Mathematics, 150. Springer-Verlag, New York, 1995. xvi+785 pp.
\bibitem[Fon90]{Fes}
J.-M.~Fontaine, Repr\'esentations $p$-adiques des corps locaux. I. The Grothendieck Festschrift, Vol. II, 249-309, Progr. Math., 87, Birkh\"auser Boston, Boston, MA, 1990.
\bibitem[Fon94]{Fon}
J.-M.~Fontaine, Repr\'esentations $p$-adiques semi-stables. Ast\'erisque No. 223 (1994), 113-184.
\bibitem[Hat14]{Hat}
S.~Hattori, Ramification theory and perfectoid spaces, arXiv:1304.5895, Compositio Mathematica 150 (2014), 798-834.
\bibitem[Hel43]{Hel}
O.~Helmer, The elementary divisor theorem for certain rings without chain condition. Bull. Amer. Math. Soc. 49, (1943), 225-236.
\bibitem[Hyo86]{Hyo}
O.~Hyodo, On the Hodge-Tate decomposition in the imperfect residue field case.
J. Reine Angew. Math. 365 (1986), 97-113.
\bibitem[Kat04]{Kat}
K.~Kato, $p$-adic Hodge theory and values of zeta functions of modular forms, Ast\'erisque~295 (2004), 117--290.
\bibitem[Ked04]{mon}
K.~Kedlaya, A $p$-adic local monodromy theorem. Ann. of Math. (2) 160 (2004), no. 1, 93-184.
\bibitem[Ked05b]{Doc}
K.~Kedlaya, Slope filtrations revisited. Doc. Math. 10 (2005), 447-525.
\bibitem[Ked07]{Sw}
K.~Kedlaya, Swan conductors for $p$-adic differential modules. I. A local construction. Algebra Number Theory 1 (2007), no. 3, 269-300. 
\bibitem[Ked10]{pde}
K.~Kedlaya, $p$-adic differential equations, Cambridge Studies in Advanced Mathematics, 125. Cambridge University Press, Cambridge, 2010. xviii+380 pp.
\bibitem[Ked]{slope}
K.~Kedlaya, Some slope theory for multivariate Robba rings, arXiv:1311.7468.
\bibitem[Laz62]{Laz}
M.~Lazard, Les z\'eros des fonctions analytiques d'une variable sur un corps valu\'e complet. Inst. Hautes \'Etudes Sci. Publ. Math. No. 14 1962 47-75. 
\bibitem[Mar04]{Mar}
A.~Marmora, Irr\'egularit\'e et conducteur de Swan $p$-adiques. Doc. Math. 9 (2004), 413-433.
\bibitem[Matsuda95]{Matd}
S.~Matsuda, Local indices of $p$-adic differential operators corresponding to Artin-Schreier-Witt coverings. Duke Math. J. 77 (1995), no. 3, 607-625.
\bibitem[Matsumura80]{Mat}
H.~Matsumura, Commutative algebra. Second edition. Mathematics Lecture Note Series, 56. Benjamin/Cummings Publishing Co., Inc., Reading, Mass., 1980. xv+313 pp.
\bibitem[Ohk10]{JNT}
S.~Ohkubo, Galois theory of $B^+_{\dR}$ in the imperfect residue field case. J. Number Theory 130 (2010), no. 7, 1609-1641.
\bibitem[Ohk13]{Ohk}
S.~Ohkubo, The $p$-adic monodromy theorem in the imperfect residue field case,  Algebra and Number Theory~7 (2013), No. 8, 1977--2037.
\bibitem[Ray70]{Ray}
M.~Raynaud, Anneaux locaux hens\'eliens. Lecture Notes in Mathematics, Vol. 169 Springer-Verlag, Berlin-New York 1970 v+129 pp. 
\bibitem[Scholl06]{Sch}
A.~Scholl, Higher fields of norms and $(\varphi,\Gamma)$-modules. Doc. Math. 2006, Extra Vol., 685--709.
\bibitem[Scholze12]{Scholze}
P. Scholze, Perfectoid spaces, Inst. Hautes \'Etudes Sci. Publ. Math. No. 116 (2012) 245--313.
\bibitem[Schn02]{Schn}
P.~Schneider, Nonarchimedean functional analysis. Springer Monographs in Mathematics. Springer-Verlag, Berlin, 2002. vi+156 pp.
\bibitem[Tsu98]{Tsu}
N.~Tsuzuki, The local index and the Swan conductor. Compositio Math. 111 (1998), no. 3, 245--288.
\bibitem[Xia10]{Xia}
L.~Xiao, On ramification filtrations and $p$-adic differential modules, I: the equal characteristic case, Algebra Number Theory 4 (2010), no. 8, 969--1027. 
\bibitem[Xia12]{Xia2}
L.~Xiao, On ramification filtrations and $p$-adic differential modules, II: mixed characteristic case, Compositio Mathematica 148 (2012), no. 2, 415--463.
\end{thebibliography}
\end{document}